\documentclass[12pt,letterpaper,reqno]{amsart}
\usepackage{mathrsfs}
\usepackage{amsmath,amsfonts,amssymb,amsthm, color}
\usepackage[T1]{fontenc}
\usepackage{fouridx}
\usepackage[normalem]{ulem}

\advance\textwidth by 2cm 
\advance\oddsidemargin by -1cm
\advance\evensidemargin by -1cm

\newcommand\delf[1]{} 
\newcommand\addf[1]{{\color{blue} #1}} 
\newcommand\deln[1]{}

\newcommand\delg[1]{} 
\newcommand\delh[1]{} 

\advance\textwidth by 2cm 
\advance\oddsidemargin by -1cm
\advance\evensidemargin by -1cm

\newcommand\del[1]{}
\newcommand\dela[1]{}

\newcommand\Greendel[1]{}
\newcommand\old[1]{}

\numberwithin{equation}{section}
\newcommand{\embed}{\hookrightarrow}

\newcommand{\lb}{\langle}
\newcommand{\rb}{\rangle}

\def\E{{\mathbb E}\,}
\def\ds{\displaystyle}
\def\vs{\vspace{.1cm}}

\newcommand{\rA}{\mathrm{A}}
\newcommand{\rB}{\mathrm{B}}

\newcommand{\rH}{\mathrm{ H}}

\newcommand{\rV}{\mathrm{ V}}
\newcommand{\cH}{\mathcal{ H}}
\newcommand{\cX}{\mathcal{ X}}

\newcommand{\cHX}{\mathcal{HX}}
\newcommand\todown{\searrow}

\newcommand{\eps}{\varepsilon}

\renewcommand{\H}{{H_\infty}}

\renewcommand{\a}{\alpha}

\renewcommand{\d}{\delta}

\def\d{{\,\rm d}}
\def\old#1{}

\def\text#1{{\rm #1}}
\def\newold#1{}

\def\divv{{\rm div}\,}

\theoremstyle{plain}
\newtheorem{theorem}{Theorem}[section]
\theoremstyle{remark}
\newtheorem{remark}[theorem]{Remark}

\theoremstyle{plain}
\newtheorem{corollary}[theorem]{Corollary}
\newtheorem{lemma}[theorem]{Lemma}
\newtheorem{proposition}[theorem]{Proposition}
\newtheorem{definition}[theorem]{Definition}

\newtheorem{assumption}[theorem]{Assumption}

\newtheorem{rem}{Remark}

\numberwithin{equation}{section}

\begin{document}

\baselineskip 12pt

\title[QP 2D SNSES st wn]
{Quasipotential and exit time for 2D Stochastic Navier-Stokes equations driven by   space time white noise}
\author{Z. Brze{\'z}niak}
\address{Department of Mathematics\\
The University of York\\
Heslington, York YO10 5DD, UK} \email{zdzislaw.brzezniak@york.ac.uk}
\author{S. Cerrai and M. Freidlin}
\address{Department of Mathematics\\
University of Maryland\\
College Park, MD, 20742,  USA} \email{cerrai@math.umd.edu, mif@math.umd.edu}
\date{\today}




\begin{abstract} We are dealing with the Navier-Stokes equation in a bounded regular domain $\mathcal{O}$ of $\mathbb{R}^2$, perturbed by an additive Gaussian noise $\partial w^{Q_\delta}/\partial t$, which is white in time and colored in space. We assume that the correlation radius of the noise gets smaller and smaller as $\delta\searrow 0$, so that the noise converges to the white noise in space and time. For every $\delta>0$ we introduce the large deviation action functional $S^\delta_{T}$ and the corresponding quasi-potential $U_\delta$ and, by using arguments from relaxation and  $\Gamma$-convergence we show that $U_\delta$ converges to $U=U_0$, in spite of the fact that the Navier-Stokes equation has no meaning in the space of square integrable functions, when perturbed by space-time white noise. Moreover, in the case of periodic boundary conditions the limiting functional $U$ is explicitly computed.

Finally, we apply these results to estimate of the asymptotics of the  expected  exit time of the solution of the stochastic Navier-Stokes equation from a basin of attraction of an asymptotically stable point for the unperturbed system.

\end{abstract}

\maketitle \tableofcontents
\section{Introduction}\addf{\label{sec-intro}}

Let $\mathcal{O}$ be a regular bounded open domain of $\mathbb{R}^2$. We consider here the $2$-dimensional  Navier-Stokes equation in $\mathcal{O}$, perturbed by a small Gaussian noise
\[\frac{\partial u(t,x)}{\partial t}=\Delta u(t,x)-\left(u(t,x)\cdot \nabla u(t,x)\right)u(t,x)-\nabla p(t,x)+\sqrt{\eps}\,\eta(t,x),\]
with the incompressibility condition
\[\text{div}\, u(t,x)=0\]
and initial and boundary conditions
\[u(t,x)=0,\ \ \ \ x \in\,\partial \mathcal{O},\ \ \ \ u(0,x)=u_0(x).\]
Here $0<\eps<<1$ and $\eta(t,x)$ is a Gaussian random field, white in time and colored in space.

\medskip

In what follows, for any $\a \in\,\mathbb{R}$ we shall denote by $\rV_\a$ the closure in the space $[H^\a(\mathcal{O})]^2$ of the set of infinitely differentiable $2$-dimensional vector fields, having   zero divergence and compact support on $\mathcal{O}$,  and we shall set $\rH=\rV_0$ and $\rV=\rV_1$. We will also set
\[D(\rA)=[H^2(\mathcal{O})]^2\cap \rV,\ \ \ \ \rA x=-\Delta x,\ \ x \in\,D(\rA).\]
The operator $\rA$ is positive and  self-adjoint, with compact resolvent. We will denote with  $0<\lambda_1\leq \lambda_2\leq \cdots$ and $\{e_k\}_{k \in\,\mathbb{N}}$  the eigenvalues and the eigenfunctions of $\rA$, respectively. Moreover, we will define the bilinear operator $B:\rV\times \rV\to \rV_{-1}$ by setting
\[\left<B(u,v),z\right>=\int_{\mathcal{O}}z(x)\cdot \left[\left(u(x)\cdot \nabla \right)v(x)\right]\,dx.\]
With these notations, if we apply to each term of the Navier-Stokes equation above the projection operator into the space of divergence free fields, we formally arrive to the abstract equation
\begin{equation}
\label{eq1}
du(t)+\rA u(t)+B(u(t),u(t))=\sqrt{\eps}\,dw^Q(t),\ \ \ u(0)=u_0,
\end{equation}
where the noise $w^Q(t)$ is assumed to be of the following form
\begin{equation}
\label{noise}
w^Q(t)=\sum_{k=1}^\infty Q e_k \beta_k(t),\ \ \ \ t\geq 0,
\end{equation}
for some sequence of independent standard Brownian motions $\{\beta_k(t)\}_{k \in\,\mathbb{N}}$ and a  linear operator $Q$ defined on $\rH$ (for all details see Section \ref{sec-prel}).

As well known,  white noise in space and time (that is $Q=I$) cannot be taken into consideration in order to study equation  \eqref{eq1} in the space $\rH$. But if we assume that $Q$ is a compact operator satisfying suitable conditions, as for example $Q\sim \rA^{-\a}$, for some $\a>0$, we have that for any $u_0 \in\,\rH$ and $T>0$ equation \eqref{eq1} is well defined in  $C([0,T];\rH)$ and the validity of a large deviation principle and the problem of the exit of the solution of equation \eqref{eq1} from a domain can be studied.

As in the previous work \cite{cerrai-freidlin}, where a class of reaction-diffusion equations in any space dimension perturbed by multiplicative noise has been considered, in the present  paper we want to see how we can describe the small noise asymptotics of equation \eqref{eq1}, as if the noisy perturbation were given by a white noise in space and time. This means that, in spite of the fact that equation \eqref{eq1} is not meaningful in $\rH$ when the noise is white in space, the relevant quantities for the large deviations and the exit problems associated with it can be approximated by the analogous quantities that one would get in the case of white noise in space. In particular, when periodic boundary conditions are imposed, such quantities can be explicitly computed and such approximation becomes particularly useful.

\medskip

In what follows we shall consider a family of positive linear operators $\{Q_\delta\}_{\delta \in\,(0,1]}$ defined on $\rH$, such that for any fixed $\delta \in\,(0,1]$ equation \eqref{eq1}, with noise
\[w^{Q_\delta}(t)=\sum_{k=1}^\infty Q_\delta e_k \beta_k(t),\ \ \ \ t\geq 0,\]
is well defined in $C([0,T];\rH)$, and $Q_\delta$ is strongly convergent to the identity operator in $\rH$, for $\delta\searrow  0$. For each fixed $\delta \in\,(0,1]$, the family $\{\mathcal{L}(u^x_{\eps,\delta})\}_{\eps \in\,(0,1]}$ satisfies a large deviation principle in $C([0,T];\rH)$ with   action functional
\[S^\delta_{T}(u)=\frac 12\int_0^T\left|Q_\delta^{-1}\left(u^\prime(t)+\rA u(t)+B(u(t),u(t))\right)\right|_\rH^2\,dt,\]
and the corresponding quasi-potential is defined by
\[U_\delta(\phi)=\inf\left\{S^\delta_{T}(u)\ :\ u \in\,C([0,T];\rH),\ u(0)=0,\ u(T)=\phi,\ T>0\right\}.\]
Our purpose here is to show  that, despite we cannot prove any limit for the solution $u^\delta$ of equation \eqref{eq1}, nevertheless,  for all $ \phi \in \rH$ such that $U_\delta(\phi)<\infty$,
\begin{equation}
\label{brc101}
\lim_{\delta\to 0}U_\delta(\phi)=U(\phi),
\end{equation}
where
$U(\phi)$ is defined as $U_\delta(\phi)$, with the action functional $S^\delta_{T}$ replaced by
\[S_{T}(u)=\frac 12\int_0^T\left|u^\prime(t)+\rA u(t)+B(u(t),u(t))\right|_\rH^2\,dt.\]
To this purpose, the key idea consists in characterizing the quasi-potentials $U_\delta$ and $U$ as
\begin{equation}
\label{brc102}
U_\delta(\phi)=\min\left\{S^\delta_{-\infty}(u)\ :\ u \in\,\mathcal{X} \mbox{ and } {u(0)=\phi}\right\}
\end{equation}
and
\begin{equation}
\label{brc103}
U(\phi)=\min\left\{S_{-\infty}(u)\ :\ u \in\,\mathcal{X} \mbox{ and } {u(0)=\phi} \right\},
\end{equation}
where
\[\mathcal{X}=\big\{ u\in C((-\infty,0];\rH): \; \lim_{t\to-\infty}\vert u(t)\vert_{\rH}=0\big\}
\]
and the functionals $S^\delta_{-\infty}$ and $S_{-\infty}$ are defined on  $\mathcal{X}$ in a natural way, see formulae \eqref{eqn-A5bis} and \eqref{eqn-A5} later on.

In this way, in the definition of $U_\delta$ and $U$, the infimum with respect to time $T>0$ has disappeared and we have only to take the infimum of suitable functionals in the   space $\mathcal{X}_\phi:=\{ u\in \mathcal{X}: u(0)=\phi\}$. In particular, the convergence of $U_\delta(\phi)$ to $U(\phi)$ becomes the convergence of the infima of $S^\delta_{-\infty}$ in $\mathcal{X}_\phi$ to   the infimum of $S_{-\infty}$ in $\mathcal{X}_\phi$, so that \eqref{brc101} follows once we prove that $S^\delta_{-\infty}$ is Gamma-convergent to $S_{-\infty}$ in $\mathcal{X}_\phi$, as $\delta\searrow 0$. Moreover, as a consequence of \eqref{brc103},  in the case of the stochastic Navier-Stokes equations with periodic boundary conditions we can prove, see section \ref{sect-periodic},  that
\begin{equation}
\label{Vper}
U(\phi)=|\phi|^2_\rV.
\end{equation}
This means that $U(\phi)$ can be explicitly computed and the use of \eqref{brc101} in applications becomes particularly relevant. Let us point out that a similar explicit formula for the quasipotential has been derived for linear SPDEs by Da Prato, Pritchard and Zabczyk in \cite{DaPrato+P+Z_1991} and in the recent work by the second and  third authors for stochastic reaction diffusion equations in \cite{cerrai-freidlin}. A finite dimensional counterpart of our formula \eqref{Vper} was first derived  in Theorem IV.3.1 in the monograph \cite{freidlin}.

The proofs of  characterizations \eqref{brc102} and \eqref{brc103} and of the Gamma-convergence of $S^\delta_{-\infty}$  to $S_{-\infty}$ are based on a thorough analysis of the Navier-Stokes equation with an external deterministic force in the domain of suitable fractional powers of the operator $\rA$.

\medskip

One of the main motivation for proving  \eqref{brc101} comes from the study of the expected exit time $\tau_\phi^{\eps,\delta}$ of the solution $u_\phi^{\eps,\delta}$ from a domain $D$ in $L^2(\mathcal{O})$,  which is attracted to the zero function. Actually, in the second part of the paper we  prove that, under suitable regularity properties of $D$,
for any fixed $\delta>0$
\begin{equation}
\label{brc110}
\lim_{\eps \to 0}\eps\,\log\,\mathbb{E}(\tau_\phi^{\eps,\delta})=\inf_{y \in\,\partial D}U_\delta(y).\end{equation}
This means that, as in finite dimension, the expectation of $\tau_\phi^{\eps,\delta}$ can be described in terms of the quantity $U_\delta(\phi)$. Moreover, once we have \eqref{brc101}, by a general argument introduced in \cite{cerrai-freidlin} and based again on Gamma-convergence, we can prove that if $D$ is a domain in $\rH$ such that any point $\phi \in\,\rV\cap \partial D$  can be approximated in $\rV$ by a sequence $\{\phi_n\}_{n \in\,\mathbb{N}}\subset D(\rA^{\frac 12+\a})\cap \partial D$ (think for example of $D$ as a ball in $\rH$), then
\[\lim_{\delta \to 0}\,\inf_{\phi \in\,\partial D}U_\delta(\phi)=\inf_{\phi \in\,\partial D}U(\phi).\]
According to \eqref{brc110}, this implies that for $0<\eps<<\delta<<1$
\[\mathbb{E}\,\tau_\phi^{\eps,\delta}\sim \exp\left(\frac 1\eps\inf_{\phi \in\,\partial D}U(\phi)\right).\]
In particular, if $D$ is the ball of $\rH$ of radius $c$ and the boundary conditions are periodic, in view of \eqref{Vper} for any $\phi \in\,D$ we get,
\[e^{-\frac {c^2 \lambda_1^2}{\eps}}\, \mathbb{E}\,\tau_\phi^{\eps,\delta}\sim 1,\ \ \ 0<\eps<<\delta<<1.\]

\medskip

At the end of this long introduction, we would like to point out that although 2-D  stochastic Navier-Stokes equations with periodic boundary conditions have been investigated by Flandoli and Gozzi in \cite{Flandoli+Gozzi_1998} and Da Prato and Debussche in \cite{DaPrato+D_2002} from the point of view of Kolmogorov equations and the existence of a Markov process, we do not know whether our results (even in the periodic case) could be derived from these papers. One should bear in mind that the solution from    \cite{DaPrato+D_2002} exists for almost every initial data $u_0$ from a certain Besov space of negative order with respect to a specific Gaussian measure while we construct a quasipotential for every $u_0$ from the space $\rH$ whose measure is equal to $0$. Of course our results are also valid for 2-D  stochastic Navier-Stokes equations with Dirichlet boundary conditions.

{
We have been recently become aware of a  work by F. Bouchet et all \cite{Bouchet_2014} where somehow related issues are considered from a physical point of view. We hope to be able to understand the relationship between our work and this work in a future publication.
}

\subsection*{Acknowledgments}
Research of the first named author was also supported by the EPSRC grant  EP/E01822X/1, research of the second named author was also supported by NSF grant DMS0907295 and research of the third named author was also  supported by the NSF grants DMS 0803287 and 0854982.

The first named author  would like to thank the University of Maryland for hospitality during his visit in Summer 2008. The second named author would like to thank Michael Salins for some useful discussions about Lemmas \ref{lem3bis} and \ref{lem4}.

Finally, the three authors would like to thank the two anonymous referees for reading carefully the original manuscript and giving very good suggestions that helped us to improve considerably the final version of the paper.

\section{Notation and preliminaries}\label{sec-prel}

Let $\mathcal{O}\subset\mathbb{R}^{2}$ be an open and bounded set. We denote by  $\Gamma=\partial\mathcal{O}$ the  boundary of $\mathcal{O}$. We will always assume that the closure $\overline{\mathcal{O}}$ of the set $\mathcal{O}$ is a manifold with boundary of $C^\infty$ class, whose boundary $\partial\mathcal{O}$ is denoted by $\Gamma$,   is a
$1$-dimensional infinitely differentiable manifold being
locally on one side of $\mathcal{O}$, see   condition (7.10) from \cite[chapter I]{Lions+Magenes_1972}. Let us also denote by $\nu$ the unit outer normal vector field to $\Gamma$.

It is known that   $\mathcal{O}$ is a Poincar\'e
domain, i.e.   there exists a constant $\lambda_1>0$ such that  the following Poincar\'{e} inequality is satisfied
\begin{equation}
\lambda_1\int_{\mathcal{O}}\varphi^{2}(x)\,dx\leq\int_{\mathcal{O}}|\nabla
\varphi(x)|^{2}\,dx,\ \ \ \varphi\in H^1_0({\mathcal O}).
\label{ineq-Poincare}
\end{equation}

 In order to formulate  our
problem in an abstract framework,  let us recall the definition of  the following
functional spaces. First of all, let $\mathcal{D}(\mathcal{O})$ (resp. $\mathcal{D
}(\overline{\mathcal{O}})$) be the set of all $C^\infty$ class vector  fields  $u:\mathbb{R}^2\to \mathbb{R}^2$ with compact support   contained in the set $\mathcal{O}$ (resp. $\overline{\mathcal{O}}$). Then, let us define
\begin{eqnarray*}
	E(\mathcal{O}) &=& \{ u\in L^2(\mathcal{O},\mathbb{R}^2):\divv u \in\,L^2(\mathcal{O})\},\\
\mathcal{V}&=&\big\{  u\in \mathcal{D}(\mathcal{O}): \divv u=0\big\} ,\\
\rH&=& \mbox{the closure of $\mathcal{V}$ in } L^2(\mathcal{O}),\\
\rH_0^1(\mathcal{O},\mathbb{R}^2)&=& \mbox{the closure of $\mathcal{D}(\mathcal{O},\mathbb{R}^2)
$ in } \rH^1(\mathcal{O},\mathbb{R}^2),\\
V&=& \mbox{the closure of $\mathcal{V}$ in } \rH^1_0(\mathcal{O},\mathbb{R}^2).
\end{eqnarray*}
The inner products in all the $L^2$ spaces will be denoted by $(\cdot,\cdot)$. The space $E(\mathcal{O})$ is a Hilbert space with a scalar product \begin{equation}\label{Temam_1.13}
\lb u,v\rb_{E(\mathcal{O})}:=(u,v)_{L^2(\mathcal{O},\mathbb{R}^2)}+(\divv u, \divv v)_{L^2(\mathcal{O},\mathbb{R}^2)}.
\end{equation}
We endow the set $H$ with the inner product $(\cdot,\cdot)_H$ and the norm   $\left\vert \cdot\right\vert_H $ induced by $L^2(\mathcal{O},\mathbb{R}^2)$.
Thus, we have
\[
(u,v)_H=\sum_{j=1}^{2}\int_{\mathcal{O}}u_{j}(x)v_{j}(x)\,{d}x,
\]
The space $H$ can also be characterised in the following way.  Let $H^{-\frac 12}(\Gamma)$ be the dual space of $H^{1/2}(\Gamma)$, the image in $L^2(\Gamma)$ of the trace operator $\gamma_0:H^{1}(\mathcal{O}) \to L^2(\Gamma)$ and let
$\gamma_\nu$ be the bounded linear map from $E(\mathcal{O})$ to $H^{-\frac12}(\Gamma)$ such that, see  \cite[Theorem I.1.2]{Temam_2001},
\begin{equation}\label{Temam_1.18}
\gamma_\nu(u)= \mbox{ the restriction of } u\cdot \nu \mbox{ to } \Gamma, \;\; \mbox{ if } u\in \mathcal{D}(\overline{\mathcal{O}}).
\end{equation}
Then, see \cite[Theorem I.1.4]{Temam_2001},
\[\begin{array}{lcl}
\rH&=&\{ u \in E(\mathcal{O}): \divv u=0 \mbox{ and } \gamma_\nu(u)=0\},
\\
&\vspace{.1cm} &\\
\rH^\perp&=&\{ u \in E(\mathcal{O}):  u= \nabla p,\ p \in H^{1}(\mathcal{O}) \}.
\end{array}\]
Let us denote by $\mathrm{P}:L^2(\mathcal{O},\mathbb{R}^2) \rightarrow \mathrm{H}$
 the orthogonal projection called usually the Leray-Helmholtz projection. It is known, see for instance \cite[Remark I.1.6]{Temam_2001} that
\begin{equation}\label{Temam_Remark 1.6}
P u= u -\nabla (p+q), \;\; u\in L^2(\mathcal{O},\mathbb{R}^2),
\end{equation}
where, for $u\in L^2(\mathcal{O})$, $p$ is the unique solution of the following homogenous boundary Dirichlet problem for the Laplace equation
\begin{equation}
\label{brc33}
 \Delta p= \divv u \in H^{-1}(\mathcal{O}),\;\; p{\vert}_{\Gamma}=0\delh{\in H^{1}_0(\mathcal{O})}.\end{equation}
and $q\in H^{1}(\mathcal{O})$ is the unique solution of the following in-homogenous  Neumann boundary  problem for the Laplace equation
\begin{equation}
\label{brc34}
 \Delta q= 0 , \;\; \frac{\partial q}{\partial \nu}{\Big\vert}_{\Gamma}= \gamma_\nu(u-\nabla p).
 \end{equation}
Note that the function $p$ above satisfies $\nabla p \in L^2(\mathcal{O},\mathbb{R}^2)$ and $\divv (u-\nabla p)=0$. In particular, $u-\nabla p\in E(\mathcal{O})$ so that $q$ is well defined. \\
It is proved in \cite[Remark I.1.6]{Temam_2001} that $P$ maps continuously the Sobolev space $\rH^{1}(\mathcal{O},\mathbb{R}^2)$ into itself. Below, we will discuss continuity of $P$ with respect to other topologies.

Since the set $\mathcal{O}$ is a Poincar\'e domain,  the norms on the space $\rV$ induced by norms from the Sobolev spaces $\rH^1(\mathcal{O},\mathbb{R}^2)$ and $\rH_0^1(\mathcal{O},\mathbb{R}^2)$ are equivalent.  The latter norm   and the associated inner
product will be denoted by $\left\vert \cdot\right\vert_\rV $ and $\big(\cdot,\cdot\big)_\rV$, respectively.  They satisfy the following equality
\[
\big(u,v\big)_\rV=\sum_{i,j=1}^{2}\int_{\mathcal{O}}{\frac{\partial
u_{j}}{\partial x_{i}}\frac{\partial v_{j}}{\partial
x_{i}}}\,{d}x,\ \ \ \ u,v\in\rH^1_0(\mathcal{O},\mathbb{R}^2).
\]
Since the space $\mathrm{V} $ is  densely and continuously
 embedded into $\mathrm{H}$, by   identifying   $\mathrm{H}$
 with its dual $\mathrm{H}^\prime$, we have the following embeddings
 \begin{equation}
 \label{eqn:Gelfanf}
\mathrm{V} \subset \mathrm{H}\cong\mathrm{H}^\prime \subset
\mathrm{V}^\prime.\end{equation}
Let us observe here  that, in particular,  the spaces $\mathrm{V}$, $\mathrm{H}$ and
$\mathrm{V}^\prime$ form a Gelfand triple.

We will denote by  $| \cdot |_{\rV^{\prime}}$   and  $\left\langle \cdot,\cdot\right\rangle $ the norm in
$\rV^{\prime}$ and  the
duality pairing between $\rV$ and $\rV^{\prime}$, respectively.

The presentation of the  Stokes operator is standard and we  follow here the one given in  \cite{Brz+Li_2006}.
We  first define the bilinear
form $a:\mathrm{V}\times \mathrm{V} \to \mathbb{R}$ by setting
 \begin{equation}
 \label{form-a}
a(u,v):=(\nabla u,\nabla v)_\rH, \quad u,v \in \mathrm{V}.
\end{equation}
As  obviously the bilinear form   $a$ coincides with   the  scalar product in $\mathrm{V}$, it is
$\mathrm{V}$-continuous, i.e. there exists some $C>0$ such that
\[ \vert a(u,u) \vert
\leq C \vert u \vert_\rV^2,\ \ \ \ u \in\,\mathrm{V}\]
Hence, by the Riesz Lemma,
 there exists a unique linear operator
$\mathcal{A}:\mathrm{V} \to \mathrm{V}^\prime$,  such that
$a(u,v)=\lb \mathcal{A}u,v\rb$, for $u, v \in \mathrm{V}$. Moreover, since  $\mathcal{O}$ is a Poincar\'e
domain,  the form $a$ is  $\mathrm{V}$-coercive, i.e. it satisfies
$a(u,u) \geq \alpha \vert u \vert_\rV^2$ for some $\alpha
>0$ and all $u \in \mathrm{V}$. Therefore, in view  of the Lax-Milgram theorem, see for instance
   Temam \cite[Theorem II.2.1]{Temam_2001},
 the operator $\mathcal{A}:\mathrm{V} \to \mathrm{V}^\prime$ is an
 isomorphism.

 Next we  define an unbounded linear operator
$\mathrm{A}$ in $\mathrm{H}$ as follows
\begin{equation}
\label{def-A} \left\{
\begin{array}{ll}
D(\mathrm{A}) &= \{u \in \mathrm{V}: \mathcal{A}u \in
\mathrm{H}\}\\
&\vspace{.1cm} \\
\mathrm{A}u&=\mathcal{A}u, \, u \in
D(\mathrm{A}).
\end{array}
\right.
\end{equation}

It is now well established that under suitable  assumptions\footnote{These assumptions are satisfied in our case}
related to  the regularity of the domain $\mathcal{O}$, the space
$D(\mathrm{A})$ can be characterized in terms of  the Sobolev spaces.
For example, (see \cite{Heywood-80}, where only the 2-dimensional case
is studied but the result is also valid in the 3-dimensional case),
if $\mathcal{O}\subset \mathbb{R}^2$ is a uniform $C^2$-class
Poincar\'{e} domain, then we have
\begin{equation}
\label{eqn:4.3} \left\{
\begin{array}{l}
D(\mathrm{A}) = \mathrm{V} \cap \rH^2(\mathcal{O},\mathbb{R}^2)=\mathrm{H} \cap \rH^1_0(\mathcal{O},\mathbb{R}^2) \cap \rH^2(\mathcal{O},\mathbb{R}^2),\\
\vspace{.1cm} \\
\mathrm{A}u=-\mathrm{P}\Delta u, \quad u\in D(\mathrm{A}).
\end{array}
\right.
\end{equation}

It is also a classical result,  see e.g. Cattabriga \cite{Cattabriga_1961} or Temam
\cite[p. 56]{Temam_1997},  that $\mathrm{A}$ is a positive self adjoint operator in
$\mathrm{H}$ and
\begin{equation}
(\rA u, u) \geq \lambda_1 \vert u\vert^2_{\rH}, \;\; u\in D(\rA).
\label{ineq-Poincare-A}
\end{equation}
where the  constant $\lambda_1>0$ is from   the  Poincar\'{e} inequality \eqref{ineq-Poincare}.
Moreover, it is well known,  see for instance    \cite[p. 57]{Temam_1997} that $\mathrm{V}=D(\mathrm{A}^{1/2})$.
Moreover,
from   \cite[Theorem 1.15.3, p. 103]{Triebel-95} it follows that
\[D(\rA^{\alpha/2})=[\rH,D(\rA)]_{\frac{\alpha}{2}},\] where
$[\cdot,\cdot]_\frac{\alpha}{2}$ is the complex interpolation functor  of
order $\frac{\alpha}{2}$, see e.g. \cite{Lions+Magenes_1972}, \cite{Triebel-95}
and   \cite[Theorem 4.2]{Taylor_1981}. Furthermore, as shown in   \cite[Section 4.4.3]{Triebel-95}, for $\alpha  \in (0, \frac12)$

\begin{equation}\label{eqn-domains}
 D(\rA^{\alpha/2})= \rH \cap \rH^{\alpha}(\mathcal{O},\mathbb{R}^2).
\end{equation}
The above equality leads to  the following result.

\begin{proposition}\label{prop-Leray-fractional} Assume that $\alpha  \in (0, \frac12)$.
Then the  Leray-Helmholtz projection $P$  is a well defined and continuous map  from
$\rH^{\alpha}(\mathcal{O},\mathbb{R}^2)$ into  $ D(\rA^{\alpha/2})$.
\end{proposition}
\begin{proof}
Let us fix $\alpha \in (0,\frac12)$. Since, by its definition, the range of $P$  is contained in  $\rH$, it is sufficient to prove that for every $u\in  \rH^{\alpha}(\mathcal{O},\mathbb{R}^2)$, $Pu \in \rH^{\alpha}(\mathcal{O},\mathbb{R}^2)$. For this aim, let us fix $u\in  \rH^{\alpha}(\mathcal{O},\mathbb{R}^2)$. Then $\divv u \in
H^{\alpha-1}(\mathcal{O})$. Therefore, by the elliptic regularity  we infer that the solution $p$ of the problem \eqref{brc33} belongs to the Sobolev space $H^{\alpha+1}(\mathcal{O})\cap H_0^{1}(\mathcal{O})$ and therefore $\nabla p \in \rH^{\alpha}(\mathcal{O},\mathbb{R}^2)$. \\
Since by \cite[Theorem I.1.2]{Temam_2001},   the linear map $\gamma_\nu$ is   bounded  from $E(\mathcal{O})$ to $H^{-\frac12}(\Gamma)$ and  from $E(\mathcal{O}) \cap \rH^{1}(\mathcal{O},\mathbb{R}^2)$ to $H^{\frac12}(\Gamma)$, by a standard interpolation argument we infer that $\gamma_\nu$ is a  bounded linear map from $E(\mathcal{O}) \cap  \rH^{\alpha}(\mathcal{O},\mathbb{R}^2)$ to $H^{-\frac12+\alpha}(\Gamma)$. Thus we infer that $\gamma_\nu(u-\nabla p) \in\,H^{-\frac12+\alpha}(\Gamma)$ and by the Stokes formula (I.1.19) from \cite{Temam_2001},
$\lb \gamma_\nu(u-\nabla p),1\rb=0$. Therefore, again by the elliptic regularity, see for instance \cite{Lions+Magenes_1972},
 the solution $q$ of the problem \eqref{brc34} belongs to  $H^{\alpha+1}(\mathcal{O})$ and therefore $\nabla q \in \rH^{\alpha}(\mathcal{O},\mathbb{R}^2)$. This proves that $Pu \in \rH^{\alpha}(\mathcal{O},\mathbb{R}^2)$ as required.\\
The proof is complete.
\end{proof}

\begin{remark}\label{rem-Leray-fractional} We only claim that the above result is true for $\alpha<\frac12$. In particular, we do not claim that is $P$ is a bounded linear map from $\rH^{1}(\mathcal{O},\mathbb{R}^2)$ to  $\rV$ and we are not aware of such a result. However, if this is true, Proposition \ref{prop-Leray-fractional} will hold for any $\alpha \in (0,1)$ with a simple proof by complex interpolation. However, it seems to us that the result for $\alpha>\frac12$ is not true, since we cannot see how one could prove that $Pu\vert _{\partial\mathcal{O}}=0$. The reason why
Proposition \ref{prop-Leray-fractional} holds for any $\alpha \in (0,\frac12)$ is that according to identity \eqref{eqn-domains} the  only boundary conditions satisfied by functions belonging to  $D(\rA^{\alpha/2})$ are those satisfied by functions belonging to the space $\rH$. One can compare with the paper \cite{Temam_1982} by Temam (or chapter 6 of his book \cite{Temam_1983}).

\end{remark}

Let us  finally recall that by a result of
Fujiwara--Morimoto \cite{Fuj+Mor_1977}  the projection $\mathrm{P}$ extends to a
bounded linear projection in the space $L^q(\mathcal{O},\mathbb{R}^2)$, for any $q \in\,(1,\infty)$.

Now, consider the trilinear form $b$ on $V\times V\times V$ given by
\[
b(u,v,w)=\sum_{i,j=1}^{2}\int_{\mathcal{O}}u_{i}{\frac{\partial v_{j}%
}{\partial x_{i}}}w_{j}\,\,{d}x,\quad u,v,w\in \rV.
\]
Indeed, $b$ is a continuous trilinear form such that
\begin{equation}
\label{eqn:b01}
b(u,v,w)=-b(u,w,v),
\quad  \, u\in \mathrm{V}, v, w\in
\rH_0^{1}(\mathcal{O},\mathbb{R}^2),
\end{equation}
and
\begin{equation}
\begin{aligned}
\label{eqn:4.0a}
 \vert b(u,v,w)  \vert \leq C\left\{
\begin{array}{ll}
 \vert u\vert_\rH ^{1/2}\vert \nabla u\vert_\rH^{1/2}\vert \nabla  v\vert _\rH^{1/2}\vert \mathrm{A}v\vert_\rH ^{1/2}\vert w\vert_\rH  &
\quad \, u \in \mathrm{V}, v\in D(\mathrm{A}), w\in \mathrm{H}\\
\vert u\vert_\rH ^{1/2}\vert \mathrm{A}u\vert_\rH ^{1/2}\vert \nabla  v\vert _\rH\vert w\vert_\rH & \quad
 u \in D(\mathrm{A}), v\in \mathrm{V}, w\in \mathrm{H}\\
\vert u\vert_\rH \vert \nabla  v\vert_\rH \vert w\vert_\rH ^{1/2}\vert \mathrm{A}w\vert _\rH^{1/2} & \quad
 u \in \mathrm{H}, v\in \mathrm{V}, w\in D(\mathrm{A})\\
\vert u\vert _\rH^{1/2}\vert \nabla  u\vert_\rH ^{1/2}\vert \nabla  v\vert_\rH \vert w\vert_\rH ^{1/2}\vert \nabla  w\vert_\rH ^{1/2} & \quad
 u, v, w \in \mathrm{V},
\end{array}
\right.
\end{aligned}
\end{equation}
for some constant $C>0$ (for a proof see  for instance \cite[Lemma 1.3, p.163]{Temam_2001}  and  \cite{Temam_1997}).

Define next the bilinear map $B:\rV\times \rV\rightarrow \rV^{\prime}$ by setting
\[
\left\langle B(u,v),w\right\rangle =b(u,v,w),\quad u,v,w\in \rV,
\]
and the homogenous polynomial of second degree $B:\rV \rightarrow \rV^{\prime}$ by

\[
B(u)=B(u,u),\; u\in \rV.
\]
Let us observe that if $v \in\,D(\rA)$, then $B(u,v) \in\,H$ and the following inequality  follows directly from the first inequality in \eqref{eqn:4.0a}
\begin{equation}
\label{ineq-B01}
\vert \rB(u,v) \vert_\rH^2 \leq C  \vert u\vert_\rH\vert \nabla u\vert_\rH\vert \nabla  v\vert_\rH\vert \mathrm{A}v\vert_\rH, \; u\in \rV, \, v\in D(\rA).
\end{equation}
Moreover, the following identity is a direct consequence of \eqref{eqn:b01}.
\begin{equation}
\label{eqn-B02}
\lb  \rB(u,v),v \rb =0,\;\;  u,v\in \rV.
\end{equation}

Let us also recall the following fact (see \cite[Lemma 4.2]{Brz+Li_2006}).

\begin{lemma}\label{lem:form-b} The trilinear map
$b:\mathrm{V}\times \mathrm{V} \times \mathrm{V} \to \mathbb{R}$
has a unique extension to a bounded trilinear map from
$L^4(\mathcal{O},\mathbb{R}^2) \times (L^4(\mathcal{O},\mathbb{R}^2)
\cap\mathrm{H})\times \mathrm{V}$ and from
$L^4(\mathcal{O},\mathbb{R}^2) \times \mathrm{V}\times
L^4(\mathcal{O},\mathbb{R}^2)$ into $\mathbb{R}$. Moreover, $B$ maps
$L^4(\mathcal{O},\mathbb{R}^2)\cap\mathrm{H}$ (and so $\mathrm{V}$)
into $\mathrm{V}^\prime$ and
\begin{equation}\label{eqn:4.0}
\vert B(u) \vert_{\mathrm{V}^\prime} \leq C_1\vert u
\vert^2_{L^4(\mathcal{O},\mathbb{R}^2)} \leq 2^{1/2}C_1 \vert u \vert_\rH  \vert
\nabla  u \vert _\rH\leq C_2\vert u \vert_\rV^2 , \quad
 u \in \mathrm{V}.
\end{equation}
\end{lemma}
\begin{proof} It it enough to observe
that due to the H\"older inequality,
  the following
inequality holds
\begin{equation}\label{eqn:4.00}
\vert b(u,v, w)\vert \le C\vert u\vert_{L^4(\mathcal{O},\mathbb{R}^2)} \vert
\nabla v \vert _{L^2(\mathcal{O})}  \vert w
\vert_{L^4(\mathcal{O},\mathbb{R}^2)}, \quad u, v, w \in \rH_0^{1}(\mathcal{O},\mathbb{R}^2).
\end{equation}
Thus, our result follows from \eqref{eqn:b01}.
\end{proof}
Let us also recall the following well known result, see \cite{Temam_2001} for a proof.

\begin{lemma}\label{lem-B}
For any $T\in  (0,\infty]$ and  for any $u\in L^2(0,T;D(\rA))$ with $u^\prime \in  L^2(0,T;\rH)$, we have
\[\int_0^T \vert \rB(u(t),u(t)) \vert_\rH^2 \, dt <\infty.\]
\end{lemma}
\begin{proof}
Our assumption implies that\footnote{Note that  in the case $T=\infty$ one also has $\lim_{t\to \infty} u(t)=0$ in $\rV$.} $u\in C([0,T];\rV)$ (for a proof see for instance \cite[Proposition I.3.1]{Vishik+Fursikov_1988}). Then, we can conclude thanks to \eqref{ineq-B01}
\end{proof}

The restriction of the map $\rB$ to the space $D(\rA)\times D(\rA)$ has also the following representation
\begin{equation}
\label{eqn-B-using-LH}
  \rB(u,v)= P( u\nabla v), \;\; u,v\in D(\rA),
\end{equation}
where $P$ is the Leray-Helmholtz projection operator and $u\nabla v=\sum_{j=1}^2 u^jD_jv \in L^2(\mathcal{O},\mathbb{R}^2)$.
This representation together with Proposition \ref{prop-Leray-fractional} allows us to prove the following property of the map $B$.

\begin{proposition}\label{prop-Leray-fractional-alpha} Assume that $\alpha  \in (0, \frac12)$.
Then for any $s \in\,(1,2]$ there exists a constant $c>0$ such that
\begin{equation}
\label{ineqn-B-fractional}
  \vert \rB(u,v)\vert_{D(\rA^{\alpha/2})} \leq  c \vert u\vert_{D(\rA^{\frac{s}2})}  \vert v\vert_{D(\rA^{\frac{1+\alpha}2})} , \;\; u,v\in D(\rA).
\end{equation}
\end{proposition}
\begin{proof}
In view of equality \eqref{eqn-B-using-LH}, since
 the  Leray-Helmholtz projection $P$  is a well defined and continuous map  from
$\rH^{\alpha}(\mathcal{O},\mathbb{R}^2)$ into  $ D(\rA^{\alpha/2})$ and since the norms in the spaces $D(\rA^{\frac{s}2})$ are equivalent to norms in $\rH^s(\mathcal{O},\mathbb{R}^2)$, it is enough to show that
\[
\vert u\nabla v \vert_{\rH^{\alpha}} \leq  C \vert u\vert_{\rH^{s}}  \vert v\vert_{\rH^{1+\alpha}} , \;\; u,v\in \rH^2(\mathcal{O},\mathbb{R}^2).
\]
The last inequality is a consequence of the Marcinkiewicz Interpolation Theorem, the complex interpolation and the following two inequalities for scalar functions which can be proved by using Gagliado-Nirenberg inequalities
\begin{eqnarray*}
\vert u v \vert_{L^2} &\leq&  C \vert u\vert_{H^{s}}  \vert v\vert_{L^2} , \;\; u \in H^{s},  v\in L^2,\\
\vspace{.1cm} \\
\vert u v \vert_{H^{1}} &\leq&  C \vert u\vert_{H^{s}}  \vert v\vert_{H^{1}} , \;\; u \in H^{s},  v\in H^{1}.
\end{eqnarray*}
\end{proof}

\section{The skeleton equation}\label{sec-skeleton}

We are here dealing with  the following functional version of the Navier-Stokes equation
\begin{equation}
\label{eqn_NSE01}
\left\{
\begin{array}{l}
u^\prime(t)+\nu \rA u(t)+\rB(u(t),u(t))=f(t), \; t\in (0,T)\\
\vspace{.1cm} \\
u(0)=u_0,
\end{array}\right.\end{equation}
where $T\in (0,\infty]$ and $\nu>0$. Let us recall the following definition (see \cite[Problem 2, section III.3]{Temam_2001})

\begin{definition}
Given {$T>0$,} $f\in L^2(0,T;\rV^\prime)$ and $u_0\in \rH$, a solution to problem \eqref{eqn_NSE01} is a function $u\in L^2(0,T;\rV)$ such that $u^\prime \in L^2(0,T;\rV^\prime)$,  $u(0)=u_0$ \footnote{It is known, see for instance \cite[Lemma III.1.2]{Temam_2001} that these two properties of $u$ imply that there exists a unique $\bar{u}\in C([0,T],\rH)$. When we write $u(0)$ later we mean $\bar{u}(0)$.} and \eqref{eqn_NSE01} is fulfilled.
\end{definition}

It is known (see e.g.  \cite[Theorems III.3.1/2]{Temam_2001}) that if $T\in (0,\infty]$, then for every  $f\in L^2(0,T;\rV^\prime)$ and $u_0\in \rH$ there exists exactly one solution $u$ to  problem \eqref{eqn_NSE01}{, which satisfies
\begin{equation}\label{ineq-aux-00}
\vert u\vert^2_{C([0,T],\rH)} +\vert u\vert^2_{L^2(0,T,\rV)}\leq \vert u_0\vert_{\rH}^2+\vert f\vert^2_{L^2(0,T,\rV^\prime)}.
\end{equation}}
Moreover, see \cite[Theorem III.3.10]{Temam_2001}, if $T<\infty$ and $f\in L^2(0,T;\rH)$ then
\begin{equation}
\label{brc82}
\begin{array}{l}
\hspace{-1truecm}{\sup_{t \in (0,T]} | \sqrt{t}\,u (t) |_V^2+ \int_0^T \vert \sqrt{t}\,\rA u (t) \vert_{\rH}^2\, dt }\\
\vspace{.1cm}\\
\leq \Big( \vert u_0\vert_{\rH}^2+ \vert f\vert^2_{L^2(0,T;\rV^\prime)}+T\vert f\vert^2_{L^2(0,T;\rH)} \Big) e^{\,c\big( \vert u_0\vert_{\rH}^4+ \vert f\vert^4_{L^2(0,T;\rV^\prime)} \big),}
\end{array}
\end{equation}
{for  a  constant $c$  independent of $T$, $f$ and $u_0$. }

Finally, if $T\in (0,\infty]$, $f\in L^2(0,T;\rH)$ and $u_0\in \rV$,  then {the unique solution $u$ satisfies}
$$u\in L^2(0,T;D(\rA)) \cap C([0,T];\rV),\;\; u^\prime\in L^2(0,T;\rH),$$
{and, for\footnote{Please note that $ [0,T]\cap[0,\infty)$ is equal to $[0,T]$ if $T<\infty$ and to $[0,\infty)$ if $T=\infty$.}
$ t \in [0,T]\cap [0,\infty)$,
\begin{eqnarray}
\label{ineq-aux-02}
\frac{d}{dt}| u(t)|_V^2+\lambda_1 |u(t)|_V^2
 &\leq &
\frac d{dt}| u(t)|_V^2+ \vert \rA u(t)\vert_H^2\\
 &\leq&2 \vert f(t)\vert_{\rH}^2 +108 |u(t)|_{V}^2 \vert u(t)\vert_{\rH}^2 | u(t)|_V^2.
\nonumber
\end{eqnarray}
 Hence, by the Gronwall Lemma and inequality \eqref{ineq-aux-00}}, for any $t \in [0,T]\cap [0,\infty)$
\begin{equation}
\label{ineq-u-V}
| u(t) |_V^2 \leq \Big( | u_0|_V^2  +2 \int_0^t\vert f(s)\vert_{\rH}^2\,ds  \Big) e^{ -\lambda_1 t+ 54 \big( \vert u_0\vert_{\rH}^2+ \int_0^t\vert f(s)\vert_{\rV^\prime}^2\,ds  \big)^2}.
\end{equation}
so that, in particular,

\begin{equation}
\label{ineq-L^infty(0,T,V)}
e^{\lambda_1 t}\vert u\vert^2_{C([0,T],\rV)} \leq \Big( | u_0|_V^2  +{2} \vert f\vert^2_{L^2(0,T;\rH)} \Big) e^{ 54 \big( \vert u_0\vert_{\rH}^2+ \vert f\vert^2_{L^2(0,T;\rV^\prime)} \big)^2}.
\end{equation}
Moreover, thanks to \eqref{ineq-aux-02}, this yields
{\begin{equation}
\label{ineq-L^2(0,T,H^2)}
\begin{array}{lcl}
\hspace{-3truecm}\lefteqn{\vert u\vert^2_{L^2(0,T,D(\rA))} \leq  | u_0|_V^2 + {2}\vert f\vert^2_{L^2(0,T;\rH)}
+ 54 \,\Big( \vert u_0\vert_{\rH}^2+\vert f\vert^2_{L^2(0,T,\rV^\prime)}\Big)^2}\\
&& \times \Big( | u_0|_V^2+ 2 \vert f\vert^2_{L^2(0,T;\rH)} \Big)
e^{ 54 \big( \vert u_0\vert_{\rH}^2+ \vert f\vert^2_{L^2(0,T;\rV^\prime)} \big)^2}
\\&&\hspace{-3truecm}\lefteqn{ =\Big(| u_0|_V^2+ 2\vert f\vert^2_{L^2(0,T;\rH)} \Big)\Big( 1+ 54 \,\big( \vert u_0\vert_{\rH}^2+\vert f\vert^2_{L^2(0,T,\rV^\prime)}\big)^2  e^{ 54 \big( \vert u_0\vert_{\rH}^2+ \vert f\vert^2_{L^2(0,T;\rV^\prime)} \big)^2}\Big).}
\end{array}
\end{equation}
}
The above results and arguments yield in particular the following corollary.

\begin{corollary}\label{cor-u-nfty-V}
If $f\in L_{\rm loc}^2(0,\infty ;\rH)$ and $u_0\in \rV$, then the  solution $u$ to problem \eqref{eqn_NSE01} satisfies
 $u \in L_{\rm loc}^2(0,\infty;D(\rA))$, $u\in C([0,\infty),\rV)$ and
\begin{equation}
\label{ineq-u-V2}
e^{ \lambda_1 t} | u(t) |_V^2 \leq \Big( | u_0|_V^2  +2 \int_0^t\vert f(s)\vert_{\rH}^2\,ds  \Big)  e^{54 \big( \vert u_0\vert_{\rH}^2+ \int_0^t\vert f(s)\vert_{\rV^\prime}^2\,ds  \big)^2},\;\;  t \geq 0.
\end{equation}
  In particular, if $f\in L^2(0,\infty ;\rH)$ then  $u \in L^2(0,\infty;D(\rA))$ and
\begin{equation}
\label{ineq-u-V3}
e^{ \lambda_1 t} | u(t) |_V^2 \leq \Big( | u_0|_V^2  +2 \int_0^\infty\vert f(s)\vert_{\rH}^2\,ds  \Big)  e^{54 \big( \vert u_0\vert_{\rH}^2+ \int_0^\infty\vert f(s)\vert_{\rV^\prime}^2\,ds  \big)^2},\;\;  t \geq 0.
\end{equation}
If also $f=0$, this gives
\begin{equation}
\label{ineq-u-V4}
 | u(t) |_V^2 \leq   e^{54 \vert u_0\vert_{\rH}^4 }e^{-\lambda_1 t} \, | u_0|_V^2 ,\;\;  t \geq 0.
\end{equation}

\end{corollary}

Now we will formulate and prove some generalizations of the above results when the data $u_0$ and $f$ are slightly more regular. Similar results in the case of integer order of the Sobolev spaces has been studied in \cite{Temam_1982} where some compatibility conditions are imposed.

\begin{proposition}
\label{prop-NSE-fractional}
Assume that $\alpha  \in (0, \frac12)$. If {$T \in (0,\infty]$,}  $f\in L^2(0,T;D(\rA^{\frac{\alpha}{2}}))$ and $u_0\in D(\rA^{\frac{\alpha+1}{2}})$,  then the unique solution $u$  to  problem \eqref{eqn_NSE01} satisfies
\begin{equation}\label{eqn-u-better regularity}
u\in L^2(0,T;D(\rA^{1+\frac{\alpha}{2}})) \cap C([0,T];D(\rA^{\frac{\alpha+1}{2}}))\mbox{ and } \;\;u^\prime(\cdot)\in L^2(0,T;D(\rA^{\frac{\alpha}{2}})).
\end{equation}
{Moreover,  for $t \in [0,T]\cap[0,\infty)$, we have
\begin{eqnarray}\label{ineq-aux-04}
 \vert \rA^{\frac{\alpha+1}2} u(t)\vert_{\rH}^2
 &\leq& e^{-\lambda_1 t} e^{C^2K_3(| u_0|_V,  \vert f \vert_{L^2(0,T;\rH)})} \big( \vert \rA^{\frac{\alpha+1}2} u_0\vert_{\rH}^2+
 \int_0^T \vert \rA^{\frac{\alpha}{2}} f(s)\vert_{\rH}^2\, ds\big),
\end{eqnarray}
where $C>0$ is a generic constant\footnote{In fact, the one from inequality \eqref{ineqn-B-fractional} in Proposition \ref{prop-Leray-fractional-alpha}.} and
\begin{equation}\label{ineq-aux-05}
K_3(R,\rho):=\big[ R^2+ 2\rho^2 \big] \times \Big[ 1+ \frac{54}{\lambda_1^2} \,\big( R^2+\rho^2\big)^2
e^{ \frac{54}{\lambda_1^2} \,\big( R^2+\rho^2\big)^2 }\Big].
\end{equation}
In particular, if $f=0$, then
\begin{equation}\label{ineq-aux-04'}
 \vert \rA^{\frac{\alpha+1}2} u(t)\vert_{\rH}^2
 \leq e^{-\lambda_1 t} e^{C^2K_3(| u_0|_V, 0)}  \vert \rA^{\frac{\alpha+1}2} u_0\vert_{\rH}^2, \;\; t \geq 0.
\end{equation}
}
\end{proposition}

\begin{proof} Let us fix $T>0$. Since by Proposition \ref{prop-Leray-fractional}, $B$ is a bilinear continuous map from $D(\rA^{\frac{\alpha+1}{2}}) \times D(\rA^{\frac{\alpha+1}{2}})$ to $D(\rA^{\frac{\alpha}{2}})$ it follows (see for instance \cite{Brz_1991} for the simplest argument) that for every $R,\rho>0$ there exists $T_\ast=T_\ast(R,\rho)\in (0,T]$ such that for every $u_0\in D(\rA^{\frac{\alpha+1}{2}})$ and $f\in L^2(0,T;D(\rA^{\frac{\alpha}{2}})$ such that
\[ \vert u_0\vert_{D(\rA^{\frac{\alpha+1}{2}})} \leq R, \ \ \ \vert f\vert_{L^2(0,T_\ast;D(\rA^{\frac{\alpha}{2}})} \leq \rho\]
there exists a unique solution $v$ to  problem \eqref{eqn_NSE01} which satisfy  conditions \eqref{eqn-u-better regularity} on the time interval $[0,T_\ast]$.  Since $D(\rA^{\frac{\alpha+1}{2}})  \subset \rV$ and $L^2(0,T;D(\rA^{\frac{\alpha}{2}}) \subset L^2(0,T;\rH)$ with the embeddings being continuous, $u_0\in \rV$ and $f\in L^2(0,T;\rH)$. Therefore {by Theorems 3.1 and 3.2 in chapter III of \cite{Temam_2001},} there exists a unique solution $u$ to  problem \eqref{eqn_NSE01} on the whole real half-line $[0,\infty)$ which satisfies \eqref{ineq-L^infty(0,T,V)} and \eqref{ineq-L^2(0,T,H^2)}. By the uniqueness {part of the above cited results,} $u=v$ on $[0,T_\ast]$.

Hence it is sufficient to show that the  norm of $u$ in $L^2(0,T_\ast;D(\rA^{1+\frac{\alpha}{2}}) \cap C([0,T_\ast];D(\rA^{\frac{\alpha+1}{2}}))$ and the norm of $u^\prime$ in $L^2(0,T_\ast;D(\rA^{\frac{\alpha}{2}}))$ are bounded by a constant depending only on  $\vert \rA^{\frac{\alpha+1}2} u_0\vert_{\rH}^2$ and  $\int_0^T \vert \rA^{\frac{\alpha}{2}} f(s)\vert_{\rH}^2\, ds$.

{
For this aim,  by calculating the derivative of $\vert \rA^{\frac{\alpha+1}2} u(t)\vert_{\rH}^2 $, i.e. applying Lemma III.1.2 from \cite{Temam_2001} and using inequality\footnote{Since $ \vert \lb \rA^{1+\alpha}u,Bu\rb \vert= \vert \lb \rA^{1+\frac{\alpha}2}u,\rA^{\frac{\alpha}2}Bu\rb \vert \leq  \vert  \rA^{1+\frac{\alpha}2}u \vert \vert \rA^{\frac{\alpha}2}Bu \vert \leq  C \vert  \rA^{1+\frac{\alpha}2} u \vert \vert \rA u \vert   \vert \rA^{\frac{1+\alpha}2}u \vert  \leq \frac14 \vert  \rA^{1+\frac{\alpha}2} u \vert^2+ C^2 \vert \rA u \vert^2   \vert \rA^{\frac{1+\alpha}2}u \vert^2    $. }
\eqref{ineqn-B-fractional}, with $s=2$, we get the following inequality
\begin{equation}
\label{ineq-aux-01}
\begin{array}{l}
\hspace{-1.1cm}\ds{\frac 12 \frac d{dt}\vert \rA^{\frac{\alpha+1}2} u(t)\vert_{\rH}^2+\vert \rA^{\frac{\alpha}2+1} u(t)\vert_{\rH}^2}\\
\vspace{.1cm}\\
\ds{ \leq
       \vert \rA^{\frac{\alpha}{2}} f(t)\vert_{\rH}^2
+C^2 \vert \rA u(t) \vert_{\rH}^2 \vert \rA^{\frac{\alpha+1}2} u(t) \vert_{\rH}^2, \;\; t\in [0,T_\ast].}
\end{array}
   \nonumber
\end{equation}
Thus, denoting  the right hand side of inequality \eqref{ineq-L^2(0,T,H^2)} by $K_1(T,f,u_0)$ and applying the
 Gronwall Lemma we get
\begin{equation}
\label{brcfine}
\begin{array}{l}
\hspace{-1.1cm}\ds{ \vert \rA^{\frac{\alpha+1}2} u(t)\vert_{\rH}^2+\int_0^t\vert \rA^{\frac{\alpha}2+1} u(s)\vert_{\rH}^2\,ds} \\
\vspace{.1cm}\\
\ds{\leq e^{C^2K_1(T,f,u_0)}  \vert \rA^{\frac{\alpha+1}2} u_0\vert_{\rH}^2+ e^{C^2K_1(T,f,u_0)}\int_0^t \vert \rA^{\frac{\alpha}{2}} f(s)\vert_{\rH}^2\, ds}\\
\vspace{.1cm}\\
\ds{
\leq e^{C^2K_1(T,f,u_0)} \big( \vert \rA^{\frac{\alpha+1}2} u_0\vert_{\rH}^2+ \int_0^T \vert \rA^{\frac{\alpha}{2}} f(s)\vert_{\rH}^2\, ds\big)
 , \;\; t\in [0,T_\ast].}
 \end{array}
\end{equation}
}

{
This proves  that the  norm of $u$ in $L^2(0,T_\ast;D(\rA^{1+\frac{\alpha}{2}}) \cap C([0,T_\ast];D(\rA^{\frac{\alpha+1}{2}}))$ is bounded by a constant depending only on  $\vert \rA^{\frac{\alpha+1}2} u_0\vert_{\rH}^2$ and  $\int_0^T \vert \rA^{\frac{\alpha}{2}} f(s)\vert_{\rH}^2\, ds$.
}

{
Finally, the corresponding bound for the norm of $u^\prime$ in $L^2(0,T_\ast;D(\rA^{\frac{\alpha}{2}}))$ follows from  estimate \eqref{ineq-aux-02}, inequality \eqref{ineqn-B-fractional} from Proposition \ref{prop-Leray-fractional-alpha}, the assumption on  $f$ and  the estimates  for $\rA u$ and $\rB(u,u)$.

This concludes the proof of the first part of Proposition \ref{prop-NSE-fractional}, in particular of \eqref{eqn-u-better regularity}.

Let us now assume that $T =\infty$, $R,\rho>0$ and  $f\in L^2(0,\infty;D(\rA^{\frac{\alpha}{2}}))$ and $u_0\in D(\rA^{\frac{\alpha+1}{2}})$ such that $| u_0|_V \leq R$ and $\vert f\vert_{L^2(0,\infty;\rH)} \leq \rho$. Since, by the Poincar\'{e} inequality \eqref{ineq-Poincare}, $\vert f\vert^2_{\rV^\prime} \leq \lambda_1^{-1} \vert f \vert^2_{\rH}$, for $f\in \rH$, we infer that
$K_1(T,f,u_0) \leq K_3(R,\rho)$, where $K_3(R,\rho)$ has been defined in \eqref{ineq-aux-05}.
Thus, from inequality \eqref{brcfine} we infer
\[
 \vert \rA^{\frac{\alpha+1}2} u(t)\vert_{\rH}^2+\int_0^t\vert \rA^{\frac{\alpha}2+1} u(s)\vert_{\rH}^2\,ds
 \leq e^{C^2\,K_3(R,\rho)} \big( \vert \rA^{\frac{\alpha+1}2} u_0\vert_{\rH}^2+ \int_0^t \vert \rA^{\frac{\alpha}{2}} f(s)\vert_{\rH}^2\, ds\big)
 , \;\; t \geq 0.
\]
Since by the Poincar\'{e} inequality \eqref{ineq-Poincare} $  \vert \rA^{\frac{\alpha}2+1} \vert_{\rH}^2 \geq \lambda_1 \vert \rA^{\frac{\alpha+1}2} u\vert_{\rH}^2$, for $u\in D(\rA^{\frac{\alpha}2+1})$, by the inequality above and the  Gronwall Lemma we infer that
\begin{eqnarray}\label{ineq-aux-03bis}
 \vert \rA^{\frac{\alpha+1}2} u(t)\vert_{\rH}^2
 &\leq& e^{-\lambda_1 t} e^{C^2\,K(R,\rho)} \big( \vert \rA^{\frac{\alpha+1}2} u_0\vert_{\rH}^2+ \int_0^t \vert \rA^{\frac{\alpha}{2}} f(s)\vert_{\rH}^2\, ds\big)
 , \;\; t \geq 0.
\end{eqnarray}
This concludes the proof of inequality \eqref{ineq-aux-04} and hence of the second part of the Proposition.
}
\end{proof}

{The previous result can be used to derive the next corollary. The proof of this corollary is analogous to the proof of properties (3.151) in Theorem 3.10 from \cite[chapter III]{Temam_2001}.
\begin{corollary}\label{cor-improvement}
Assume that $\alpha  \in (0, \frac12)$ and $\beta\in (\alpha,1)$. If {$T \in (0,\infty)$,}  $f\in L^2(0,T;D(\rA^{\frac{\alpha}{2}}))$ and $u_0\in \rV$,  then the unique solution $u$  to the problem \eqref{eqn_NSE01} satisfy
\begin{eqnarray}\label{ineq-aux-09}
\nonumber
\sup_{s \in (0,T]}  s^\beta \vert \rA^{\frac{\alpha+1}2} u(s)\vert_{\rH}^2+\int_0^T s^\beta \vert \rA^{\frac{\alpha}2+1} u(s)\vert_{\rH}^2\,ds &\leq&
e^{C^2K_1(T,f,u_0)} \big[  \int_0^T s^\beta \vert \rA^{\frac{\alpha}{2}} f(s)\vert_{\rH}^2\, ds \\
&&\hspace{-3truecm}\lefteqn{ +  C T^{\frac1{\beta-\alpha}}  ( K_0(T,f,u_0)+ K_1(T,f,u_0))  \big].}
\end{eqnarray}
\end{corollary}
\begin{proof} Let us fix $\alpha \in (1,\frac12)$, $\beta\in (\alpha,1)$,   $f\in L^2(0,T;D(\rA^{\frac{\alpha}{2}}))$ and $u_0\in D(\rA^{\frac{\alpha+1}2})$.
Multiplying the differential inequality  \eqref{ineq-aux-01} by $t^\beta$ and then integrating it, we get

\begin{eqnarray}\label{ineq-aux-06}
 t^\beta\vert \rA^{\frac{\alpha+1}2} u(t)\vert_{\rH}^2+\int_0^t\, s^\beta\vert \rA^{\frac{\alpha}2+1} u(s)\vert_{\rH}^2\,ds
 &\leq&
   \int_0^t s^\beta \vert \rA^{\frac{\alpha}{2}} f(s)\vert_{\rH}^2\, ds \\
   +\,C^2  \int_0^t  \vert \rA u(s) \vert_{\rH}^2 \, s^\beta\,\vert \rA^{\frac{\alpha+1}2} u(s) \vert_{\rH}^2\, ds
   &+& \beta \int_0^t  s^{\beta-1}\vert \rA^{\frac{\alpha+1}2} u(s)\vert_{\rH}^2 \,ds, \;\; t\in (0,T].
\nonumber
\end{eqnarray}
Since $A$ is a self-adjoint operator in $\rH$ we have
\[
\vert \rA^{\frac{\alpha+1}2} u\vert_{\rH}^2
 \leq \vert \rA^{\frac{1}2} u\vert_{\rH}^{2(1-\alpha)} \vert \rA^{\frac{2}2} u\vert_{\rH}^{2\alpha}=| u |_V ^{2(1-\alpha)} \vert \rA u\vert_{\rH}^{2\alpha}
 , \;\; u\in D(\rA).
\]
Therefore, by the H\"older inequality,
\begin{equation}\label{ineq-aux-07}
\begin{array}{l}
\ds{
\int_0^t  s^{\beta-1}\vert \rA^{\frac{\alpha+1}2} u(s)\vert_{\rH}^2 \,ds \leq
\left(\sup_{s\in (0,t]} | u(s)|_V^2\int_0^t s^{\frac{\beta-1}{1-\alpha}}\, ds \right)^{1-\alpha} \left(\int_0^t \vert \rA u(s)\vert^2\, ds\right)^{\alpha} .}
\end{array}
\end{equation}
Since we are assuming $\beta>\alpha$ and $\alpha<1$, we infer that $\frac{\beta-1}{1-\alpha}>-1$ and therefore
\[
\int_0^t s^{\frac{\beta-1}{1-\alpha}}\, ds= \frac{1-\alpha}{\beta-\alpha}\;t^{\frac{1-\alpha}{\beta-\alpha}}<\infty.
\]
Let us recall that by  $K_1(T,f,u_0)$ we denote the right hand side of inequality \eqref{ineq-L^2(0,T,H^2)}. Let us also denote by $K_0(T,f,u_0)$ the right hand side of inequality \eqref{ineq-L^infty(0,T,V)}. Then, for a constant $C$ depending only on $\alpha$ and $\beta$ we  get
\[
\beta \int_0^T  s^{\beta-1}\vert \rA^{\frac{\alpha+1}2} u(s)\vert_{\rH}^2 \,ds \leq C T^{\frac1{\beta-\alpha}}  K_0(T,f,u_0)^{1-\alpha} K_1(T,f,u_0)^{\alpha}<\infty.
\]
Therefore,  we can deduce from inequality
\eqref{ineq-aux-06} the following one

\begin{eqnarray}\label{ineq-aux-08}
 t^\beta\vert \rA^{\frac{\alpha+1}2} u(t)\vert_{\rH}^2+\int_0^ts^\beta \vert  \rA^{\frac{\alpha}2+1} u(s)\vert_{\rH}^2\,ds &\leq& e^{C^2K_1(T,f,u_0)}\Big( \int_0^t s^\beta \vert \rA^{\frac{\alpha}{2}} f(s)\vert_{\rH}^2\, ds\\
 &+& C t^{\frac1{\beta-\alpha}}  K_0(t,f,u_0)^{1-\alpha} K(t,f,u_0)^{\alpha}\Big)
 \nonumber\\
 \leq e^{C^2K_1(T,f,u_0)} \big(  \int_0^T s^\beta \vert \rA^{\frac{\alpha}{2}} f(s)\vert_{\rH}^2\, ds &+& C T^{\frac1{\beta-\alpha}}  ( K_0(T,f,u_0)+ K_1(T,f,u_0))  \big)
 , \;\; t\in [0,T].
\nonumber
\end{eqnarray}
This implies inequality \eqref{ineq-aux-09} under the additional assumption that  $u_0\in D(\rA^{\frac{\alpha+1}2})$.

Now, if $u_0\in \rV$, then by \cite[Theorem III.3.10]{Temam_2001}, see also inequality \eqref{ineq-L^2(0,T,H^2)}, there exists a sequence $\{t_n\}$ such $t_n \todown 0$ and $u(t_n)\in D(\rA)\subset D(\rA^{\frac{\alpha+1}2})$. Let us also denote $f_n=f_{|[t_n,T]}$ and observe that
$K_i(T-t_n,f_n,u(t_n))\leq K_i(T,f,u(t_n))$. Thus, by  applying inequality \eqref{ineq-aux-09} to our solution $u$ on the time interval $[t_n,T]$  we get  for each $n\in\mathbb{N}$
\begin{eqnarray}\label{ineq-aux-10}
\sup_{s \in (t_n,T]}  (s-t_n)^\beta \vert \rA^{\frac{\alpha+1}2} u(s)\vert_{\rH}^2+\int_{t_n}^T (s-t_n)^\beta \vert \rA^{\frac{\alpha}2+1} u(s)\vert_{\rH}^2\,ds &\leq&
e^{K_1(T,f,u(t_n))}  \\
&&\hspace{-11.0truecm}\lefteqn{  \times \Big(\int_{t_n}^T (s-t_n)^\beta \vert \rA^{\frac{\alpha}{2}} f(s)\vert_{\rH}^2\, ds  + C (T-(t_n))^{\frac1{\beta-\alpha}}  ( K_0(T,f,u(t_n))+ K_1(T,f,u(t_n)))  \Big).}
\nonumber
\end{eqnarray}
Since $u\in C([0,T];\rV)$, we infer that $| u(t_n)|_V \to | u_0|_V$ and thus $K_i(T,f,u(t_n)) \to K_i(T,f,u_0)$, $i=1,2$. Moreover, by the Lebesgue Monotone Convergence Theorem,
\[
\int_{t_n}^T (s-t_n)^\beta \vert \rA^{\frac{\alpha}2+1} u(s)\vert_{\rH}^2\,ds
=\int_0^T 1_{(t_n,T]}(s)(s-t_n)^\beta \vert \rA^{\frac{\alpha}2+1} u(s)\vert_{\rH}^2\,ds
\to \int_0^T s^\beta \vert \rA^{\frac{\alpha}2+1} u(s)\vert_{\rH}^2\,ds ,
\]
\[
\int_{t_n}^T (s-t_n)^\beta \vert \rA^{\frac{\alpha}{2}} f(s)\vert_{\rH}^2\,ds
=\int_0^T 1_{(t_n,T]}(s)(s-t_n)^\beta \vert \rA^{\frac{\alpha}{2}} f(s)\vert_{\rH}^2\,ds
\to \int_0^T s^\beta \vert \rA^{\frac{\alpha}{2}} f(s)\vert_{\rH}^2\,ds.
\]
Hence, from \eqref{ineq-aux-10} we deduce \eqref{ineq-aux-09}. The proof is complete.
\end{proof}
}

\medskip

Now, for any $-\infty\leq a<b\leq \infty$ such that $a<b$ and for any two reflexive Banach spaces $X$ and $Y$ such that $X \embed Y$ continuously, we denote by\footnote{Some authours, for instance Vishik and Fursikov, use the notation $\mathcal{H}^{1,2}N(a,b;X,Y)$. Our choice is motivated by the notation used in the monograph \cite{Lions+Magenes_1972}, who however use notation $W(a,b)$.  } $W^{1,2}(a,b;X,Y)$ the space of all $u\in L^2(a,b;X)$ which are weakly differentiable as  $Y$-valued functions and their weak derivative
belongs to $L^2(a,b;Y)$. The space $W^{1,2}(a,b;X,Y)$  is a separable Banach space (and Hilbert if both $X$ and $Y$ are Hilbert spaces), with the natural norm
$$\vert u\vert_{W^{1,2}(a,b;X,Y)}^2= \vert u\vert^2_{ L^2(a,b;X)}+\vert u^\prime\vert^2_{ L^2(a,b;Y)},\;\; u\in W^{1,2}(a,b;X,Y).
$$
Later on, we will use  the shortcut notation
\[W^{1,2}(a,b)=W^{1,2}(a,b;D(\rA),\rH).\]

\medskip

We  conclude  this section with the statement of a couple of results which are obvious adaptations of deep results from \cite{Lions+M_2001} to the $2$-dimensional case. To this purpose, there is no need to mention that all what we have said about equation \eqref{eqn_NSE01} in the time interval $[0,T]$ applies to any time interval $[a,b]$, with $-\infty<a<b<\infty$.

\begin{definition}
\label{def-very weak}
Assume that $-\infty \leq a < b\leq \infty$ and $f\in L^2_{\textrm{loc}}((a,b);\rH)$. A function $u\in C((a,b);\rH)$ is called a {\em very weak solution} to the Navier-Stokes equations \eqref{eqn_NSE01} on the interval $(a,b)$ if for all $\phi \in C^\infty ((a,b)\times D)$, such that $\text{div} \phi =0$ on $(a,b)\times  D$ and $\phi=0$ on $(a,b)\times \partial D$,
\begin{equation}
\begin{array}{l}
\ds{
\int_D u(t_1,x) \phi(t_1,x)\, dx}\\
\vs
\ds{= \int_D u(t_0,x) \phi(t_0,x)\, dx+\int_{[t_0,t_1]\times D} u(s,x)\cdot (\partial_s\phi(s,x) + \nu \Delta \phi(s,x))\, dsdx }\\
\vs
\ds{
+\int_{t_0}^{t_1} b(u(s),u(s),\phi(s))\, ds + \int_{[t_0,t_1]\times D}  f(s,x) \cdot \phi(s,x) \, dsdx,}
\end{array}
\label{eqn-very weak}
\end{equation}
 for all $a<t_0<t_1<b$.

\end{definition}

\begin{proposition}\label{prop-very weak}
Assume that $-\infty \leq a < b\leq \infty$ and $f\in L^2_{\textrm{loc}}((a,b);\rH)$. Suppose that the functions $u, v\ \in\, C((a,b);\rH)$ are very weak solutions to the Navier-Stokes equations \eqref{eqn_NSE01} on the interval $(a,b)$, with   $u(t_0)=v(t_0)$,  for some $t_0\in (a,b)$. Then $u(t)=v(t)$ for all $t\geq t_0$.
\end{proposition}

\medskip

In the whole paper we will assume, without any loss of generality, that $\nu=1$.

\medskip

\begin{definition}
\label{def-action}
Assume that $-\infty \leq a < b\leq \infty$. Given a  function $u\in C((a,b);\rH)$ we say that
\[u^\prime+ \rA u +\rB(u,u)\ \in\, L^2(a,b;\rH),\ \ \  (\text{resp.}\  \in\, L^2_{\textrm{loc}}((a,b);\rH))\]
 if there exists $f\in L^2(a,b;\rH)$, (resp. $f \in L^2_{\textrm{loc}}((a,b);\rH)$) such that $u$ is a very weak solution of the Navier-Stokes equations \eqref{eqn_NSE01} on the interval $(a,b)$.

Clearly, the corresponding function $f$ is unique and we will denote it by $\mathcal{H}(u)$, i.e.
\begin{eqnarray}
\label{eqn-A3}
[\cH(u)]&:=&\dela{\big[G(u)\big]^{-1}}u^\prime+\rA u+\rB(u,u).
\end{eqnarray}
\end{definition}

An obvious sufficient condition for the finiteness of  the norm of $\mathcal{H}(u)$ in $L^2(t_0,t_1;H)$ is that $u^\prime$, $Au$  and $B(u,u)$ all belong to $L^2(t_0,t_1;\rH)$. The next result shows that this is not so far from a necessary condition. This is the reason why we have decided present  the following  result based  \cite{Lions+M_2001}, see also  \cite{Furioli+LR+T_2000}.

\begin{lemma}\label{lem-clar}
Suppose that $T>0$ and $u\in C([0,T];\rH)$ is such that
$$u^\prime+\rA u+\rB(u,u)\in L^2(0,T;\rH).$$
Then $u(T)\in \rV$,  $u\in W^{1,2}(t_1,T)$, for any $t_1\in (0,T)$ and
\begin{equation*}
\label{brc82+}
\sup_{t \in (0,T]} | \sqrt{t}\,u (t) |_V^2+ \int_0^T \vert \sqrt{t}\,\rA u (t) \vert_{\rH}^2\, dt <\infty.
\end{equation*}
 Moreover, if $u(0)\in \rV$, then $u\in W^{1,2}(0,T)$.
\end{lemma}

\begin{proof} Let us fix $T>0$ and $u$ as in the assumptions of the Lemma and let  us denote $f=u^\prime+\rA u+\rB(u,u)$. By assumptions we infer that $f\in L^2(0,T;\rH)$. Since $u(0)\in \rH$, by  \cite[Theorem III.3.10]{Temam_2001}, there exists a unique solution $v$ to problem \eqref{eqn_NSE01} which satisfies inequality \eqref{brc82}. Then by  $v$ is also a mild solution to \eqref{eqn_NSE01} and since by assumptions and \cite[Proposition 2.5]{Lions+M_2001} $u\in C([0,T];\rH)$ is also a mild solution to \eqref{eqn_NSE01}, by Proposition \ref{prop-very weak} (i.e. \cite[Theorem 1.2]{Lions+M_2001}) we infer that $u=v$. Hence $u\in L^2(0,T;\rV)$, $u^\prime\in L^2(0,T;\rV^\prime)$ and $u$ satisfies \eqref{brc82}. In particular,   for every $t_1\in (0,T)$  we can find $t_0\in (0,t_1)$ such that $u(t_0)\in \rV$ and therefore
by  \cite[Theorem III.3.10]{Temam_2001},
$u\in W^{1,2}(t_0,T)$. In particular,  $u \in C([t_0,T];\rV)$ and   hence    $u(T)\in \rV$.\\
If the additional assumption that  $u(0)\in \rV$ is satisfied, then  by what we have just seen (\cite[Theorem III.3.10]{Temam_2001}) or by the maximal regularity and  the uniqueness of solutions to 2D NSEs), we can conclude that $u\in W^{1,2}(0,T)$.\\

\end{proof}

\begin{rem}\delh{ {\em It should be pointed out that even in the linear case, i.e. when $B=0$, the converse to the last statement of the above Lemma \ref{lem-clar} is not true. Moreover, \delg{even in the linear case,} a stronger version of the first statement, i.e. that $u\in W^{1,2}(0,T)$  is not true {in general}. Indeed, if $u\in W^{1,2}(0,T)$ then $u\in C([0,T];\rV)$. See however Proposition \ref{prop-infty} where a sort of the converse holds true  but on an unbounded interval $(-\infty,0]$.}}
{\em It should be pointed out that we cannot claim that  $u(0)\in \rV$. Indeed, if we take  a solution $u$ of the problem \eqref{eqn_NSE01} with data $u_0\in \rH\setminus\rV$ and $f=0$, then $u$ satisfies the Assumptions of Lemma \ref{lem-clar} but nethervelles $u(0)\notin \rV$. See however Proposition \ref{prop-infty} for a positive result on an unbounded interval $(-\infty,0]$.
}
\end{rem}

A {result} analogous of Lemma \ref{lem-clar} holds in domains of fractional powers of $\rA$.
\begin{lemma}\label{lem-clar-2}
Assume that  \dela{$\delta \in (0,1]$ and} $\alpha\in [0,1/2)$ and suppose that  $u\in C([0,T];\rH)$, for some $T>0$, is such that
$$u^\prime+\rA u+\rB(u,u)\in L^2(0,T;D(\rA^{\frac{\alpha}2})).$$
Then $u(T)\in D(\rA^{\frac{\alpha+1}2})$ and $ u\in W^{1,2}\big(t_0,T;D(\rA^{\frac{\alpha}2+1}),D(\rA^{\frac{\alpha}2})\big)$,
for any $t_0\in (0,T)$.
Moreover, if $u(0)\in D(\rA^{\frac{\alpha+1}2})$, then $u\in W^{1,2}\big(0,T;D(\rA^{\frac{\alpha}2+1}),D(\rA^{\frac{\alpha}2})\big)$.
\end{lemma}

\begin{proof}
Denote $f=u^\prime+\rA u+\rB(u,u)$ and let us fix $t_0\in (0,T)$ and some $t_1\in (0,t_0)$.  By Lemma \ref{lem-clar} we infer that $u\in W^{1,2}(t_1,T)$. In particular, there exists $t_2\in (t_1,t_0)$ such that $u(t_2)\in D(\rA) \subset D(\rA^{\frac{\alpha+1}2})$. The last embedding holds since  by assumptions $\alpha <\frac12$. Since by  our assumption, $f\in L^2(0,T;D(\rA^{\frac{\alpha}2}))$, in view of  Proposition \ref{prop-NSE-fractional}
and   Proposition \ref{prop-very weak}, we infer that $u\in W^{1,2}\big(t_2,T;D(\rA^{\frac{\alpha}2+1}),D(\rA^{\frac{\alpha}2})\big)$. This implies that $u \in C([t_2,T];D(\rA^{\frac{\alpha+1}{2}}))$  and in particular that  $u(T)\in D(\rA^{\frac{\alpha+1}2})$ and
$ u\in W^{1,2}\big(t_0,T;D(\rA^{\frac{\alpha}2+1}),D(\rA^{\frac{\alpha}2})\big)$ as in our first claim. The second claim follows from our last argument. \end{proof}

\medskip

In what follows, for any $r>0$ and  $\gamma\geq 0$ we shall denote {by $B_\gamma(r)$ the closed ball in $D(\rA^{\frac\gamma 2})$ of radius $r$ and centered at the origin, i.e. }
\[{B_\gamma(r):=\left\{ x \in\,D(\rA^{\frac\gamma 2})\ :\ |x|_{D(\rA^{\frac\gamma 2})}\leq r\right\}}.\]
Moreover, for any $\phi \in\,\rH$ and $s\in \mathbb{R}$,  we shall denote
 by {$u_\phi(t;s)$, $t\geq s$ (simply  $u_\phi(t)$, $t\geq 0$, when $s=0$)}  the solution of  problem {\eqref{eqn_NSE01} with the external force $f$ equal to $0$, i.e. }
\begin{equation}
\label{brc80}
\left\{\begin{array}{l}
u^\prime(t)+\rA u(t)+\rB(u(t),u(t))=0,\ \ \ t>s,\\
\vspace{.1cm}\\
u(s)=\phi.
\end{array}\right.
\end{equation}

Moreover, for  ${\phi} \in\,\rH$, $r>0$ and ${\gamma}\geq 0$ we shall denote
\[t_\phi^{r,\gamma}:=\inf\left\{t\geq 0\ :\ {u_\phi(t)} \in\,B_{\gamma}(r)\right\}.\]

\begin{proposition}
\label{propbrc3}
For any $c_1, c_2>0$ and $\sigma\in [0,\frac32)$, there exists $T=T({\sigma}, c_1, c_2)>0$ such that
for every $\phi\in \rH$ such $\vert \phi\vert_{\rH}\leq c_1$, one has
\begin{equation}
\label{eqn-absorption}
\vert A^{\frac\sigma{2}}u_\phi(t)\vert_{\rH}\leq c_2, \mbox{ for all } t\geq T.
\end{equation}
\end{proposition}

\begin{proof} By inequality \eqref{brc82}, for any $\phi \in\,\rH$
\begin{equation}
\label{eqn-absorption+01}
| u_\phi (1) |_V^2 \leq \vert \phi \vert_{\rH}^2 e^{\,C \vert \phi\vert_{\rH}^4}.
\end{equation}
Then, by  inequality \eqref{ineq-u-V4} in Corollary \ref{cor-u-nfty-V}, we infer
\[ | u_\phi(t) |_V^2 \leq   e^{54 \vert \phi\vert_{\rH}^4 }e^{-\lambda_1 (t-1)} \, | u_\phi(1)|_V^2 ,\;\;  t \geq 1.
\]
Combining these two we get
\[ | u_\phi(t) |_V^2 \leq   e^{(54+C) \vert \phi\vert_{\rH}^4 + \lambda_1}e^{-\lambda_1 t} \vert \phi \vert_{\rH}^2  ,\;\;  t \geq 1,
\]
and \eqref{eqn-absorption} follows for $\sigma \leq 1$.

Consider now the case $\sigma \in (1,\frac32)$.
Let us fix constants $c_1, c_2>0$ and an initial data $\phi \in \rH$ such that $\vert \phi\vert_{\rH}\leq c_1$.  We will be applying the previous step  with $\alpha=\sigma-1\in (0,\frac12)$. Choose an auxiliary $\beta\in (\sigma-1,1)$. Then by inequality \eqref{ineq-aux-09} in Corollary \ref{cor-improvement},   we get
\begin{eqnarray*}
 \vert \rA^{\frac{\alpha+1}2} u_\phi(2)\vert^2 &\leq&  e^{C^2K_1(2,0,u_\phi(1))}C   ( K_0(2,0,u_\phi(1))+ K_1(2,0,u_\phi(1))).
 \end{eqnarray*}
 Therefore, recalling how $K_0$ and $K_1$ were defined, due to  \eqref{eqn-absorption+01}
 \[|\rA^{\frac{1+\alpha}2}u_\phi(2)|^2_{\rH}\leq K_4(|\phi|_{\rH}),\]
 for some continuous, increasing function $K_4$. According to \eqref{ineq-aux-04} this allows to conclude.
\end{proof}

\section{Some basic facts on relaxation and $\Gamma$-convergence}
\label{sec-gamma}

Let us assume that $X$ is  a topological space satisfying the first axiom of countability, i.e. every point in $X$   has a countable local base.
For any $x\in\,X$, we shall denote by $\mathcal N(x)$ the set of all open neighborhoods of $x$ in $X$.

\begin{definition}
 \label{def1}
Let $F:X\to \overline{\mathbb{R}}$ be a  function.
\begin{enumerate}
 \item The function  $F$ is called  {\em lower semi-continuous} if  for any $t \in\,\mathbb{R}$, the  inverse image set $F^{-1}((-\infty])=\{x\in X: F(x)\leq t\}$ is closed in $X$,

\item  The function  $F$ is called {\em coercive} if   for any $t \in\,\mathbb{R}$, the closure of the level set $\{x\in X: F(x)\leq t\}$ is countably compact, i.e. every countable open cover has a finite subcover.
\end{enumerate}

\end{definition}
Now, let $\{F_n\}_{\in\,\mathbb{N}}$ be a sequence of functions all defined on $X$ with values in $\overline{\mathbb{R}}$.

\begin{definition}
 \label{def2}
The sequence of functions $\{F_n\}_{\in\,\mathbb{N}}$ is called {\em equi-coercive} if for any $t \in\,\mathbb{R}$ there exists a closed countably compact set $K_t\subset X$ such that
\[ \bigcup_{n \in\,\mathbb{N}}\{x\in X:F_n(x)\leq t\}\subset K_t.\]
\end{definition}
Let us note that if $Y$ is a closed subspace of $X$, then the restrictions to $Y$ of  lower semi-continuous, coercive and equi-coercive functions, remain such on $Y$.

As proved in \cite[Proposition 7.7]{DalMaso_1993}, the following characterization of equi-coercive sequences holds.
\begin{proposition}
 \label{prop1}
The sequence $\{F_n\}_{\in\,\mathbb{N}}$ is equi-coercive  if and only if there exists a lower semi-continuous coercive function $\Psi:X\to \overline{\mathbb{R}}$ such that
\[F_n(x)\geq \Psi(x),\ \ \ \ x \in\,X,\ \ \ n \in\,\mathbb{N}.\]
\end{proposition}

Now, we introduce the notion of relaxation of a function $F$.

\begin{definition}
{\em The {\em lower semi-continuous envelope}, or the {\em relaxed
function}, of a function  $F:X\to \bar{\mathbb R}$ is defined by
\[(sc^-F)(x)=\sup\{G(x): G \in\,\mathcal G(F)\} ,\ \ \ \ x \in\,X,\]
where $\mathcal G(F)$ is the set of all lower semi-continuous
functions $G:X \to \bar{\mathbb R}$ such that  $G\leq F$.}
\end{definition}

From the definition, one has immediately that $sc^-F$ is lower
semi-continuous, $sc^-F\leq F$ and $sc^-F\geq G$, for any $G
\in\,\mathcal G(F)$, so that $sc^-F$ can be regarded as the
greatest lower semi-continuous function majorized by $F$.
Moreover, it is possible to prove that
\[(sc^-F)(x)=\sup_{U \in\,\mathcal N(x)}\,\inf_{y \in\,U}F(y),\ \ \ \ x \in\,X,\]
(see \cite[Proposition 3.3]{DalMaso_1993}).

The following result, whose proof can be found  in
\cite[Proposition 3.6]{DalMaso_1993}, provides a possible
characterization of $sc^- F$ which we will use later on in the
paper.

\begin{proposition}
\label{charac}
For any function  $F:X\to \bar{\mathbb R}$, its lower semi-continuous function $sc^-F$ is characterized by the following properties:
\begin{enumerate}
\item for any $x \in\,X$ and any sequence $\{x_n\}_{n \in\,\mathbb{N}}$ convergent to $x$
in $X$, it holds
\[(sc^-F)(x)\leq \liminf_{n\to \infty}F(x_n);\]
\item  for any $x \in\,X$ there exists a sequence $\{x_n\}_{n \in\,\mathbb{N}}$ convergent to $x$
in $X$ such that
\[(sc^-F)(x)\geq \limsup_{n\to \infty}F(x_n).\]
\end{enumerate}

\end{proposition}

Next, we introduce the notion of $\Gamma$-convergence for sequences of functions.
\begin{definition}
 \label{def3}
The $\Gamma$-{\em lower limit} and the $\Gamma$-{\em upper limit} of the sequence $\{F_n\}_{n \in\,\mathbb{N}}$ are the functions from $X$ into $\overline{\mathbb{R}}$ defined respectively by
\[\begin{array}{l}
 \displaystyle{\Gamma-\liminf_{n\to \infty} F_n(x)=\sup_{U \in\,\mathcal N(x)}\,\liminf_{n\to \infty} \inf_{y \in\,U}F_n(y),}\\
\vspace{.cm}
\displaystyle{\Gamma-\limsup_{n\to \infty} F_n(x)=\sup_{U \in\,\mathcal N(x)}\,\limsup_{n\to \infty} \inf_{y \in\,U}F_n(y).}
\end{array}\]
If there exists a function $F:X\to \overline{\mathbb{R}}$ such that $\Gamma-\liminf_{n\to \infty} F_n=\Gamma-\limsup_{n\to \infty} F_n=F$, then we write
\[F=\Gamma-\lim_{n\to \infty} F_n,\]
and we say that the sequence $\{F_n\}_{n \in\,\mathbb{N}}$ is $\Gamma$-convergent to $F$.

\end{definition}

In \cite[Proposition 5.7]{DalMaso_1993} we can find the proof of  the following result, which links $\Gamma$-convergence and relaxation of functions and provides a useful criterium for $\Gamma$-convergence.
\begin{proposition}
 \label{prop3}
If $\{F_n\}_{n \in\,\mathbb{N}}$ is a decreasing sequence converging to $F$ pointwise, then $\{F_n\}_{n \in\,\mathbb{N}}$ is $\Gamma$-convergent to $sc^-F$.
\end{proposition}

We conclude by giving a criterium for convergence of minima for $\Gamma$-convergent sequences (for a proof see \cite[Theorem 7.8]{DalMaso_1993}).
\begin{theorem}
 \label{teo4}
Suppose that the sequence $\{F_n\}_{n \in\,\mathbb{N}}$ is equi-coercive in $X$ and $\Gamma$-converges to a function $F$ in $X$. Then, $F$ is coercive and
\[\min_{x \in\,X}F(x)=\lim_{n\to \infty} \inf_{x \in\,X}F_n(x).\]
\end{theorem}

\section{The large deviation action functional}\label{sec-action}

For any fixed $\eps, \delta \in (0,1]$ and $\phi \in\,\rH$, we consider the problem
\begin{equation}
\label{eqn-SNSE-eps}
du(t)+\rA u(t)+\rB(u(t),u(t))=\sqrt{\eps}\,dw^{Q_\delta}(t),\ \ \ \ u(0)=\phi,
\end{equation}
where
\[w^{Q_\delta}(t)=\sum_{k=1}^\infty Q_\delta e_k \beta_k(t),\ \ \ \ t\geq 0,\]
 $\{e_k\}_{k \in\,\mathbb{N}}$ is the basis which diagonalizes the operator $\rA$, $\{\beta_k\}_{k \in \mathbb{N}}$ is a sequence of independent Brownian motions all defined on the stochastic basis $(\Omega, \mathcal {F}, \mathbb{F}, \mathbb{P})$, where $\mathbb{F}=(\mathcal {F}_t)_{t\ge 0}$,  and $Q_\delta$ is a   bounded linear operator on $\rH$, for any $\delta \in\,(0,1]$.

In what follows, we shall assume that the family $\{Q_\delta\}_{\delta \in\,(0,1]}$ satisfies the following conditions.

\begin{assumption}\label{ass-Q}
For every $\delta \in (0,1]$, $Q_\delta$ is a positive linear operator on $\rH$,  the operator
 $ A^{-1}Q_\delta^2$ is trace class, and there exists some $\beta >0$ such that $Q_\delta: \rH \to D(\rA^{\frac \beta 2})$  is an isomorhism. Moreover,
 \[\lim_{\delta\to 0} Q_\delta\, y= y,\ \ y \in\,\rH,\ \ \ \ \ \lim_{\delta\to 0} Q^{-1}_\delta y= y,\ \ y \in\,D(\rA^{\frac \beta 2}),\]
the limites above being in $\rH$, and for any $1 \geq \sigma \geq \delta\geq 0$
\begin{equation}
\label{eqn-B12}
\vert Q_\sigma^{-1}y\vert_{\rH} \geq  \vert Q_\delta^{-1}y\vert_{\rH}, \;\; y\in D(\rA^{\frac\beta2}).
\end{equation}
\end{assumption}
 \begin{rem}\label{rem-RKHS} {\em
The reproducing kernel Hilbert space of the Wiener process $w^{Q_\delta}$ is equal to $Q_\delta(\rH)$ and hence by Assumption \ref{ass-Q} it coincides with the space
  $D(\rA^\beta)$ for some $\beta>0$. This implies that  the results from \cite{Brz+Li_2006} are applicable.

}  \end{rem}
 \begin{rem} {\em It is easy  to see that for the Navier-Stokes equations in a $d$-dimensional domain, $d\geq 2$, any number $\beta > \frac{d}2-1$ and the operators
 \[Q_\delta:=\big(I+\delta \rA^{\beta/2}\big)^{-1}\]
  satisfy  Assumption \ref{ass-Q}.}
 \end{rem}

\medskip

Now, for any $-\infty \leq t_0< t_1\leq \infty$, $\delta\in [0,1]$ and $u \in\,C([t_0,t_1];H)$, we define

\begin{equation}
\label{eqn-BA}
S^\delta_{t_0,t_1}(u):=\frac12 \int_{t_0}^{t_1}\vert Q_\delta^{-1}\left( \cH(u)(t)\right)\vert_{\rH}^2\, dt,
\end{equation}
where $\mathcal{H}(u)$ is defined as in \eqref{eqn-A3}, with the usual convention that $S^\delta_{t_0,t_1}(u)=+\infty$, if
$Q_\delta^{-1}\left( \cH(u)(\cdot)\right) \not\in L^2(t_1,t_2;\rH)$.

When $\delta=0$, the superscript $0$ will be omitted. So we put $S_{t_0,t_1}=S^0_{t_0,t_1}$. Note that according to Lemma \ref{lem-clar-2} a necessary condition for $S^\delta_{t_0,t_1}(u)$ to be  finite is that  $u(T_1)\in D(\rA^{\frac{\alpha+1}2})$ and $ u\in W^{1,2}\big(t_2,t_1;D(\rA^{\frac{\alpha}2+1}),D(\rA^{\frac{\alpha}2})\big)$,
for any $t_2\in (t_0,t_1)$.

\medskip

For any $T>0$, $p\geq 1$, $\eps, \delta \in\,(0,1]$ and $\phi \in\,\rH$, equation \eqref{eqn-SNSE-eps} admits a unique solution $u_\phi^{\eps,\delta} \in L^p(\Omega;C([0,T];\rH))$; for a proof see e.g. the fundamental work of Flandoli \cite{flandoli}.  To be more precise let us now formulate a definition of a solution following \cite{Brz+Li_2006}.

\begin{definition}\label{def-SNSEs}
If $u_0 \in \rH$, then  an $\mathbb{F}$-adapted   process $u(t)$, $t\ge 0$ with
trajectories in $C([0,\infty);\rH)\cap L^4_{\mathrm{loc}}([s,\infty);\mathbb{L}^4(D))$ is a solution to
problem \eqref{eqn-SNSE-eps} iff   for any $ v \in
\mathrm{V}$, $t \geq 0$, $\mathbb{P}$-almost surely, 
\begin{eqnarray*}\nonumber
(u(t),v)&=&(u_0,v)-\nu
\int_s^t(u(r),\mathrm{A}v)\,dr-\int_s^t b(u(r),u(r),v)\, dr \\
&  + & \sqrt{\eps}\;  ( v,  \,w^{Q_\delta}(t) ). \label{eqn:6.14}
\end{eqnarray*}
\end{definition}

As shown in the next theorem, as an immediate consequence of the contraction principle, we have  that the family $\left\{\mathcal{L}(u_\phi^{\eps,\delta})\right\}_{\eps \in\,(0,1]}$ satisfies a large deviation principle in $C([0,T];\rH)$.

\begin{theorem}
\label{brc.teo.4.2}
For any $x \in\,\rH$ and $\delta \in (0,1]$, the family
$\{{\mathcal L}(u_\phi^{\eps,\delta})\}_{\eps \in\,(0,1]}$ satisfies  a large deviation principle on $C([0,T];\rH)$, uniformly with respect to  initial data $\phi$ in  bounded sets of $\rH$, with   good action functional $S^\delta_{T}$.
\end{theorem}
\begin{proof}
For every $\eps>0$ { and $\delta\in (0,1]$}, we denote by $z_{\eps,\delta}(t)$ the Ornstein-Uhlenbeck process associated with $\rA$ and $Q_\delta$, that is the solution of the linear problem
\begin{equation}\label{eqn-OUP-eps} dz(t)+\rA z(t)=\sqrt{\eps}\,dw^{Q_\delta}(t),\ \ \ \ z(0)=0.
\end{equation}
We have
\[z_{\eps,\delta}(t)=\sqrt{\eps}\int_0^t e^{-(t-s)\rA}\, dw^{Q_\delta}(s),\ \ \ \ t\geq 0.\]
As well known (see e.g. \cite[Theorem 3]{zab}), under our assumptions the family $\{{\mathcal L}(z_{\eps,\delta})\}_{\eps \in\,(0,1]}$ satisfies a large deviation principle in $C([0,T];L^4(\mathcal{O}))$, with good action functional
\[I_{0,T}^\delta(u)=\frac 12\,\int_0^T\left|Q_\delta^{-1}(u^\prime(t)+\rA u(t))\right|_\rH^2\,dt.\]
Moreover, if we define the mapping ${\mathcal F}:\rH\times C([0,T];L^4(\mathcal{O}))\to C([0,T];\rH)$ which associates to every $\phi \in\,\rH$ and $g \in\,C([0,T];L^4(\mathcal{O}))$ the solution $v \in\,C([0,T];\rH)$ of the problem
\[v^\prime(t)+\rA v(t)+B(v(t)+g(t),v(t)+g(t))=0,\ \ \ \ v(0)=\phi,\]
we have, $\mathbb{P}$-a.s.,
\[u_\phi^{\eps,\delta}={\mathcal F}(\phi,z_{\eps,\delta}).\]
{Since, by} \cite[Theorem 4.6]{Brz+Li_2006},  the mapping ${\mathcal F}:\rH\times C([0,T];L^4(\mathcal{O}))\to C([0,T];\rH)$ is continuous, by the contraction principle, the large deviation principle for $\{z_{\eps,\delta}\}_{\eps \in\,(0,1]}$ on $C([0,T];L^4(\mathcal{O}))$ with action functional $I_{0,T}^\delta$ may be transferred to a large deviation principle for $\{u_\phi^{\eps,\delta}\}_{\eps \in\,(0,1]}$ on $C([0,T];\rH)$,  with action functional $S^\delta_{0,T}$.

Moreover, in \cite[Theorem 4.6]{Brz+Li_2006} it is shown that  for any $R>0$ there exists $c_R>0$ such that for any $z_1, z_2 \in\,B_R(C([0,T];L^4(\mathcal{O})))$
\[ \sup_{\phi \in\,B_0(R)}|{\mathcal F}(\phi,z_1)-{\mathcal F}(\phi,z_2)|_{C([0,T];\rH)}\leq c_R\,|z_1-z_2|_{C([0,T];L^4(\mathcal{O}))}.\]
This implies that the large deviation principle proved above is uniform with respect to the initial data $\phi$  in any bounded subset of $\rH$.

\end{proof}

In what follows, for any $T \in\,(0,+\infty]$ we set
\[S^\delta_{T}:=S^\delta_{0,T},\ \ \ \ \ S^\delta_{-T}:=S^\delta_{-T,0}\]
and
\[S_{T}:=S_{0,T},\ \ \ \ \ S_{-T}:=S_{-T,0}.\]
In particular,
\begin{equation}
\label{eqn-A5}
S^\delta_{-\infty}(u):=\frac12 \int_{-\infty}^0\vert Q_\delta^{-1}\left(\cH(u)(t)\right)\vert_{\rH}^2\, dt
\end{equation}
and
\begin{equation}
\label{eqn-A5bis}
S_{-\infty}(u):=\frac12 \int_{-\infty}^0\vert \cH(u)(t)\vert_{\rH}^2\, dt.
\end{equation}

We conclude the present section with the description of some relevant properties of { the functionals} $S_{-\infty}$ and $S^\delta_{-\infty}$.

To this purpose, we need to introduce the following  functional spaces
\begin{equation}
\label{eqn-B6}
\mathcal{X}=\big\{ u\in C((-\infty,0];\rH): \; \lim_{t\to-\infty}\vert u(t)\vert_{\rH}=0\big\},\ \
\mathcal{X}_\phi=\big\{ u\in \mathcal{X}:u(0)=\phi\big\}.
\end{equation}
 We endow the space  $\mathcal{X}$ with the topology of uniform convergence on compact intervals, i.e. the topology induced by the metric $\rho$ defined by
$$ \rho(u,v):= \sum_{n=1}^\infty 2^{-n}\left(\sup_{s\in [-n,0]} \vert u(s)-v(s)\vert_{\rH}\wedge 1\right),\;\;\; u,v\in \mathcal{X}.
$$
The set $\mathcal{X}_\phi$ is  closed in $\mathcal{X}$ and we endow it with the trace topology induced by $\mathcal{X}$.

Let us note here, see for instance \cite[Proposition I.3.1]{Vishik+Fursikov_1988},  that if  $u\in  W^{1,2}(t_0,\infty)$, then \[\lim_{t\to \infty} \vert u(t)\vert_\rV=0.\]
 Similarly, as shown in Proposition \ref{prop--infty2},  if  $u\in W^{1,2}(-\infty,t_1)$, then
 \[\lim_{t\to -\infty} \vert u(t)\vert_\rV=0.\]
\dela{\begin{eqnarray}
\label{eqn-B7}
\cHX_x&:=& \mathcal{X}_x \cap W^{1,2}(-\infty,0),
 W^{1,2}&:=&W^{1,2}(-\infty,0)
\end{eqnarray}}

\medskip

\begin{proposition}\label{Fact 3} \del{Assume that $x\in \rH$.
Then }\addf{T}he functionals $S_{-\infty}$ and $S^\delta_{-\infty}$, $\delta \in\,(0,1]$, are lower-semicontinuous in $\mathcal{X}$.
\end{proposition}
\begin{proof}
In order to prove the lower semi-continuity of $S_{-\infty}$ and $S^\delta_{-\infty}$, it is sufficient to show that if a $\mathcal{X}$-valued sequence  $\{u_n\}_{n=1}^\infty$ is convergent in $\mathcal{X}$ to a function $u \in \mathcal{X}$, then
 for any $\delta \in\,[0,1]$
\begin{equation}
\label{brc3}
\liminf_{n\to\infty} S^\delta_{-\infty}(u_n)\geq S^\delta_{-\infty}(u).
\end{equation}

First,  we assume that $u\in\mathcal{X}$ is such that  $S^\delta_{-\infty}(u)=\infty$. We want to show that
\[\liminf_{n\to\infty} S^\delta_{-\infty}(u_n)=+\infty.  \] Suppose by contradiction that $\liminf_{n}S^\delta_{-\infty}(u_n)<\infty$.  Then, after extracting a subsequence, we can find  $C>0$ such that
\[ \vert u^\prime_n+A u_n+B(u_n,u_n)\vert_{L^2(-\infty,0;D(\rA^{\frac \beta 2}))} \leq C, \;\; n\in \mathbb{N}.
\]

By Proposition \ref{prop-inftybis} (Proposition \ref{prop-infty}, if $\delta=0$), we have that the sequence
 $\{u_n\}$ is bounded in $W^{1,2}(-\infty,0;D(\rA^{1+\frac{\beta}2}),D(\rA^{\frac{\beta}2}))$ and hence we can find $\tilde{u} \in W^{1,2}(-\infty,0;D(\rA^{1+\frac{\beta}2}),D(\rA^{\frac{\beta}2}))$ such that, after another extraction of  a subsequence,
 \[u_n \to \tilde{u},\ \mbox{ as, } n\to \infty \ \ \mbox{weakly in}\ W^{1,2}(-\infty,0;D(\rA^{1+\frac{\beta}2}),D(\rA^{\frac{\beta}2})).\]
 By the uniqueness of the limit  we infer that $u=\tilde{u}$, so that $u\in  W^{1,2}(-\infty,0;D(\rA^{1+\frac{\beta}2}),D(\rA^{\frac{\beta}2}))$ and  $S^\delta_{-\infty}(u)<\infty$, which contradicts our assumption.

Thus,   assume that  $S^\delta_{-\infty}(u)<\infty$. In  view of the last part of Lemma \ref{lem-clar-2} (Lemma \ref{lem-clar} if $\delta=0$) we have that  $u(0)\in  D(\rA^{\frac{1+\beta}2})$ and $u \in\,W^{1,2}(-\infty,0;D(\rA^{1+\frac{\beta}2}),D(\rA^{\frac{\beta}2}))$.
  \del{as $\{u_n\}\subset \mathcal{X}_x$.}

Now, assume that \eqref{brc3} is not true. Then there exists $\eps>0$ such that, after the extraction of a subsequence,
\[S^\delta_{-\infty}(u_n)<S^\delta_{-\infty}(u)-\eps,\ \ \ \ n \in\,\mathbb{N}.\]

Hence, if we set $f_n=u_n^\prime+Au_n+B(u_n,u_n)$, we have that the sequence $\{f_n\}$ is bounded in $L^2(-\infty,0;D(\rA^{\frac \beta 2}))$ and then, by Proposition \ref{prop-inftybis} (Proposition \ref{prop-infty} if $\delta=0$), we have that  $\{u_n\}$ is bounded in $W^{1,2}(-\infty,0;D(\rA^{1+\frac{\beta}2}),D(\rA^{\frac{\beta}2}))$. This implies that we can find $\tilde{u} \in W^{1,2}(-\infty,0;D(\rA^{1+\frac{\beta}2}),D(\rA^{\frac{\beta}2}))$ such that, after a further extraction of  a subsequence,
\[u_n=\tilde{u},\ \mbox{ as } n\to \infty, \ \mbox{weakly in}\ \ W^{1,2}(-\infty,0;D(\rA^{1+\frac{\beta}2}),D(\rA^{\frac{\beta}2})),\] and, by uniqueness of the limit,  we infer that $u=\tilde{u}$. Moreover,
as  the sequence $\{f_n\}$ is bounded in $L^2(-\infty,0;D(\rA^{\frac{\beta}2}))$, after another   extraction of  a subsequence, we can find  $\tilde{f}\in L^2(-\infty,0;D(\rA^{\frac{\beta}2}))$ such that  $f_n$ converges weakly to  $\tilde{f}$  in $L^2(-\infty,0;D(\rA^{\frac{\beta}2}))$. By employing nowadays standard  compactness argument, see for instance \cite[section 5]{Brz+Li_2006} we can show that $\tilde{f}=f=u^\prime+\rA u+\rB(u,u)$.  Thus, since the mapping
\[f \in\,L^2(-\infty,0;D(\rA^{\frac{\beta}2}))\mapsto \int_{-\infty}^0 |Q_\delta^{-1}f(t)|^2_H\,dt \in\,\mathbb{R}\]
is convex and lower semi-continuous, it is also  weakly lower semi-continuous, so that
\[\liminf_{n\to \infty} S^\delta_{-\infty}(u_n)=\frac 12\liminf_{n\to \infty}|Q_\delta^{-1}f_n|^2_{L^2(-\infty,0;D(\rA^{\frac{\beta}2}))}\geq\frac 12 |Q_\delta^{-1}f|^2_{L^2(-\infty,0;D(\rA^{\frac{\beta}2}))}=S_{-\infty}^\delta(u).\]

\end{proof}

\begin{proposition}\label{cor-compact}
The operators $S_{-\infty}$ and $S^\delta_{-\infty}$ have  compact level sets in $\mathcal{X}$. Moreover, the family $\{S^\delta_{-\infty}\}_{\delta \in\,(0,1]}$ is equi-coercive.
\end{proposition}
\begin{proof} First, notice that we have only to prove the compactness of the level sets of $S_{-\infty}$. Actually,  due to Assumption \ref{ass-Q},
\begin{equation}
\label{brc500}
S_{-\infty}^\delta \geq S_{-\infty},\ \ \ \ \delta \in\,(0,1],\end{equation}
and then, as $S^\delta_{-\infty}$ is lower-semicontinuous, the compactness of the level sets of $S_{-\infty}$ implies the compactness of the level sets of $S_{-\infty}^\delta$.
Moreover,  in view of   Proposition \ref{prop1},  \eqref{brc500} and the compactness of the level sets of $S_{-\infty}$ imply the equi-coercivity of the family $\{S^\delta_{-\infty}\}_{\delta \in\,(0,1]}$.

Hence,  we have only to prove that
every  sequence $\big\{u_n\big\}$ in $\mathcal{X}$
  such that
$S_{-\infty}(u_n)\leq r,$ for any $n \in\,\mathbb{N}$,
 has a subsequence convergent in $\mathcal{X}$ to some $u\in \mathcal{X}$ such that $S_{-\infty}(u)\leq r$.

According to the last part of Proposition \ref{prop-infty}  there exists  $M>0$ such that
\[|u_n|_{W^{1,2}(-\infty,0)}\leq M,\ \ \ \ n \in\,\mathbb{N}.\]
 Hence by the Banach-Alaoglu Theorem, we can find $u\in W^{1,2}(-\infty,0)$ and a subsequence of the original sequence that is weakly convergent to $u$ in $W^{1,2}(-\infty,0)$. Note that $u$ being an element of  $W^{1,2}(-\infty,0)$ it must satisfy
 \[\lim_{t\to -\infty} \vert u(t)\vert_{\rH}=0.\]
Since the embedding $D(\rA)\embed \rH$ is compact, by the last part of \cite[Corollary 2.8]{Brz+Gat_1999} with $\alpha=1$ and $q=2$, we infer that that for each $T>0$ the embedding $W^{1,2}(-T,0) \embed C([-T,0],\rH)$ is compact. Hence,  for each $T>0$ we can extract a subsequence strongly convergent in $C([-T,0],\rH)$. By the uniqueness of the limit we infer that the later limit is equal to the restriction of $u$ to the interval $[-T,0]$. In particular \delh{$u(0)=x$ and therefore} $u \in \mathcal{X}\delh{_x}$. Moreover, by  employing the Helly's diagonal procedure, we can find a subsequence of $\big\{u_n\big\}$ which is convergent in $\mathcal{X}$ to $u$ and, as $S_{-\infty}$ is lower semicontinuous, we have that $S_{-\infty}(u)\leq r$. This completes the proof of the compactness of the level sets of $S_{-\infty}$.
\end{proof}

\medskip

\section{The quasi-potential}\label{sec-quasi}

We define, for $\phi\in \rH$, the following  of $[0,\infty]$-valued functions
\begin{equation}
\label{eqn-B4}
U(\phi):=\inf\big\{S_{-T}(u):  T>0,\, u\in C([-T,0];\rH),\  \text{with}\ u(-T)=0,\ u(0)=\phi\big\},
\end{equation}
and for any $\delta \in\,(0,1]$
\begin{equation}
 \label{eqn-A4}
U_\delta(\phi):=\inf\big\{S^\delta_{-T}(u): T>0,\, u\in C([-T,0];\rH),\  \text{with}\ u(-T)=0,\ u(0)=\phi\big\}.
\end{equation}
Note that with our notation $U=U_0$.

As a consequence of Lemma \ref{lem-clar}
 and Lemma \ref{lem-clar-2}, we have the following result.

 \begin{proposition}\label{cor-equiv} {We have}
 \begin{equation}
 \label{ops1}
  U(\phi)<\infty \Longleftrightarrow \phi\in \rV.
  \end{equation}
Moreover, if Assumption  \ref{ass-Q} is satisfied for some $\beta \in (0,\frac12)$, then  we have
\begin{equation}
\label{ops2} U_\delta(\phi)<\infty \Longleftrightarrow \phi\in D(\rA^{\frac{\beta+1}2}).
\end{equation}

\end{proposition}

\begin{proof}
We   prove \eqref{ops2}, as \eqref{ops1} turns out to be a special case, corresponding to the case $\beta=0$. Assume that $U_\delta(\phi)<\infty$.
Then, according to   \eqref{eqn-A4} we can find $T>0$ and $u\in C([-T,0];\rH)$ such that $u(-T)=0$, $u(0)=\phi$ and
$$u^\prime+\rA u+\rB(u,u)\in L^2(-T,0;D(\rA^{\frac \beta 2})).$$
Hence, by   Lemma \ref{lem-clar-2} we infer that $\phi\in\,D(\rA^{\frac{1+\beta}2})$.
\\
Conversely, let us assume that $\phi\in D(\rA^{\frac{1+\beta}2})$. Since for any $T>0$ the map
\[W(-T,0;D(\rA^{1+\frac{\beta}2}),D(\rA^{\frac{\beta}2}))\ni v\mapsto v(0)\in D(\rA^{\frac{1+\beta}2})\] is  surjective, see \cite[Theorem 3.2, p.21 and Remark 3.3, p.22]{Lions+Magenes_1972} for a proof,  we can find $u_1\in W(-T,0;D(\rA^{1+\frac{\beta}2}),D(\rA^{\frac{\beta}2}))$ such that $u_1(0)=\phi$.  By Proposition \ref{prop-Leray-fractional-alpha}, we infer that $u_1^\prime+Au_1+B(u_1,u_1) \in\,L^2(-T,0;D(\rA^{\frac{\beta}2}))$. Moreover,  there exists $t_0 \in\,(-T,0)$ such that $u_1(t_0) \in\,D(\rA^{1+\frac{\beta}2})$. This means that if we define
\[u_2(t):=\frac{t+T}{t_0+T} u_1(t_0),\ \ \ \ t \in\,[-T,t_0],\]
we have that $u_2(-T)=0$, $u_2 \in\,W(-T,t_0;D(\rA^{1+\frac{\beta}2}),D(\rA^{\frac{\beta}2}))$ and $u_2^\prime+Au_2+B(u_2,u_2) \in\,L^2(-T,t_0;D(\rA^{\frac{\beta}2}))$.
Finally, if we define
\[u(t):=\begin{cases}
u_2(t),  &  t \in\,[-T,t_0]\\
u_1(t),  &  t \in\,[t_0,0],
\end{cases}\]
we can conclude that
$S^\delta_{-T}(u)<\infty$.

\end{proof}

Now we can prove the following crucial characterization of the functionals $U_\delta$ and $U$.

\begin{theorem}
\label{Fact 1} For any $\phi \in\,\rV$ we have
\begin{equation}
\label{eqn-B1}
U(\phi):= \min\big\{S_{-\infty}(u): u\in \mathcal{X}_\phi\big\}.
\end{equation}
Analogously, if Assumption  \ref{ass-Q} is satisfied for some $\beta \in (0,\frac12)$, then
 for any  $\delta\in (0,1]$ and $\phi \in D(\rA^{\frac{1+\beta}2})$ we have
\begin{equation}
\label{eqn-A1}
U_\delta(\phi):= \min\big\{S^\delta_{-\infty}(u): u\in \mathcal{X}_\phi\big\}.
\end{equation}
\end{theorem}

\begin{proof}
We prove \eqref{eqn-A1}, as \eqref{eqn-B1} is a special case, corresponding to $\beta=0$ in Assumption \ref{ass-Q}. Let us fix $T>0$ and $u\in C([-T,0];\rH)$ such that $u(-T)=0$, $u(0)=\phi$ and $S^\delta_{-T}(u)<\infty$ and let us define
\begin{equation}
\label{eqn-bar-u}
\bar{u}(t):= \begin{cases}u(t),& \mbox{if } t \in [-T,0],\\
                            0, & \mbox{if } t \in (-\infty,-T].
\end{cases}
\end{equation}
Obviously, $\bar{u}\in \mathcal{X}_\phi$. We will prove that
\begin{equation}
\label{eqn-*1}
S^\delta_{-\infty}(\bar{u})=S^\delta_{-T}(u)
\end{equation}
Since $S^\delta_{-T}(u)<\infty$,  the function $u$ satisfies the assumptions of Lemma \ref{lem-clar-2}. Therefore, $\phi\in D(\rA^{\frac{1+\beta}2})$ and   $u$ belongs to $W^{1,2}(-T,0;D(\rA^{1+\frac \beta 2}),D(\rA^{\frac \beta 2}))$. Since obviously the zero function is an elements of the space $W^{1,2}(-\infty,-T;D(\rA^{1+\frac \beta 2}),D(\rA^{\frac \beta 2}))$, we infer, see for instance \cite{Brz_1991},  that
\[\bar{u} \in\,W^{1,2}(-\infty,0;D(\rA^{1+\frac \beta 2}),D(\rA^{\frac \beta 2})),\] and  \eqref{eqn-*1} holds.
In particular, this implies  that
$$ \inf\big\{S^\delta_{-\infty}(u): u\in \mathcal{X}_\phi\big\}\leq S^\delta_{-T}(u).$$
Taking now the infimum over all $u$ as above, in view of the definition of $U_\delta(\phi)$  we infer that
$$ \inf\big\{S^\delta_{-\infty}(u): u\in \mathcal{X}_\phi\big\}\leq  U_\delta(\phi).$$

\medskip

It  remains to prove the converse inequality. To this purpose, we will need the following two results, whose proofs are postponed to the end of this section.

\begin{lemma}\label{prop-control} For every  $\delta \in\,[0,1]$, $T>0$ and $\eps>0$, there exists $\eta>0$  such that  for any  $y\in D(\rA^{\frac{1+\beta}2})$ such that  $\vert y\vert_{D(\rA^{\frac{1+\beta}2})} <\eta$, we can find
\[v \in\,W(0,T;D(\rA^{1+\frac \beta 2}),D(\rA^{\frac \beta 2}))\]
with
\begin{equation*}
S^\delta_{T}(v)< \eps,\ \ \ v(0)=0,\ \ v(T)=y.
\end{equation*}
\end{lemma}
Recall that in the case $\delta=0$, we have $S^\delta_{T}=S_{T}$ and we take $\beta=0$.

\begin{lemma}
\label{lem-claim1}
Assume that $u\in \mathcal{X}$. Then for each $\delta \in\,[0,1]$ and $\eps>0$ we can find $T_\eps>0$ and $v_\eps\in C([-T_\eps,0];\rH)$ such that  $v_\eps(-T_\eps)=0$,  $v_\eps(0)=u(0)$ and
$$
S^\delta_{-T_\eps}(v_\eps) \leq S^\delta_{-\infty}(u)+\eps.$$
\end{lemma}

\medskip

Thus, let us prove
\begin{equation}
\label{brc73}
U_\delta (\phi)\leq \inf\{ S^\delta_{-\infty}(u): u\in \mathcal{X}_\phi\}.
\end{equation}
Obviously, we may assume that the right hand side above is finite and so we can find  $u\in \mathcal{X}_\phi$ such that $S^\delta_{-\infty}({u})<\infty$.
In view of Lemma \ref{lem-claim1}, for any $\eps>0$,
$$\inf \big\{S^\delta_{-T} (v): T>0, \, v \in C([-T,0],\rH),\ v(-T)=0,\ v(0)=\phi\big\} \leq S^\delta_{-\infty}(u)+\eps.$$
This implies that $U_\delta(\phi)\leq S^\delta_{-\infty}(u)+\eps$. Thus, by taking the infimum over $\eps>0$ and then over all admissible $u$ we get \eqref{brc73}.
Finally, we remark that the infima are in fact minima, as the level sets of $S_{-\infty}$ and $S^\delta_{-\infty}$ are compact (see Proposition \ref{cor-compact}).

This completes the proof of \eqref{eqn-A1}, provided we can prove Lemmas \ref{prop-control} and  \ref{lem-claim1}.
\end{proof}

\medskip

\begin{proof}[Proof of Lemma \ref{prop-control}]
Let us fix $T>0$ and consider the mapping

\begin{equation*}
\mathcal{H}_T:  W^{1,2}(0,T;D(\rA^{1+\frac \beta 2}),D(\rA^{\frac \beta 2}))\ni v\mapsto v^\prime+\rA v+\rB(v,v)\in L^2(0,T;D(\rA^\frac \beta 2)).
\end{equation*}
Due to Lemma \ref{prop-Leray-fractional-alpha}, the mapping  $\mathcal{H}_{T}$ is well defined and  continuous. Moreover
\begin{equation}
\label{brc25}
S^\delta_{T}(v)\leq c\vert \mathcal{H}_T(v)\vert_{L^2(0,T;D(\rA^{\frac \beta 2}))}, \;\; v\in W^{1,2}(0,T;D(\rA^{1+\frac \beta 2}),D(\rA^{\frac \beta 2})).\end{equation}

Now, by proceeding as in  \cite[Remark 3.3, p. 22]{Lions+Magenes_1972} we can show that there exists a continuous linear map
\[R:D(\rA^{\frac{1+\beta}2}) \to   W^{1,2}(0,T;D(\rA^{1+\frac \beta 2}),D(\rA^{\frac \beta 2})),\]
 such that $[Ry](T)=y$ for every $y\in D(\rA^{\frac{1+\beta}2})$. By using the same augments used in the proof of Proposition \ref{cor-equiv}, we can construct $Ry$ such that $Ry(0)=0$. Thus the map
\[ \mathcal{H}_T \circ R: D(\rA^{\frac{1+\beta}2})\to L^2(0,T;D(\rA^{\frac{\beta}2}))\] is continuous and then for every $\eps>0$ we can find $\eta>0$ such that
\[\vert y\vert_{D(\rA^{\frac{1+\beta}2})}<\eta\Longrightarrow \vert \mathcal{H}_{T} (R y)\vert_{L^2(0,T;D(\rA^{\frac \beta 2})}<\frac{\eps}c.\] Since  $v=Ry$ satisfies $v\in W^{1,2}(0,T;D(\rA^{1+\frac \beta 2}),D(\rA^{\frac \beta 2}))$, $v(0)=0$ and $v(T)=y$, due to \eqref{brc25} the proof is complete.
\end{proof}

\medskip

\begin{proof}[Proof of Lemma \ref{lem-claim1}] We give the proof here for $\delta>0$, as $\delta=0$ is a special case. Let us assume that   $u\in \mathcal{X}_\phi$ for some  $\phi\in H$, and fix  $\eps>0$. We can assume  that $S^\delta_{-\infty}(u)<\infty$.   Then by \eqref{eqn-A5}   we can find $T_\eps>0$ such that
$$S^\delta_{-\infty,-T_\eps}(u)<\frac\eps3.$$
Moreover, the function $u$ satisfies the assumptions of Lemma \ref{lem-clar-2}. Therefore, $\phi\in D(\rA^{\frac{1+\beta}2})$ and   $u$ belongs to $W_{\mathrm{loc}}^{1,2}(-\infty,0;D(\rA^{1+\frac \beta 2}),D(\rA^{\frac \beta 2}))$. As a consequence of Proposition \ref{prop-inftybis}, this implies
\begin{equation}
\label{eqn-*3}
\lim_{t\to-\infty} \vert u(t)\vert_{D(\rA^{\frac{1+\beta}2})}=0.
\end{equation}
Then,  $T_\eps$ can be chosen in such a way that
$$ \vert u(-T_\eps)\vert_{D(\rA^{\frac{1+\beta}2})} < \eta,$$
where we  choose $\eta>0$ as in Lemma \ref{prop-control}, corresponding to $T=1$ and  $\frac\eps3$.
Then by Lemma \ref{prop-control},  we can find   $w\in W^{1,2}(-T_\eps-1,-T_\eps;D(\rA^{1+\frac \beta 2}),D(\rA^{\frac \beta 2}))$ such that
\begin{equation*}
S^\delta_{-T_\eps-1,-T_\eps}(w)< \frac\eps3,\ \ \ w(-T_\eps-1)=0,\ \ w(-T_\eps)=u(-T_\eps).
\end{equation*}
Next,  we define
\begin{equation}
\label{eqn-v}
\bar{u}(t):= \begin{cases} u(t),& \mbox{if } t \in [-T_\eps,0],\\
                            w(t), & \mbox{if } t \in [-T_\eps-1,-T_\eps].
\end{cases}
\end{equation}
Obviously, $\bar{u}(0)=\phi$ and $\bar{u}\in C([-T_\eps-1,0];\rH)$ and,   arguing as before (and hence using for instance  \cite{Brz_1991}), we infer that $\bar{u}\in W^{1,2}(-T_\eps-1,0;D(\rA^{1+\frac \beta 2}),D(\rA^{\frac \beta 2}))$.  Moreover,
\begin{eqnarray*}
S^\delta_{-T}(\bar{u})&=&S^\delta_{-T_\eps-1,-T_\eps}(w)+S^\delta_{-T_\eps}(u)\\
&<& \frac\eps3+\big[ S^\delta_{-\infty}(u)-S^\delta_{-\infty,-T_\eps}(u) \big]< \frac\eps3+ S^\delta_{-\infty}(u).
\end{eqnarray*}
This concludes the proof of Lemma \ref{lem-claim1}.
\end{proof}

\medskip

Next, we prove that both $U$ and $U^\delta$ have compact level sets.

\begin{proposition}\label{prop-lsc}
For any $r>0$ and $\delta \in\,(0,1]$, the sets
\[K_r=\{\phi\in \rH: U(\phi)\leq r\},\ \ \ \ K_r^\delta=\{\phi\in \rH: U_\delta(\phi)\leq r\}\]
are compact in $\rH$.
In particular, both  $U$ and $U^\delta$ are lower semi-continuous in $\rH$.

\end{proposition}

\begin{proof}
Let $\{\phi_n\}$ be a sequence in $K_r$. In view of \del{Proposition} identity  \eqref{eqn-B1}, for any $n \in\,\mathbb{N}$ there exists $u_n \in\,\mathcal{X}_{\phi_n}$ such that
\[S_{-\infty}(u_n)=U(\phi_n)\leq r.\]
In particular,
\[\{u_n\}\subset \{S_{-\infty}\leq r\},\]
so that, thanks to the compactness of the level sets of $S_{-\infty}$ proved in Proposition \ref{cor-compact}, we {can find a subsequence} $\{u_{n_k}\}\subset \{u_n\}$ and $\bar{u} \in\,C((-\infty,0];\rH)$ such that
\[\lim_{k\to\infty}u_{n_k}=\bar{u},\ \ \ \ \text{in}\ C((-\infty,0];\rH).\]
This implies that
\[\del{\lim_{k\to\infty}\phi_{n_k}=}\lim_{k\to\infty}u_{n_k}(0)=\bar{u}(0),\]
and, due to the lower semi-continuity  of $S_{-\infty}$ proved in Proposition \ref{Fact 3},
\[S_{-\infty}(\bar{u})\leq \liminf_{k\to \infty} S_{-\infty}(u_{n_k})\leq r.\]
On the other hand,  by the definition of $U$,
$U(\bar{u}(0))\leq S_{-\infty}(\bar{u})$.
Hence we can conclude that $\bar{u}(0) \in\,K_r$, and the compactness of $K_r$ follows.

The compactness of the level sets \del{$K_r^\delta$} of $U_\delta$ can be proved analogously.

\end{proof}

\medskip

We conclude this section by studying the continuity of $U$ in $\rV$ and of $U_\delta$ in $D(\rA^{\frac{1+\beta}2})$.

\begin{proposition}
\label{brc505}
The maps $U:\rV\to \mathbb{R}$  and $U_\delta: D(\rA^{\frac{1+\beta}2})\to \mathbb{R}$ are continuous.
\end{proposition}

\begin{proof}
In the previous proposition we have seen that $U$ is lower semi-continuous in $\rH$. In particular, it is lower semi-continuous in $\rV$. Thus, is we prove that $U$ is also upper semi-continuous in $\rV$, we can conclude that it is continuous on $\rV$.

Let $\{\phi_n\}_{n \in\,\mathbb{N}}$ be a sequence in $\rV$ converging to some $\phi$ in $\rV$. As $\phi \in\,\rV$, according to Proposition \ref{cor-equiv} and  Theorem \ref{Fact 1}, there exists $u \in\,\mathcal{X}_\phi\cap W^{1,2}(-\infty,0)$ such that $U(\phi)=S_{-\infty}(u)$.
Now, we define
\[u_n(t)=u(t)+e^{t\rA}(\phi_n-\phi),\ \ \ \ t\leq0.\]

Clearly $u_n(0)=\phi_n$. Then, as $\phi_n-\phi \in\,\rV$, we have that $u_n \in\,\mathcal{X}_{\phi_n}\cap W^{1,2}(-\infty,0)$.
Moreover, as $\phi_n$ converges to $\phi$ in $\rV$ and $\rV=(\rH,D(\rA))_{\frac12,2}$, we infer  that $u_n$ converges to $u$ in $W^{1,2}(-\infty,0)$, so that
\[\lim_{n\to\infty}S_{-\infty}(u_n)=S_{-\infty}(u).\]
This allows to conclude that
\[U(\phi)=S_{-\infty}(u)=\lim_{n\to \infty} S_{-\infty}(u_n)\geq \limsup_{n\to\infty} U(\phi_n),\]
so that upper semi-continuity follows.

The proof of the continuity of {the map} $U_\delta: D(\rA^{\frac{1+\beta}2})\to \mathbb{R}$ is analogous.

\end{proof}

\section{Stochastic Navier Stokes equations with periodic boundary conditions}\label{sect-periodic}

All what we have discussed throughout the paper until now applies to the case when the Dirichlet boundary conditions are replaced by the periodic boundary conditions. In the latter case, it is customary to study our problem  in the $2$-dimensional torus $\mathbb{T}^2$ (of fixed dimensions $L\times L$), instead of a regular bounded domain $\mathcal{O}$. All the mathematical background can be found in the small book \cite{Temam_1983} by Temam. In particular, the space $\rH$ is equal to
\[\H=\{ u\in L_0^2(\mathbb{T}^2,\mathbb{R}^2): \divv (u)=0 \mbox{ and } \gamma_\nu(u)_{\vert \Gamma_{j+2}}=-\gamma_\nu(u)_{\vert \Gamma_{j}}, \; j=1,2\}, \]
where $L_0^2(\mathbb{T}^2,\mathbb{R}^2)$ is the Hilbert space consisting of those $u\in L^2(\mathbb{T}^2,\mathbb{R}^2)$ which satisfy $\int_{\mathbb{T}^2} u(x)\, dx=0$ and
$\Gamma_j$, $j=1,\cdots,4$ are the four (not disjoint) parts of the boundary of $\partial(\mathbb{T}^2)$ defined by
\[
\Gamma_j=\{ x=(x_1,x_2) \in [0,L]^2: x_j=0\},\;\Gamma_{j+2}=\{ x=(x_1,x_2) \in [0,L]^2: x_j=L\},\;\; j=1,2.
\]
The Stokes operator $\rA$ can be defined in a natural way and it satisfies all the properties know in the bounded domain case, inclusive the positivity property \eqref{ineq-Poincare-A} (with $\lambda_1=\frac{4\pi^2}{L^2}$) and the following property involving the nonlinear term $B$
\begin{eqnarray}\label{eqn-A-B}
\left<\rA u,B(u,u)\right>_\rH=0, \;\; u\in D(\rA),
\end{eqnarray}
see \cite[Lemma 3.1]{Temam_1983} for a proof.
The Leray-Helmholtz projection operator $P$ has the following explicit formula using the Fourier series, see   \cite[(2.13)]{Temam_1983}
\[
[P(f)]_k=\frac{L^2}{4\pi^2}\big(f_k-\frac{(k\cdot f_k)f_k}{\vert k\vert^2} \big),\;\ k
\in\mathbb{Z}^2\setminus\{0\},\; f=\sum_{n\in\mathbb{Z}^2\setminus\{0\}} f_n e^{\frac{2\pi i n\cdot x}{L}}\in L^2_0(\mathbb{T}^2,\mathbb{R}^2).
\]
It follows from the above that $P$ is a bounded linear map from $D(\rA^\alpha)$ to itself for every $\alpha\geq 0$, compare with Proposition \ref{prop-Leray-fractional} in the bounded domain case.

In  the next Theorem we will show that, in this case,  an explicit representation of $U(x)$ can be given, for any $x \in\,\rV$.

\begin{theorem}\label{thm-Vper}
Assume that periodic boundary conditions hold. Then
\[U(\phi)=\begin{cases}
|\phi|_\rV^2,  &\phi \in\,\rV,\\
+\infty,  &  \phi \in\,\rH\setminus \rV.
\end{cases}\]
\end{theorem}

\begin{proof}
By  Theorem \ref{Fact 1}, we have that
\[U(\phi)=\min\,\left\{\,S_{-\infty}(u)\ :\ u \in\,\mathcal{X}_\phi\,\right\},\ \ \ \ \phi \in\,\rV,\]
and by Proposition \ref{cor-equiv} we have  that $U(\phi)<\infty$ if and only if $\phi \in\,\rV$.
Now, let us fix $\phi \in\,\rV$ and $u \in\,\mathcal{X}_\phi$ such that $S_{-\infty}(u)<\infty$. In view of Proposition \ref{prop-infty}, we have that
\[u \in\,C((-\infty,0];\rV)\cap L^2(-\infty,0;D(\rA)),\ \ \ \ u^\prime \in\,L^2(-\infty,0;\rH)\]
and
\begin{equation}
\label{brc74}
\lim_{t\to -\infty}|u(t)|_\rV=0.
\end{equation}
We have
\[\begin{array}{l}
|u^\prime(t)+\rA u(t)+\rB(u(t),u(t))|_\rH^2=|u^\prime(t)-\rA u(t)+\rB(u(t),u(t))|_\rH^2\\
\vspace{.1cm}\\
+4|\rA u(t)|_\rH^2+4\left<u^\prime(t)-\rA u(t)+\rB(u(t),u(t)),\rA u(t)\right>_\rH.
\end{array}\]
Then, thanks to \eqref{eqn-A-B} we get
\[|u^\prime(t)+\rA u(t)+\rB(u(t),u(t))|_\rH^2=|u^\prime(t)-\rA u(t)+\rB(u(t),u(t))|_\rH^2+4\left<u^\prime(t),\rA u(t)\right>_\rH.\]
According to \eqref{brc74}, this means that
\[\begin{array}{l}
S_{-\infty}(u)=\frac 12\int_{-\infty}^0|u^\prime(t)-\rA u(t)+\rB(u(t),u(t))|_\rH^2\,dt+\int_{-\infty}^0\frac{d}{dt}|u(t)|_\rV^2\,dt\\
\vspace{.1cm}\\
=\frac 12\int_{-\infty}^0|u^\prime(t)-\rA u(t)+\rB(u(t),u(t))|_\rH^2\,dt+|u(0)|_\rV^2.
\end{array}\]
In particular,
\[U(\phi)\geq |\phi|_\rV^2.\]
On the other hand, if we show that for any $\phi \in\,\rV$ there exists $\bar{u} \in\,W^{1,2}(-\infty,0)\cap \mathcal{X}_\phi$ such that
$\bar{u}^\prime(t)-\rA \bar{u}(t)+B(\bar{u}(t),\bar{u}(t))=0,$ for $t \in\,(-\infty,0)$, we conclude that $U(\phi)=|\phi|_\rV^2$.

As we have seen in Section \ref{sec-prel}, if $\phi \in\,\rV$ then the problem
\[\left\{\begin{array}{l}
v^\prime(t)+\rA v(t)-\rB(v(t),v(t))=0,\ \ \ t>0,\\
\vspace{.1cm}\\
v(0)=\phi,
\end{array}\right.\] admits a unique solution $v \in\,L^2(0,+\infty;D(\rA))\cap C([0,+\infty);\rV)$, with $v^\prime \in\,L^2(0,+\infty;\rH)$,  with
\[\lim_{t\to\infty}|v(t)|_\rV^2=0.\]
This means that if we define
\[\bar{u}(t)=v(-t),\ \ \ \ t\leq 0,\]
we can conclude our proof, as $\bar{u} \in\,W^{1,2}(-\infty,0)\cap \mathcal{X}_\phi$ and
$\bar{u}^\prime(t)-\rA \bar{u}(t)+B(\bar{u}(t),\bar{u}(t))=0$.

\end{proof}

We have already mentioned in the Introduction that  a  finite dimensional counterpart of Theorem  \eqref{thm-Vper} was first derived  in Theorem IV.3.1 in the monograph \cite{freidlin}. It has later been discussed in Example B.2 for finite dimensional Landau-Lifshitz-Gilbert equations by Kohn et al \cite{KRV}.

\section{Convergence of $U_\delta$ to $U$}\label{sec-conv}

Our aim in this section  is to prove Theorem \ref{thm-aim}, that is
\[\lim_{\delta \to 0}U_\delta(\phi)=U(\phi), \ \ \ \ \phi \in\,D(\rA^{\frac{\beta+1}2}).\]

To this purpose, we introduce an auxiliary  functional $\tilde{S}_{-\infty}: \mathcal{X}_\phi \to [0,\infty]$, where
$\phi \in\,D(\rA^{\frac{\beta+1}2})$ is fixed,  by the formula
\begin{equation}\label{eqn-St}
\tilde{S}_{-\infty}(v):=\begin{cases}
{S}_{-\infty}(v), & \mbox{ if } v \in \mathcal{X}_\phi \cap W^{1,2}\big(-\infty,0;D(\rA^{\frac{\beta}2+1}),D(\rA^{\frac{\beta}2})\big),\\
+ \infty, & \mbox{ if } v \in \mathcal{X}_\phi \setminus W^{1,2}\big(-\infty,0;D(\rA^{\frac{\beta}2+1}),D(\rA^{\frac{\beta}2})\big) .
\end{cases}
\end{equation}

\medskip

\begin{lemma}\label{Fact 4} Under Assumption \ref{ass-Q}, if $\phi\in\rH$, then
\begin{equation}\label{eqn-Fact 4}
\Gamma-\lim_{\delta\to 0} S^\delta_{-\infty} =\mathrm{sc}^{-}\tilde{S}_{-\infty} \mbox{ in } \mathcal{X}_\phi.
\end{equation}
\end{lemma}

\begin{proof}
According to Proposition \ref{prop3}, the proof of \eqref{eqn-Fact 4} follows, once we show that for any $u \in \fourIdx{}{}{}{\phi}\cX $ the function
\[(0,1]\ni \delta \mapsto S^\delta_{-\infty}(u)\]
 is decreasing
 and
\begin{equation}
\label{brc6}
\lim_{\delta \to 0} S^\delta_{-\infty}(u)= \tilde{S}_{-\infty}(u),\;\; u \in \mathcal{X}_\phi.
\end{equation}

\medskip
Let us fix a function $u \in \fourIdx{}{}{}{\phi}\cX $.
In view of  Assumption \ref{ass-Q}, for each $y\in D(A^{\frac\beta 2})$, the function
$(0,1]\ni \delta \mapsto \vert Q^{-1}_\delta y\vert_{\rH}^2 \in\,\mathbb{R}$ is decreasing. This implies that for any fixed $u$ the mapping
$(0,1]\ni \delta \mapsto S^\delta_{-\infty}(u)$ is decreasing.

We notice that if $u \in\, \mathcal{X}_\phi \setminus W^{1,2}\big(-\infty,0;D(\rA^{\frac{\beta}2+1}),D(\rA^{\frac{\beta}2})\big)$, then for any $\d \in\,(0,1)$
\[S^\delta_{-\infty}(u)=\tilde{S}_{-\infty}(u)=+\infty,\]
so that \eqref{brc6} follows.
On the other end, if $u \in \mathcal{X}_\phi \cap W^{1,2}\big(-\infty,0;D(\rA^{\frac{\beta}2+1}),D(\rA^{\frac{\beta}2})\big)$, we have
$$\tilde{S}_{-\infty}(u)=S_{-\infty}(u)=\frac12 \int_{-\infty}^0\vert \cH(u)(t)\vert_{\rH}^2\, dt.$$
Thus, since
$S^\delta_{-\infty}(u)=\frac12 \int_{-\infty}^0\vert Q_\delta^{-1}\cH(u)(t)\vert_{\rH}^2\, dt$,  by the Lebesgue dominated convergence theorem we obtain \eqref{brc6}, once we have observed  that  according to Assumption \ref{ass-Q},  for all $y\in D(\rA^{\frac\beta 2})$ it holds $Q_\delta^{-1} y\to y$, as $\delta \searrow 0$, and $\vert Q_\delta^{-1} y\vert_{\rH} \leq \vert Q_1^{-1} y\vert_{\rH} $.

\end{proof}

\begin{lemma}\label{thm-Gamma} If  $\phi\in \rV$, then
\begin{equation}
\label{brcf1000} \mathrm{sc}^{-}\tilde{S}_{-\infty}(u)=S_{-\infty}(u),\ \ \ \ u \in\,\mathcal{X}_\phi.\end{equation}
\end{lemma}

\begin{proof} In view of  Proposition \ref{charac},   we get \eqref{brcf1000} if we show that
for every sequence $\big\{u_n\big\}_n\subset \mathcal{X}_\phi$ convergent to $u$ in $\mathcal{X}_\phi$ it holds
\begin{equation}
\label{brc4}
S_{-\infty}(u) \leq \liminf_{n} \tilde{S}_{-\infty} (u_n),
\end{equation}
and for some sequence
$\big\{u_n\big\}_n\subset \mathcal{X}_\phi$ convergent to $u$ in $\mathcal{X}_\phi$ it holds
\begin{equation}
\label{brc5}
S_{-\infty}(u) \geq \limsup_{n} \tilde{S}_{-\infty} (u_n).
\end{equation}
It is immediate to check that \eqref{brc4} follows from the lower semi-continuiuty of $S_{-\infty}$ and the definition of $\tilde{S}_{-\infty}$. Thus, let us prove   \eqref{brc5}.
We are going to prove that there exists a sequence $\{u_n\}$ in  $ \mathcal{X}_\phi \cap W^{1,2}\big(-\infty,0;D(\rA^{\frac{\beta}2+1}),D(\rA^{\frac{\beta}2})\big)$  such that
\begin{equation}\label{eqn-Fact 6_1}
 \lim_{n\to \infty} \; \sup_{t\in (-\infty,0]} \vert u_n(t)-u(t)\vert_\rH=0,
\end{equation}
 and
\begin{equation}\label{eqn-Fact 6_2}
S_{-\infty}(u) \geq \limsup_{n\to \infty} \tilde{S}_{-\infty}(u_n).
\end{equation}

To this purpose, we can assume that $S_{-\infty}(u)<\infty$. Then,  according to  Proposition \ref{prop-infty}, we have that $u  \in \mathcal{X}_\phi \cap W^{1,2}\big(-\infty,0)$ and $\phi=u(0)\in \rV$.
Since
\[W^{1,2}\big(-\infty,0) \embed C_b((-\infty,0],\rH),\] it is enough to find a sequence $\{u_n\}$ in  $\mathcal{X}_\phi \cap W^{1,2}\big(-\infty,0;D(\rA^{\frac{\beta}2+1}),D(\rA^{\frac{\beta}2})\big)$ satisfying \eqref{eqn-Fact 6_2} and, instead of \eqref{eqn-Fact 6_1}, the following stronger condition
\begin{equation}\label{eqn-Fact 6_3}
 \lim_{n\to \infty} \;  \vert u_n-u\vert_{W^{1,2}(-\infty,0)} =0.
\end{equation}
Actually, if we have found a  sequence
 $\big\{u_n\big\}\subset \mathcal{X}_\phi \cap W^{1,2}\big(-\infty,0;D(\rA^{\frac{\beta}2+1}),D(\rA^{\frac{\beta}2})\big)$    satisfying \eqref{eqn-Fact 6_3}, then in view of   \eqref{eqn-St}, $\tilde{S}_{-\infty}(u_n)={S}_{-\infty}(u_n)$ for every $n$. Therefore,  in view of \eqref{eqn-Fact 6_3}, we obtain \eqref{eqn-Fact 6_2}, as ${S}_{-\infty}$ is a
continuous functional on $W^{1,2}(-\infty,0)$.
 Let us finally observe that the existence of the required sequence is just a consequence of the  density of the space $\fourIdx{}{}{}{\phi}\cX \cap W^{1,2}(-\infty,0;D(\rA^{\frac{1+\beta}2}),D(\rA^{\frac{\beta}2}))$ in $\fourIdx{}{}{}{\phi}\cX \cap W^{1,2}(-\infty,0)$.

\end{proof}

Thus we can conclude that the following limit holds.

\begin{theorem}\label{thm-aim}
Under Assumption \ref{ass-Q}, we have
\begin{equation}\label{eqn-4.01} \lim_{\delta \to 0}U_\delta(\phi)=U(\phi), \ \ \ \ \phi \in\,D(\rA^{\frac{\beta+1}2}).
\end{equation}
\end{theorem}
\begin{proof}
 Let us fix $\phi\in D(\rA^{\frac{\beta+1}2})$. In view of Theorem \ref{Fact 1}
   \[U(\phi)=\min\big\{S_{-\infty}(u): u\in \mathcal{X}_\phi\big\}.\]
   and for any $\delta  \in\,(0,1]$
 \[U_\delta(\phi)=\min\big\{S^\delta_{-\infty}(u): u\in \mathcal{X}_\phi\big\}.\]

Thus, thanks to Theorem \ref{teo4}, our result is proved  as we have shown  that for any $\phi \in\,D(\rA^{\frac{\beta+1}2})$ the family $\{S^\delta_{-\infty}\}_{\delta \in\,(0,1]}$ is equi-coercive in $\mathcal{X}_\phi$ and, as a consequence of Lemma \ref{Fact 4}   and Lemma
\ref{thm-Gamma},
\begin{equation}
\label{brcf13}
\Gamma-\lim_{\delta \to 0}S^\delta_{-\infty}=S_{-\infty},\ \ \ \text{in}\ \mathcal{X}_\phi.
\end{equation}

\end{proof}

\section{Application to the exit problem}\label{sec-exit}

A  domain $D\subset \rH$ is said to be invariant and attracted to the asymptotically stable equilibrium $0$ of  the deterministic  Navier-Stokes equations \eqref{eqn_NSE01}
\begin{equation}\label{eqn-SNS-det}
u^\prime(t)+A u(t)+B(u(t),u(t))=0,\ \ \ \ u(0)=\phi,
\end{equation}
if,  for any $\phi \in\,D$,  the solution $u_\phi(t)$ to \eqref{eqn-SNS-det} remains in $D$, for every $t\geq 0$,   and
\[\lim_{t\to\infty }|u_\phi(t)|_\rH=0.\]

It is well known that, as the solution $u_\phi$ satisfies inequality \eqref{ineq-aux-00}  by the Poincar\'e inequality \eqref{ineq-Poincare-A},  every  ball in $\rH$ is invariant and attracted to $0$.

\medskip

Throughout this section, we will assume the following conditions on  $D$.
\begin{assumption}
\label{D}
 The set $D\subset \rH$  is bounded, open, connected, contains  $0$,  is invariant and attracted to $0$.  Moreover,
 for any $\phi \in\partial D\cap D(\rA^{\frac{1+\beta}{2}})$ there exists a sequence $\{\phi_n\}\subset  (\rH \setminus \bar{D})\cap D(\rA^{\frac{1+\beta}{2}})$ such that
\[\lim_{n\to \infty}|\phi_n-\phi|_{D(\rA^{\frac{1+\beta}{2}})}=0.\]
\end{assumption}

\begin{remark}
If for every $\phi \in\partial D\cap D(\rA^{\frac{1+\beta}{2}})$ there exists $y \in\,\rH\setminus \bar{D}\cap D(\rA^{\frac{1+\beta}{2}})$ such that
\[\{t\phi +(1-t)y{\;\;: t\in [0,1)}\}\subset \rH\setminus \bar{D},\]
then Assumption \ref{D} is clearly satisfied. Such a property is true if, for example, $D$ is convex.
\end{remark}

\begin{lemma}
\label{point}
For any $\delta \in\,(0,1]$,  there exists $y_\delta \in\,\partial D$ such that
\begin{equation}\label{eqn-boundary point}
\inf_{y \in\,\partial D} U_\delta(y)=U_\delta(y_\delta).
\end{equation}
\end{lemma}

\begin{proof}{ First we will show that
\begin{equation}\label{eqn-finitequasi}
\inf_{\phi \in \partial D} U_\delta(\phi)<\infty.\end{equation}
}
Since  $\rH\setminus \bar{D}$ is an open subset of $\rH$ and the space $D(\rA^{\frac{1+\beta}2})$ is dense in $\rH$,  there exists $\tilde{\phi} \in (\rH\setminus \bar{D}) \cap D(\rA^{\frac{1+\beta}2})$. Since $0 \in D$, and the path $t \mapsto t\tilde{\phi}$ is continuous, there must exist $0<t_0<1$ such that $ t_0\tilde{\phi} \in \partial D$. Clearly, $t_0 \tilde{\phi} \in\,D(\rA^{\frac{1+\beta}2}) $, so that, as $\partial D\cap D(\rA^{\frac{1+\beta}2}) \neq \emptyset$, according to Proposition \ref{cor-equiv}, {property \eqref{eqn-finitequasi} follows. }

Due to the compactness of the level sets  {of the functionals $U_\delta$}, {we infer}  that there exists $y_{\delta} \in\,\partial D\cap D(\rA^{\frac{1+\beta}2})$ such   that \eqref{eqn-boundary point} holds.

\end{proof}

Now, for any $\phi \in\,D$, $\eps>0$ and $\delta \in\,(0,1]$, we will denote by $\tau_\phi^{\eps,\delta}$ the exit time of the solution $u_\phi^{\eps,\delta}$ of equation \eqref{eqn-SNSE-eps}  from the domain $D$, that is
\[\tau_\phi^{\eps,\delta}=\inf\left\{t\geq  0\ :\ u_\phi^{\eps,\delta}(t)  \in\,\partial D\right\}.\]

Our purpose here is to prove the following exponential estimate for the expectation of $\tau_\phi^{\eps,\delta}$ in terms of the infimum of $U_\delta$ on the boundary of $D$.

\begin{theorem}
\label{brcteo2}
 For any $\delta \in\,(0,1]$ and $\phi \in\,D$
\[
\lim_{\eps\to 0}\eps \log\,\E\,\tau_\phi^{\eps,\delta}=\min_{y  \in\,\partial D} U_\delta(y).
\]
\end{theorem}

As we already pointed out in \cite[Section 7]{cerrai-freidlin}, the proof of the previous result is based on the few lemmas below,  whose proofs  are postponed till   Appendix \ref{app:B}. Actually, the arguments used in finite dimension (see
\cite[proof of Theorem 5.7.11]{dembo} and \cite[proof of Theorem 4.1]{freidlin}), can be adapted to this infinite dimensional case, once the following preliminary results   are proven.

\medskip

\begin{lemma}
\label{lem1}
For any $\eta>0$ and $\mu>0$, there exist
{
$T_0=T_0(\eta,\mu)>0$ and $h=h(\eta)>0$ such that for all $\phi \in\,B_{0}(\mu)$ there exist $T\leq T_0$ and $v \in\,C([0,T];\rH)$, with $v(0)=\phi$, such that
\begin{equation}
\label{brc89}
d_{\rH}(v(T),\bar{D})=h
\end{equation}
and
\begin{equation}
\label{brc89'}
S^\delta_{0,T}(v)\leq \inf_{y \in\,\partial D} U_\delta(y)+\eta.
\end{equation}
}
\end{lemma}

\medskip

\begin{lemma}
\label{lem2}
There exists $\mu_0>0$ such that for any $\eta>0$ and $\mu \in\,(0,\mu_0]$
\[
\lim_{\eps\to 0}\eps \log\left(\inf_{\phi \in\,B_{0}(\mu)}\mathbb{P}\left(\tau_\phi^{\eps,\delta}\leq T\right)\right)>-\left(\inf_{y \in\,\partial D}U_\delta(y)+\eta\right),
\]
for some  $T=T(\eta,\mu)>0$.

\end{lemma}

\medskip

\begin{lemma}
\label{lem3}
For any $\mu>0$  such that $B_{0}(\mu) \subset D$,
\[\lim_{t\to +\infty}\,\liminf_{\eps\to 0}\eps\,\log \left(\,\sup_{\phi \in\,D}\mathbb{P}\left(\sigma_\phi^{\eps,\delta,\mu}>t\right)\right)=-\infty,
\]
where
\[\sigma_\phi^{\eps,\delta,\mu}:=\inf \left\{\,t\geq 0\ ;\ u_\phi^{\eps,\delta}(t) \in\,B_{0}(\mu)\cup \partial D\,\right\}.\]
Moreover,
\[
\lim_{\eps \to 0}\mathbb{P}\left(u_\phi^{\eps,\delta}(\sigma_\phi^{\eps,\delta,\mu}) \in\,B_{0}(\mu)\right)=1.
\]

\end{lemma}

\medskip

\begin{lemma}
\label{lem3bis}
For any closed set $N\subset \partial D$,
\begin{equation}
\label{z9}
\lim_{\mu\to 0}\limsup_{\eps\to 0}\eps\log\left(\sup_{\phi \in\,B_{0}(3\mu)}\mathbb{P}\left(u_\phi^{\eps,\delta}(\sigma_\phi^{\eps,\delta,\mu}) \in\,N\right)\right)\leq -\inf_{\phi \in\,N}U_\delta(\phi).\end{equation}
\end{lemma}

\begin{lemma}
\label{lem4}
 For every $\lambda >0$ and $\mu>0$ such that $B_{0}(\mu)\subset D$, there exists $T=T(\mu,\lambda)<\infty$ such that
\[ \limsup_{\eps\to 0}\eps \log\,\left(\sup_{\phi \in\,B_{0}(\mu)}\mathbb{P}\left(\sup_{t \in\,[0,T]}|u_\phi^{\eps,\delta}(t)-\phi|_{\rH}\geq 3\mu\right)\right)<-\lambda.\]
	\end{lemma}

\bigskip

Next, by proceeding as in  the proof of \cite[Theorem 7.7]{cerrai-freidlin})  we can conclude that the following approximation result holds.
\label{sec:6}
\begin{theorem}
\label{brcteo}
Suppose that Assumption \ref{ass-Q} is satisfied. If for any $\phi \in\,\rV\cap \partial D$ there exists a sequence $\{\phi_n\}\subset D(\rA^{\frac{1+\beta}2})\cap \partial D$ such that
\begin{equation}
\label{brc76}
\lim_{n\to\infty}|\phi_n-\phi|_{\rV}=0,
\end{equation}
then
\begin{equation}
\label{brc75}
\lim_{\delta \to 0}\,\inf_{\phi \in\,\partial D}U_\delta(\phi)= \inf_{\phi \in\,\partial D}U(\phi).
\end{equation}
\end{theorem}

\begin{proof}[Sketch of the Proof]
Limit \eqref{brc75} follows from Theorem \ref{Fact 1} and \eqref{brc76} in virtue of a general argument based on $\Gamma$-convergence and relaxation, which applies to more general situations, and which has been introduced in \cite{cerrai-freidlin}. Actually, we define
\[\tilde{U}(\phi)=\begin{cases}
U(\phi),  &   \phi \in\,\Lambda_\beta,\\
+\infty,  &  \phi \in\,H\setminus \Lambda_\beta,
\end{cases}\]
and for any $\delta \in\,(0,1]$
\[\tilde{U}_\delta(\phi)=\begin{cases}
U_\delta(\phi),  &   \phi \in\,\Lambda_\beta,\\
+\infty,  &  \phi \in\,H\setminus \Lambda_\beta,
\end{cases}\]
where $\Lambda_\beta=D(A^{\frac{1+\beta}2})\cap \partial D$. One can prove that
\[\Gamma-\lim_{\delta\to 0}\tilde{U}_\delta=\text{sc}^-\tilde{U},\ \ \ \ \text{in}\ \rH,\]
and then, by using \eqref{brc76} and the continuity of $U$ in the space $\rV$ proved in Proposition \ref{brc505}, one can show
\[\text{sc}^-\tilde{U}(\phi)=\begin{cases}
U(\phi),  &   \phi \in\,\partial D,\\
+\infty,  &  \phi \in\,\rH\setminus \partial D.
\end{cases}\]
This implies \eqref{brc75}.

\end{proof}

In view of Theorems \ref{brcteo2} and \ref{brcteo}, we obtain the  following result.

\begin{corollary}\label{corollary}
Under the same assumptions of Theorem \ref{brcteo}, we have
\[\lim_{\delta\to 0}\,\lim_{\eps\to 0}\eps \log \mathbb{E}\,\tau_\phi^{\eps,\delta}=\inf_{\phi \in\,\partial D}U(\phi).\]

\end{corollary}

Informally, this means that  for $0<\eps<<\delta<<1$, the following asymptotic formula holds
\[\mathbb{E}\,\tau_\phi^{\eps,\delta}\sim \exp\left(\frac 1\eps \inf_{\phi \in\,\partial D}U(\phi)\right).\]

\begin{rem}
{\em  As in \cite[Remark 7.8]{cerrai-freidlin}, we notice that if we take $D=B_0(r)$, for $r>0$, then   condition \eqref{brc76} assumed in Theorem \ref{brcteo} is fulfilled. Actually, as $D(A^{\frac{1+\beta}2})$ is dense in $\rV$, we can find a sequence $\{\hat{\phi}_n\}\subset D(A^{\frac{1+\beta}2})$ which is convergent to $\phi$ in $\rV$. Then, if we set $\phi_n=r \hat{\phi}_n/|\hat{\phi}_n|_\rH$, we conclude that $\{\phi_n\}\subset D(\rA^{\frac{1+\beta}2})\cap \partial D$ and \eqref{brc76} holds.
}
\end{rem}

\appendix

\section{Proofs of some auxiliary results}
\label{app:A}

\begin{proposition}\label{prop-infty}
Assume that \del{$x\in \rV$ and}
$z\in \mathcal{X}$  is such that $S_{-\infty}(z)<\infty$.
Then, $z(0) \in\,V$,
 \begin{equation}
\label{eqn-A04}
\lim_{t\to-\infty} \vert z(t)\vert_{\rV}=0,
\end{equation}
and  $z\in W^{1,2}(-\infty,0)$, i.e.
\begin{equation}
\label{eqn-A10}
\int_{-\infty}^0 \vert Az(t)\vert_\rH^2\, dt  + \int_{-\infty}^0 \vert z^\prime(t)\vert_\rH^2\, dt<\infty.
\end{equation}
Moreover, there exists a continuous and strictly increasing function $\varphi: [0,\infty) \to [0,\infty)$ such that
$\varphi(0)=0$ and, if
$z \in \mathcal{X}$ is a solution to the problem
\[
z^\prime(t)+\rA z(t)+\rB(z(t),z(t))=f(t), \ \ \  t\leq 0,
\]
with  $f$ being an element of $L^2(-\infty,0)$, then  $z\in W^{1,2}(-\infty,0)$, $z(0)\in V$ and
\[ \vert z(0) \vert_{\rV}^2+\vert z\vert^2_{W^{1,2}(-\infty,0)} \leq
\varphi( \int_{-\infty}^0 \vert f(t)\vert_\rH^2\, dt).    \]
\end{proposition}

\begin{proof}
The argument below is a bit informal but it can easily be made fully rigorous. We will be careful with the constants as we want to prove the last part of the Proposition as well.

We have that $z(0) \in\,V$, as a consquence of Lemma \ref{lem-clar}. Next, we will be proving \eqref{eqn-A04}.

In view  of Lemma \ref{lem-clar}, we can assume that $z\in W^{1,2}_{\mathrm{loc}}(-\infty,0)$. Since $S_{-\infty}(z)<\infty$, if we set
\begin{equation}
\label{brc10}
f(t)=z^\prime(t)+A z(t)+B(z(t),z(t)),\ \ \ \ t\leq 0,
\end{equation}
we have that $f \in\,L^2(-\infty,0;H)$. If we multiply equation \eqref{brc10} by $z$ and  use   equality \eqref{eqn-B02},   we get

\begin{eqnarray}\label{eqn-NSE03b}
\frac12 \frac{d}{dt}
 \vert z(t) \vert_\rH^2+ \vert  z(t) \vert_\rV^2&=&(f, z)_\rH
\leq \frac12 \vert  z(t) \vert_\rV^2 + \frac{1}{2\lambda_1}\vert f(t) \vert_\rH^2,\;\; t<0,
\end{eqnarray}
where     $\lambda_1$ is the Poincar\'e constant of the domain $\mathcal{O}$. Hence, 
\begin{eqnarray}\label{eqn-NSE04'}
\vert z(t)\vert_\rH^2 +\int_s^t \vert  z(r)\vert_\rV^2\, dr
&\leq&  \vert z(s)\vert_\rH^2  +\frac 1{\lambda_1} \int_s^t \vert f(r)\vert_\rH^2  dr, \;\; -\infty < s\leq t\leq 0.
\nonumber
\end{eqnarray}
As
\[\lim_{s\to -\infty}|z(s)|_\rH=0,\]
we infer that
\begin{eqnarray}\label{eqn-NSE05'}
\vert z(t)\vert_\rH^2 +\int_{-\infty}^t \vert  z(s)\vert_\rV^2\, ds
&\leq&    \frac 1{\lambda_1} \int_{-\infty}^t \vert f(r)\vert_\rH^2  dr, \;\; -\infty <  t\leq 0.
\nonumber
\end{eqnarray}
This implies
\begin{equation}\label{eqn-NSE06'}
\vert z(t)\vert_\rH^2
\leq \frac 1{\lambda_1}\int_{-\infty}^t \vert f(r)\vert_\rH^2  \, dr \leq \frac 1{\lambda_1} \int_{-\infty}^0 \vert f(r)\vert_\rH^2  dr,\ \ \ t\leq 0,
\end{equation}
and
\begin{equation}
\label{eqn-NSE07'}
\int_{-\infty}^0 \vert z(s)\vert_\rV^2  ds \leq \frac 1{\lambda_1} \int_{-\infty}^0 \vert f(r)\vert_\rH^2  dr.
\end{equation}
The latter inequality means that  $z\in L^2((-\infty,0],\rV)$, which implies that
 we can find a decreasing sequence $\{s_n\}$ such that $s_n\todown -\infty$ and
\begin{equation}\label{eqn-NSE02'}
\lim_{n\to\infty} \vert z(s_n)\vert_\rV=0.
\end{equation}

Next we multiply equation \eqref{brc10} by $\rA z(t)$. Thanks to \eqref{ineq-B01} and to  the Young inequality,  we get

\begin{eqnarray}
\label{eqn-NSE03}
\frac12 \frac{d}{dt}\vert z(t) \vert_\rV^2+\vert \rA z(t) \vert_\rH^2&=&-(B(z(t),z(t)),Az(t))_\rH+(f(t),\rA z(t))_\rH\\
&\leq& \frac14 \vert \rA z(t) \vert_\rH^2+\frac{C_2}2\vert z(t) \vert_\rH^2\vert z(t) \vert_\rV^4+\frac14 \vert \rA z(t) \vert_\rH^2 +  \vert f(t) \vert_\rH^2.
\nonumber
\end{eqnarray}
where $C_2=\frac{54}4C^2$ and $C$ is the constant from inequality \eqref{ineq-B01} .

 Applying next the Poincar{\'e} inequality \eqref{ineq-Poincare} we get,
\begin{equation}\label{eqn-NSE04}
 \frac{d}{dt}\vert z(t) \vert_\rV^2+\lambda_1 \vert z(t) \vert_\rV^2
\leq C_2 \big[ \vert z(t) \vert_\rH^2\vert z(t) \vert_\rV^2\big] \vert z(t) \vert_\rV^2 +2\vert f(t) \vert_\rH^2.
\end{equation}
Hence, since $\lambda_1\geq 0$, we have
\begin{equation}\label{eqn-NSE04''}
 \frac{d}{dt}\vert z(t) \vert_\rV^2
\leq C_2 \big[ \vert z(t) \vert_\rH^2\vert z(t) \vert_\rH^2\big] \vert z(t) \vert_\rV^2 +2\vert f(t) \vert_\rH^2,
\end{equation}
and so,  by the Gronwall Lemma, for any $-\infty < s\leq t\leq 0$ we get
\begin{eqnarray}\label{eqn-NSE05}
\vert z(t)\vert_\rV^2
&\leq&  \vert z(s)\vert_\rV^2 \exp\left(C_2 \int_s^t\dela{\big( -\lambda_1 +}\vert z(r)\vert_\rH^2\vert z(r)\vert_\rV^2\, dr\right)\\
&+&2\int_s^t \vert f(r)\vert_\rH^2 \exp\left(C_2\int_r^t\dela{\big( -\lambda_1 +}\vert z(\rho)\vert_\rH^2\vert z(\rho)\vert_\rV^2\, d\rho\right)\, dr.
\nonumber
\end{eqnarray}
Using the above inequality with $s=s_n$ from \eqref{eqn-NSE02'} and then taking the limit as $n\to\infty$,  we infer that
\begin{eqnarray}\label{eqn-NSE06}
\vert z(t)\vert_\rV^2
&\leq&   2\int_{-\infty}^t \vert f(r)\vert_\rH^2 \exp\left(C_2\int_r^t\dela{\big( -\lambda_1 +}\vert z(\rho)\vert_\rH^2\vert z(\rho)\vert_\rV^2\, d\rho\right)\, dr,\ \ \  t\leq 0.
\end{eqnarray}
Of course, for the above to be correct we need to show that the sequence
\[
\Big\{\int_{s_n}^t\dela{\big( -\lambda_1 +}\vert z(r)\vert_\rH^2\vert z(r)\vert_\rV^2\, dr \Big\}_{n\geq 1}
\] is bounded from above. But  in view of estimates \eqref{eqn-NSE06'} and \eqref{eqn-NSE07'} we have
\begin{equation}\label{eqn-A15}
 \int_{-\infty}^0 \vert z(\rho)\vert_\rH^2\vert z(\rho)\vert_\rV^2\, d\rho\leq \frac 1{\lambda_1^2} \vert f\vert^4_{L^2(-\infty,0,\rH)} <\infty.
 \end{equation}
Therefore, since
\[\int_r^t \vert z(\rho)\vert_\rH^2\vert z(\rho)\vert_\rV^2\, d\rho\leq  \int_{-\infty}^0  \vert z(\rho)\vert_\rH^2\vert z(\rho)\vert_\rV^2\, d\rho,\ \ \ \ -\infty< r\leq t\leq 0,\]
 we can conclude that
\begin{equation}\label{eqn-A16}
\sup_{t\leq 0} \vert z(t)\vert_\rV^2 \leq 2 \exp\left(\frac {C_2}{\lambda_1^2} \vert f\vert^4\right)\sup_{t\leq 0} \int_{-\infty}^t \dela{\e^{ -\lambda_1(t-r)}} \vert f(r)\vert_\rH^2 \,dr  \leq  2 \exp\left(\frac {C_2}{\lambda_1^2} \vert f\vert^4\right) \vert f\vert^2
\end{equation}
(here, for the sake of brevity, we denote $\vert f\vert_{L^2(-\infty,0;\rH)}=\vert f\vert$).

Moreover, as
\[\int_{-\infty}^0 \vert f(r)\vert_\rH^2 \exp\left(C_2\int_r^t\dela{\big( -\lambda_1 +}\vert z(\rho)\vert_\rH^2\vert z(\rho)\vert_\rV^2\, d\rho\right)\, dr<\infty,\]
we have
\[\lim_{t\to-\infty} \int_{-\infty}^t \vert f(r)\vert_\rH^2 \exp\left(C_2\int_r^t\dela{\big( -\lambda_1 +}\vert z(\rho)\vert_\rH^2\vert z(\rho)\vert_\rV^2\, d\rho\right)\, dr=0,\]
so that from  \eqref{eqn-NSE06} we conclude that \eqref{eqn-A04} holds.

Now,
to prove that $z \in\,W^{1,2}(-\infty,0)$, we observe that from \eqref{eqn-NSE03} we  also have
\[
\vert z(0)\vert_\rV^2+\int_{-\infty}^0 \vert A z(t)\vert_\rH^2\, dt
\leq C_2\int_{-\infty}^0 \big[ \vert z(t)\vert_\rH^2\vert z(t)\vert_\rV^2\big] \vert z(t)\vert_\rV^2\,dt  +2\int_{-\infty}^0\vert f(t)\vert_\rH^2 \,dt,
\]
where we have  used \eqref{eqn-A04}. Since by \eqref{eqn-NSE06'} and \eqref{eqn-A16},
$$\sup_{t\in (-\infty,0]} \vert z(t)\vert_\rH^2\vert z(t)\vert_\rV^2\leq \frac{2}{\lambda_1} \exp\left(\frac {C_2}{\lambda_1^2} \vert f\vert^4\right) \vert f\vert^4 <\infty,$$
 we infer that
\begin{equation}\label{eqn-NSE08}
\vert z(0)\vert_\rV^2+\int_{-\infty}^0 \vert A z(t)\vert_\rH^2\, dt
\leq \frac{2\,C_2}{\lambda_1} \exp\left(\frac {C_2}{\lambda_1^2} \vert f\vert^4\right) \vert f\vert^4\int_{-\infty}^0 \vert z(t)\vert_\rV^2\,dt  +2\int_{-\infty}^0\vert f(t)\vert_\rH^2\, dt.
\end{equation}
Hence, in view of \eqref{eqn-NSE07'}, we infer that
\begin{equation}\label{eqn-NSE09}
\vert z(0)\vert_\rV^2+\int_{-\infty}^0 \vert A z(t)\vert_\rH^2\, dt
\leq \frac{2\,C_2}{\lambda_1^2} \exp\left(\frac {C_2}{\lambda_1^2} \vert f\vert^4\right) \vert f\vert^6+2 \vert f\vert^2,
\end{equation}
and this  concludes the proof of the first part of \eqref{eqn-A10}.

In order to prove the second part of \eqref{eqn-A10}, it is enough to show that
\[\int_{-\infty}^0 \vert B(z(t),z(t))\vert_\rH^2\, dt<\infty.\]
 Indeed, by the Minkowski inequality  we have
\begin{equation}\label{eqn-A20}
\vert z^\prime\vert_{L^2(-\infty,0;\rH)} \leq \vert Az\vert_{L^2(-\infty,0;\rH)} +\vert B(z,z) \vert_{L^2(-\infty,0;\rH)}+\vert f\vert_{L^2(-\infty,0;\rH)}.
\end{equation}
According to inequalities \eqref{eqn-NSE06'}, \eqref{eqn-NSE06} and \eqref{eqn-NSE09} and to  inequality \eqref{ineq-B01}, we have
\begin{eqnarray}\label{eqn-A.21}
\int_{-\infty}^0 \vert B(z(t),z(t))\vert_\rH^2\, dt &\leq&  C  \int_{-\infty}^0 \vert z(t) \vert_\rH \vert z(t) \vert_\rV^2 \vert A z(t) \vert_\rH \, dt\\
&\leq & C \sup_{t\leq 0}  \vert z(t) \vert_\rH\vert z(t) \vert_\rV  \left(\int_{-\infty}^0\vert z(t)  \vert_\rV^2 \, dt\right)^\frac12  \left(\int_{-\infty}^0  \vert A z(t) \vert_\rH^2 \, dt \right)^\frac12
\nonumber
\\ &&\hspace{-2truecm}\lefteqn{\leq \frac{2}{\lambda_1} \exp\left(\frac {C_2}{\lambda_1^2} \vert f\vert^4\right) \vert f\vert^4
\frac 1{\sqrt{\lambda_1}}|f|\left( \frac{2\,C_2}{\lambda_1^2} \exp\left(\frac {C_2}{\lambda_1^2} \vert f\vert^4\right) \vert f\vert^6+2 \vert f\vert^2\right)^{\frac 12}<\infty.}
\nonumber
\end{eqnarray}
The final statement follows from inequalities \eqref{eqn-NSE09}, \eqref{eqn-A20} and \eqref{eqn-A.21}.
\end{proof}

\begin{rem}
{\em \begin{enumerate}
\item  Our proof of Proposition \ref{prop-infty}  has been inspired by \cite{Brz+Li_2006}.
\item
Roughly speaking, the above result says that the following two equalities hold
 \begin{equation*}
 \{z\in \mathcal{X}: S_{-\infty}(z) <\infty\} =\mathcal{X} \cap W^{1,2}(-\infty,0)
 \end{equation*}
 and
\begin{equation*}
 \{z\in \mathcal{X}: u(0)=\phi \mbox{ and } S_{-\infty}(z) <\infty\} =\mathcal{X}_\phi  \cap W^{1,2}(-\infty,0), \;\; \phi\in \rV.
 \end{equation*}
\end{enumerate}}
\end{rem}

\medskip

 Next Proposition generalizes Proposition \ref{prop-infty} to $S^\delta_{-\infty}$.

\begin{proposition} \label{prop-inftybis} \del{Let Assumption \ref{ass-Q} be satisfied, with $\beta \in\,(0,1/2)$.}
Assume that $\alpha \in\,(0,1/2)$ and
$z\in \mathcal{X}$  is such that $S^\delta_{-\infty}(z)<\infty$. Let  $f \in L^2((-\infty,0];D(\rA^{\frac\a 2})$
be  defined as
\begin{equation}\label{eqn-NSE01`}
z^\prime(t)+\rA z(t)+\rB(z(t),z(t))=\del{Q_\delta} f(t), \ \ \ \ t\leq 0.
\end{equation}
Then  $z(0)\in D(\rA^{\frac{\a+1}2})$,
\begin{equation}
\label{eqn-A04'}
\lim_{t\to-\infty} \vert z(t)\vert_{D(\rA^{\frac{\a+1}2})}=0,
\end{equation}
and $z\in W^{1,2}\big(-\infty,0;D(\rA^{1+\frac{\a}2}),D(\rA^{\frac{\a}2})\big)$, i.e.
\begin{equation}
\label{eqn-A11}
\int_{-\infty}^0 \vert A^{\frac{\a}2+1} z(t)\vert_\rH^2\, dt  + \int_{-\infty}^0 \vert  A^{\frac{\a+1}2}  z^\prime(t)\vert_\rH^2\, dt<\infty.
\end{equation}
Moreover, there exists a continuous and strictly increasing function $ \varphi: [0,\infty) \to [0,\infty)$ such that
$ \varphi(0)=0$ and if
$z \in \mathcal{X}$ is a solution to the problem
\[
z^\prime(t)+\rA z(t)+\rB(z(t),z(t))=f(t), \ \ \  t\leq 0,
\]
with  $f$ being an element of $L^2(-\infty,0;D(\rA^{\frac\a 2}))$, then  $z\in W^{1,2}\big(-\infty,0;D(\rA^{\frac{\a}2+1}),D(\rA^{\frac{\a}2})\big) $, $z(0)\in D(\rA^{\frac{\a+1}2})$ and
\[ \vert z(0) \vert_{D(\rA^{\frac{\a+1}2})}^2+ \vert z\vert^2_{W^{1,2}(-\infty,0;D(\rA^{\frac{\a}2+1}),D(\rA^{\frac{\a}2}))} \leq \varphi( \vert f\vert^2_{L^2(-\infty,0)}).\]
\end{proposition}
\begin{proof} Following the methods from the proof of Proposition \ref{prop-infty} it is sufficient to prove the first part of  Proposition
\ref{prop-inftybis}.

Let us fix $\a \in\,(0,1/2)$, $\phi \in\,D(\rA^{\frac{1+\a}2})$ and $z\in \mathcal{X}_\phi$  such that $S^\delta_{-\infty}(z)<\infty$. Let us define
  $f \in L^2(-\infty,0;D(\rA^{\frac\a 2})$ by \eqref{eqn-NSE01`}. Since the assumptions in the present proposition
 are stronger than the assumptions of Proposition
\ref{prop-infty}, we can freely use the results from the proof of the latter.

So, firstly, let us notice that by inequality  \eqref{eqn-NSE08}
 we can find a decreasing sequence $\{s_n\}$ such that $s_n\downarrow -\infty$ and
\begin{equation}\label{eqn-NSE02''}
\lim_{n\to\infty} \vert \rA^{\frac{1+\a}2} z(s_n)\vert_\rH=0.
\end{equation}
Arguing as in the proof of Proposition \ref{prop-NSE-fractional},
if we calculate the derivative of $\vert \rA^{\frac{\a+1}2} u(t)\vert_{\rH}^2 $ and use inequality
\eqref{ineqn-B-fractional}, with $s=2$, to get the following generalisation of \eqref{eqn-NSE03}

\begin{eqnarray}\label{eqn-NSE03c}
\frac12 \frac{d}{dt}\vert \rA^{\frac{1+\a}2} z(t) \vert_\rH^2+\vert \rA^{\frac{\a}2+1} z(t) \vert_\rH^2&=&-(B(z(t),z(t)),\rA^{{\a}+1}z(t))_\rH+(f(t),\rA^{\a+1} z(t))_\rH\\
&\leq& \frac12 \vert \rA^{\frac{\a}2+1} z(t) \vert_\rH^2+C\vert \rA z(t) \vert_\rH^2\vert \rA^{\frac{\a}2+1} z(t) \vert^2 + C\, \vert \rA^{\frac{\a}2} f(t) \vert_\rH^2.
\nonumber
\end{eqnarray}
Let us note that contrary to \eqref{eqn-NSE03}, the highest power of $z$ on the RHS of \eqref{eqn-NSE03c} is $4$.

Hence, we infer that
\begin{equation}\label{eqn-NSE04'''}
 \frac{d}{dt}\vert \rA^{\frac{1+\a}2} z(t) \vert_\rH^2 +\vert \rA^{\frac{\a}2+1} z(t) \vert_\rH^2
\leq C\vert \rA z(t) \vert_\rH^2 \vert \rA^{\frac{1+\a}2} z(t) \vert_\rH^2  +2\vert \rA^{\frac{\a}2} f(t) \vert_\rH^2.
\end{equation}

Therefore,   by the Gronwall Lemma, for any $-\infty < s\leq t\leq 0$ we get
\begin{eqnarray}\label{eqn-NSE05''}
\vert \rA^{\frac{1+\a}2}z(t)\vert_\rH^2
&\leq&  \vert \rA^{\frac{1+\a}2}z(s)\vert_\rH \exp\left(C \int_s^t \vert\rA z(r)\vert_\rH^2\,dr\right)\\
&+&2\int_s^t \vert \rA^{\frac{\a}2} f(r)\vert_\rH^2 \exp\left(C\int_r^t \vert \rA z(\rho)\vert_\rH^2\, d\rho\right)\, dr.
\nonumber
\end{eqnarray}
Using the above with $s=s_n$ from \eqref{eqn-NSE02''} and then taking the limit as $n\to\infty$,  we infer that
\begin{equation}\label{eqn-NSE06''}
\vert \rA^{\frac{1+\a}2}z(t)\vert_\rH^2
\leq   2\int_{-\infty}^t \vert \rA^{\frac{\a}2} f(r)\vert_\rH^2 \exp\left(C\int_r^t\vert \rA z(\rho)\vert_\rH^2\, d\rho\right)\, dr,\ \ \  t\leq 0.
\end{equation}
As in the proof of the previous Proposition \ref{prop-infty} the above is true because now by inequality \eqref{eqn-NSE09}   the sequence
\[
\Big\{\int_{s_n}^t\dela{\big( -\lambda_1 +}\vert \rA z(r)\vert_\rH^2\, dr \Big\}_{n\geq 1}
\] is bounded from above by
$ \frac{2\,C_2}{\lambda_1^2} \exp\left(\frac {C_2}{\lambda_1^2} \vert f\vert^4\right) \vert f\vert^6+2 \vert f\vert^2$, where
$\vert f\vert$ denotes $\vert f\vert_{L^2(-\infty,0;\rH)}$. Therefore,
 we can conclude that
\begin{equation}\label{eqn-A16'}
\sup_{t\leq 0} \vert \rA^{\frac{1+\a}2} z(t)\vert \leq
2\int_{-\infty}^0 \vert \rA^{\frac{\a}2} f(r)\vert_\rH^2\,dr \exp\left(C \frac{2\,C_2}{\lambda_1^2} \exp\left(\frac {C_2}{\lambda_1^2} \vert f\vert^4\right) \vert f\vert^6+2 \vert f\vert^2 \right).
\end{equation}

Moreover, as
\[\int_{-\infty}^0 \vert \rA^{\frac{\a}2} f(r)\vert_\rH^2 \exp\left(C\int_r^t \vert \rA z(\rho)\vert_\rH^2\, d\rho\right)\, dr<\infty,\]
we have that
\[\lim_{t\to-\infty} \int_{-\infty}^t \vert \rA^{\frac{\a}2} f(r)\vert_\rH^2 \exp\left(C\int_r^t\vert \rA z(\rho)\vert_\rH^2\, d\rho\right)\, dr=0.\]
Hence  \eqref{eqn-A04'} follows  from  \eqref{eqn-NSE06''}.

Now, to prove the second part of Proposition \ref{prop-inftybis}, i.e. the first inequality in \eqref{eqn-A04'},
 we observe that from \eqref{eqn-NSE04'''} we  also have
\[
\vert \rA^{\frac{\a+1}2} z(0)\vert_\rH^2+\int_{-\infty}^0 \vert \rA^{\frac{\a+2}2} z(t)\vert_\rH^2\, dt
\leq C\int_{-\infty}^0 \big[ \vert \rA z(t)\vert_\rH^2\big] \vert \rA^{\frac{\a+1}2} z(t)\vert_\rH^2\,dt
+2\int_{-\infty}^0\vert \rA^{\frac{\a}2} f(t)\vert_\rH^2 \,dt.
\]
Taking into account inequalities \eqref{eqn-A16'} and \eqref{eqn-NSE09} we infer that
\begin{eqnarray}\label{eqn-NSE08'}
&&\vert \rA^{\frac{\a+1}2} z(0)\vert_\rH^2+\int_{-\infty}^0 \vert \rA^{\frac{\a+2}2} z(t)\vert_\rH^2\, dt
\leq C \Big( \frac{2\,C_2}{\lambda_1^2} \exp\left(\frac {C_2}{\lambda_1^2} \vert f\vert^4\right) \vert f\vert^6+2 \vert f\vert^2\Big)
\\
&&\Big(\int_{-\infty}^0 \vert \rA^{\frac{\a}2} f(r)\vert_\rH^2\,dr \exp\left(C \frac{2\,C_2}{\lambda_1^2} \exp\left(\frac {C_2}{\lambda_1^2} \vert f\vert^4\right) \vert f\vert^6+2 \vert f\vert^2 \right)\Big)+2\int_{-\infty}^0\vert \rA^{\frac{\a}2} f(t)\vert_\rH^2 \,dt,
\nonumber
\end{eqnarray}
and this  concludes the proof of the first part of inequality \eqref{eqn-A04'}.

As in the proof of the previous Proposition, in order to prove the third  part of Proposition \ref{prop-inftybis}, i.e. the second inequality in \eqref{eqn-A04'},
it is enough to show that
\[ \int_{-\infty}^0 \vert \rA^{\frac{\a}2}   B(z(t),z(t))\vert_\rH^2\, dt < \infty.\]

According to inequalities \eqref{ineqn-B-fractional} (with $s=2$)), \eqref{eqn-NSE09}  and  \eqref{eqn-A16'}
we have
\begin{eqnarray}
&&\hspace{-10truecm}\lefteqn{
\int_{-\infty}^0 \vert \rA^{\frac{\a}2}   B(z(t),z(t))\vert_\rH^2\, dt \leq  C  \int_{-\infty}^0 \vert \rA z(t) \vert^2_\rH
\vert \rA^{\frac{\a+1}2} z(t) \vert_\rH^2  \, dt}\\
&&\hspace{-9truecm}\lefteqn{\leq  C \sup_{t\leq 0}  \vert \rA^{\frac{\a+1}2} z(t) \vert^2_\rH
\int_{-\infty}^0  \vert \rA z(t) \vert_\rH^2 \, dt \leq \Big(\frac{4C\,C_2}{\lambda_1^2} \exp\left(\frac {C_2}{\lambda_1^2} \vert f\vert^4\right) \vert f\vert^6+2 \vert f\vert^2\Big)}
\nonumber
\\
&&\hspace{-5truecm}\lefteqn{\int_{-\infty}^0 \vert \rA^{\frac{\a}2} f(r)\vert_\rH^2 \exp\left(C \frac{2\,C_2}{\lambda_1^2} \exp\left(\frac {C_2}{\lambda_1^2} \vert f\vert^4\right) \vert f\vert^6+2 \vert f\vert^2 \right).
}
\nonumber
\end{eqnarray}
The proof is now complete.
\end{proof}

\begin{proposition}\label{prop--infty2}
Assume that $T\in \mathbb{R}\cup\{+\infty\}$. If $z$ belongs to $W^{1,2}(-\infty,T)$ then
\[\lim_{t\to -\infty} \vert z(t)\vert_\rV=0.\]
\end{proposition}
\begin{proof}
In view of \cite[Theorem 2.2, p. 13]{Lions+Magenes_1972} it is enough to consider the case $T=+\infty$.
So let us take $z\in W^{1,2}(-\infty,\infty)$
Then, since $\rV=[D(\rA),\rH]_{1/2}$,  according to \cite[Theorem 3.1, p. 19]{Lions+Magenes_1972}, $z:\mathbb{R}\to \rV$
is a bounded and continuous function. Moreover, by \cite[(2.27), p. 16]{Lions+Magenes_1972}
\[\lim_{t\to\pm\infty}|u(t)|_{|rH}=0.\]
Hence the result  follows by applying \eqref{prop-infty}.
\end{proof}

\section{Proofs of Lemmas in Section \ref{sec-exit}}
\label{app:B}

\begin{proof}[Proof of Lemma \ref{lem1}]
In view of Proposition \ref{propbrc3}, for any $\kappa>0$ there exists $T(\beta,\kappa,\mu)>0$ such that
\[u^\phi(t;0) \in\,B_{\frac{\beta+1}2}(\kappa),\ \ \ \phi \in\,B_{0}(\mu),\ \ \ t>T(\beta,\kappa,\mu).\]
Thus, if we set $T_1=T(\beta,\kappa,\mu)+1$ and
\[z_1(t)=u^\phi(t;0),\ \ \ t \in\,[0,T_1],\]
we have that $z_1(0)=\phi$, $z_1(T_1) \in\,D(\rA^{\frac{1+\beta}2})$ and
\begin{equation}
\label{p1}
S^\delta_{T_1}(z_1)=0.
\end{equation}

Now, we define
\[z_2(t)=(T_1+1-t)e^{-(t-T_1)\rA}z_1(T_1),\ \ \ \ t \in\,[T_1,T_1+1].\]
We have $z_2(T_1)=z_1(T_1)$ and $z_1(T_1+1)=0$. Moreover,
\[{\mathcal H}(z_2)(t)=-e^{-(t-T_1)\rA}z_1(T_1)+B(z_2(t),z_2(t)),\]
so that, according to Assumption \ref{ass-Q}, we have
\[\begin{array}{l}
\displaystyle{S^\delta_{T_1,T_1+1}(z_2)\leq c\,\int_{T_1}^{T_1+1} \left|e^{-(t-T_1)\rA}z_1(T_1)\right|_{D(\rA^{\frac \beta 2})}^2\,dt+
c\,\int_{T_1}^{T_1+1} \left|B(z_2(t),z_2(t))\right|_{D(\rA^{\frac \beta 2})}^2\,dt.}
\end{array}\]
Now, thanks to \eqref{ineqn-B-fractional}, with $s=1+\beta$, we have
\[\left|B(z_2(t),z_2(t))\right|_{D(\rA^{\frac \beta 2})}\leq \left|e^{-(t-T_1)\rA}z_1(T_1)\right|_{D(\rA^{\frac {1+\beta} 2})}^2\leq c\,\kappa^2,\]
so that
\[S^\delta_{T_1,T_1+1}(z_2)\leq c\,\kappa^2+c\,\kappa^4.\]
Therefore, we fix $\bar{\kappa}>0$ small enough such that
\begin{equation}
\label{p2}
S^\delta_{T_1,T_1+1}(z_2)\leq c\,\bar{\kappa}^2+c\,\bar{\kappa}^4<\frac \eta 2,
\end{equation}
and in correspondence of $\bar{\kappa}$ we fix $T_1=T(\beta,\bar{\kappa},\mu)$.

As we have seen in Lemma \ref{point}, there exists $\phi_\delta \in\,\partial D\cap D(\rA^{\frac{1+\beta}2})$ such that
\[U_\delta(\phi_\delta)=\inf_{\phi \in\,\partial D}U_\delta(\phi).\]
According to Assumption \ref{D},  we can fix $\{\phi_{\delta,n}\}_{n \in\,\mathbb{N}}\subset  \overline{D}^c\cap D(\rA^{\frac{1+\beta}2})$, such that
\begin{equation}
\label{brc92}
\lim_{n\to\infty}|\phi_{\delta,n}-\phi_\delta|_{D(\rA^{\frac{1+\beta}2})}=0.
\end{equation}
Thus, since the mapping $U_\delta:D(\rA^{\frac{1+\beta}2})\to [0,+\infty)$ is continuous (see Proposition \ref{brc505}), we have
\[\lim_{n\to\infty}U_\delta(\phi_{\delta,n})=U_\delta(\phi_\delta).\]
This means that we can fix $\bar{n} \in\,\mathbb{N}$ such that
\[U_\delta(\phi_{\delta,\bar{n}})<U_\delta(\phi_\delta)+\frac \eta 4,\]
and $T_2=T_2(\eta)>0$ and $z_3 \in\,C([0,T_2];\rH)$ such that $z_3(0)=0$, $z_3(T_2)=\phi_{\delta,\bar{n}}$ and
\[S^\delta_{T_2}(z_3)<U_\delta(\phi_{\delta,\bar{n}})+\frac \eta 4<U_\delta(\phi_\delta)+\frac \eta 2.\]

Therefore, if we define $T:=T_1+T_2+1$ and
\[v^\phi(t)=\begin{cases}
z_1(t),  &  t \in\,[0,T_1],\\
z_2(t),  &  t \in\,[T_1,T_1+1],\\
z_3(t-(T_1+1)),  &  t \in\,[T_1+1,T],
\end{cases}\]
we get $v^\phi \in\,C([0,T;\rH)$, with $v^\phi(0)=\phi$ and $v^\phi(T)=\phi_{\delta,\bar{n}}$ and
\[S^\delta_{T}(v^\phi)\leq U_\delta(\phi_\delta)+\eta.\]
Moreover, $T$ only depends on $\mu$ and $\eta$.

\end{proof}

\medskip

\begin{proof}[Proof of Lemma \ref{lem2} and lemma \ref{lem3}]
The proofs of these two lemmas are analogous to the proofs of \cite[Lemmas 7.3, 7.4 and 7.5]{cerrai-freidlin} and is based on the validity of a large deviation principle for the $2$-D Navier-Stokes equation perturbed by additive noise, as proved in Theorem \ref{brc.teo.4.2},  which is uniform with respect to the initial condition $\phi$ in a bounded set of $\rH$. The arguments used in \cite{cerrai-freidlin} are an adaptation to an infinite dimensional setting of the methods used in \cite[Chapter 5]{dembo}.
\end{proof}

\medskip

\begin{proof}[Proof of Lemma \ref{lem3bis}]
Our proof follows a path  analogous to the one followed in the proof of Lemmas 9.5, 9.6 and 9.9  in \cite{cerrai-salins}. We proceed here in three steps.
\smallskip

{\em Step 1.} {We will show that} there exists a  strictly increasing continuous  function $ \varphi: [0,\infty) \to [0,\infty)$
 such that  for any $\phi \in\,\rH$ and $f \in\,L^2(0,T;\rH)$ and for any $T>0$,
\begin{equation}
\label{z1}
|u_\phi^f-u_0^f|_{C([0,T];H)}\leq \varphi\left(|f|_{L^2(0,T;\rH)}\right)|\phi|_{\rH},
\end{equation}
where {$u_\phi^f \in L^2(0,T;\rV)$, with $(u_\phi^f)^\prime \in L^2(0,T;\rV^\prime)$,  is the solution of  problem \eqref{eqn_NSE01}, i.e.}
\[u^\prime(t)+\rA u(t)+B(u(t),u(t))=f(t),\ \ \ \ u(0)=\phi.\]

\medskip

{\em Proof.} Let us fix $T>0$, $\phi \in\,\rH$ and  $f \in\,L^2(0,T;\rH)$, and  denote $v:=u_0^f-u_\phi^f$. Then
\[v^\prime(t)+\rA v(t)+\rB(v(t),u^0_f(t))+\rB(u^0_f(t),v(t))=0, \;\;\; v(0)=-\phi\]
{and hence, by \cite[Lemmata III.2.1 and III.3.2]{Temam_2001},}
\begin{eqnarray}
\nonumber
\frac 12 \frac{d}{dt}|v(t)|_{\rH}^2&+&| v(t)|_V^2=-\left<\rB(v(t),u^0_f(t)),v(t)\right>_{\rH}\leq  \sqrt{2}\,| u^0_f(t))|_V |v(t)|_{\rH} | v(t)|_V\\
\nonumber\vspace{.1mm}\\
&\leq& \frac 12 | v(t)|_V^2+\,| u^0_f(t))|_V^2|v(t)|^2_{\rH}.
\end{eqnarray}
By the  Gronwall Lemma, this implies
\begin{equation}
\label{z2}
|u_\phi^f(t)-u_0^f(t)|_{\rH}^2\leq |\phi|_{\rH}^2\exp\left(2\int_0^t| u^0_f(s)|_V ^2\,ds\right), \;\; t \in [0,T].
\end{equation}

This together with inequality \eqref{ineq-aux-00} implies that
\begin{equation}
\label{z3}
|u_\phi^f-u_0^f|_{C([0,T];\rH)}\leq |\phi|_{\rH}\varphi \big( \vert f\vert^2_{L^2(0,T;\rV^\prime)}\big).
\end{equation}

\medskip

{\em Step 2.}  {Assume that   $C\subset \rH$ is a  closed set and  a real number $\beta$ satisfies
\begin{equation}\label{ineq-beta}
\beta<\inf_{\phi \in\,C}U_\delta(\phi).
\end{equation}}
Then  there exists { a positive number} $\rho_0>0$ such that for every $T>0$ and every $u \in\,C([0,T];\rH)$, with $|u(0)|_{\rH}<\rho_0$ and $S^\delta_{0,T}(u)\leq \beta$, it holds
\[\text{dist}_{\rH}(u(t),C) > {\rho_0}, \mbox{ for every } t \in [0, T].\]

\medskip
{\em Proof.} Suppose our claim is not true.
 Then for every $n \in\,\mathbb{N}$ we can find  $\phi_n \in\, \rH$, $T_n>0$, $\hat{T}_n \in\,[0,T_n]$ and $f_n \in\,L^2(0,T_n;\rH)$ such that
for every $ n \in\,\mathbb{N}$
\begin{equation}
\label{z7}
|\phi_n|_{\rH} < \frac1n,\ \ \ |f_n|_{L^2(0,T_n;\rH)}^2<2\beta\end{equation}
and
\begin{equation}
\label{z6}
\mathrm{dist}(u_{\phi_n}^{f_n}(\hat{T}_n),C)\leq \frac1n.\end{equation}
Now, if we set
$\hat{\phi}_n:=u_{0}^{f_n}(\hat{T_n})$, then by \eqref{z1} we have
\[|\hat{\phi}_n-u_{\phi_n}^{f_n}(\hat{T_n})|_{\rH}\leq \varphi\left(|f_n|_{L^2(0,T_n;\rH)}\right)|\phi_n|_{\rH},\]
so that, thanks to \eqref{z7} and \eqref{z6}, we get
\begin{equation}
\label{z8}
\lim_{n\to \infty }\text{dist}(\hat{\phi}_n,C)\leq \lim_{n\to\infty}
\Big[ \varphi\left(\sqrt{2\beta}\right)|\phi_n|_{\rH}+\frac1n\Big]=0.\end{equation}
Moreover,
\[U_\delta(\hat{\phi}_n)\leq \frac 12 |f_n|_{L^2(0,T_n;\rH)}^2 \leq \beta,\]
and then, by  the compactness in $\rH$ of the level sets of the functional $U_\delta$, we infer that there is a subsequence $\{\hat{\phi}_{n_k}\}$ and an element $\hat{\phi} \in\,\rH$ such that
$\hat{\phi}_{n_k}\to \hat{\phi}$ in $\rH$ and $U_\delta(\hat{\phi}) {\leq }\beta$.  {On the other hand,  $C$ is a closed subset of $\rH$ so
by  \eqref{z8} we infer that  $\hat{\phi} \in\,C$. This  contradicts our assumption \eqref{ineq-beta}.}
\medskip

{\em Step 3.}  {Assume that  $N\subset \partial D$ is a  closed set. Then} \eqref{z9} holds.

\medskip

{\em Proof.} {Let us choose a real number $\beta$ such that condition \eqref{ineq-beta} holds. Let us also choose a positive number $\mu>0$ such that $B_{0}(3\mu)\subset D$.}
 For any $T>0$, we have
\begin{equation}
\label{z11}
\mathbb{P}\left(u_\phi^{\eps,\delta}(\sigma_\phi^{\eps,\delta,\mu}) \in\,N\right)\leq \mathbb{P}\left(\sigma_\phi^{\eps,\delta,\mu}>T\right)+\mathbb{P}\left(u_\phi^{\eps,\delta}(t) \in\,N,\ \text{for\ some}\ t\in [0,T]\right).\end{equation}
According to Step 2 and to the fact that the large deviation principle proved in Theorem \ref{brc.teo.4.2} is uniform with respect to initial conditions $\phi$ in bounded sets of $\rH$, for any  $\mu>0$ such that $B_{0}(3\mu)\subset D$ and any $\beta<\inf_{x \in\,C}U_\delta(x)$, we can find $\eps_1>0$ such that for every $\eps \in (0, \eps_1]$
\begin{eqnarray}
&&\hspace{-2truecm}\lefteqn{\sup_{\phi \in\,B_{0}(3\mu)}\mathbb{P}\left(u_\phi^{\eps,\delta}(t) \in\,N,\ \text{for\ some}\ t\in [0,T]\right)}\\
&\leq& \sup_{\phi \in\,B_{0}(3\mu)} \mathbb{P}\left(\text{dist}_{C([0,T];\rH)}(u_\phi^{\eps,\delta},K^\delta_T(\beta))>3\mu\right)\leq e^{-\frac{\beta-\gamma}\eps},
\end{eqnarray}
where $K^\delta_T(\beta)=\{u \in\,C([0,T];\rH)\,:\,S^\delta_{T}(u)\leq \beta\}$.

Moreover, in view of Lemma \ref{lem3}, there exist $T>0$ and $\eps_2\leq \eps_1$ such that
\[\sup_{x \in\,D}\mathbb{P}\left(\sigma_\phi^{\eps,\delta,\mu}>T\right)\leq e^{-\frac{\beta}\eps},\ \ \ \ \eps\leq \eps_2.\]

Then, thanks to \eqref{z11}, we can conclude the proof of Step 3, due to the arbitrariness of $\gamma>0$ and {condition \eqref{ineq-beta}}.\end{proof}
\medskip

\begin{proof}[Proof of Lemma \ref{lem4}] Let us fix $\phi\in \rH$ and $\delta \in (0,1]$. For any $\eps>0$, let us now denote   by $z_{\eps,\delta}$ the Ornstein-Uhlenbeck process defined by equation \eqref{eqn-OUP-eps} and by $u_\phi^{\eps,\delta}$ the solution to the stochastic Navier-Stokes equation \eqref{eqn-SNSE-eps}. Thanks to \cite[Theorem 1.2 ]{Brz+Peszat_2000} (with $\xi(t)$ being the $\gamma$-radonifying natural embedding operator from $Q_\delta(\rH)$ to $\rH \cap L^4(\mathcal{O})$) we infer that there exists a constant $C>0$ such that for any $R>0$ and $\eps>0$
\begin{equation}
\label{eqn-proof-01}
\eps \log \mathbb{P} \big( \vert z^{\eps,\delta}\vert_{C([0,T];L^4(\mathcal{O}))} \geq R\big) \leq -\frac{R^2}{CT}.
\end{equation}
Let us now fix $\mu>0$ and $\lambda>0$. By the above inequality there exists $T_0 >0$ such that
\begin{equation}
\label{eqn-proof-02}
\eps \log \mathbb{P} \big( \vert z^{\eps,\delta}\vert_{C([0,T_0];L^4(\mathcal{O}))} \geq \frac{\mu}{3} \big) \leq -\frac{\lambda}{2}, \; \;\eps>0.
\end{equation}
	For a given $z\in\, C([0,T_0];L^4(\mathcal{O})$ and $\phi \in\,\rH$ let us denote by $v_\phi^z$ the unique solution to the problem
\begin{equation}
\label{eqn-proof-03}
(v_\phi^z)^\prime(t)+\rA v_\phi^z(t)+B(v_\phi^z(t)+z(t),v_\phi^z(t)+z(t))=0,\;\; t\in [0,T_0], \ \ \ v_\phi^z(0)=\phi.
\end{equation}
Note that $v_\phi^0$ is the unique solution to the deterministic NSE satisfying  the initial condition $v_\phi^0(0)=\phi$. Hence \begin{equation}
\label{eqn-proof-07}
 \sup_{\phi \in\,B_{0}(\mu)}\vert v_\phi^0-\phi\vert_{C([0,T];\rH)} \leq 2\mu.
\end{equation}

By  \cite[Theorem 4.6]{Brz+Li_2006} we infer that \del{for every $\eta>0$} there exists $\beta>0$ such that
\begin{equation}
\label{eqn-proof-04}
 \vert z \vert_{C([0,T_0];L^4(\mathcal{O}))} < \beta\Longrightarrow \sup_{\phi \in\,B_{0}(\mu)}\vert v_\phi^z -v_\phi^0\vert_{C([0,T_0];\rH)} < \del{\eta}\frac{\mu}{3}.
\end{equation}
By a simple uniqueness argument, the above holds with the same constant for all $T\in (0,T_0]$, i.e.
\del{for every $\eta>0$} there exists $\beta>0$ such that for every $T\in (0,T_0]$,
\begin{equation}
\label{eqn-proof-05}
 \begin{aligned}
&\vert z \vert_{C([0,T];L^4(\mathcal{O}))} < \beta\Longrightarrow  \sup_{\phi \in\,B_{0}(\mu)}\vert v_\phi^z -v_\phi^0\vert_{C([0,T];\rH)} < \frac{\mu}{3},\\
&\sup_{\phi \in\,B_{0}(\mu)}\vert v_\phi^0-\phi\vert_{C([0,T];\rH)} < 2\mu.
\end{aligned}
\end{equation}

Since, see  \cite{Brz+Li_2006}, $ u_\phi^{\eps,\delta}-\phi=z^{\eps,\delta}+v_\phi^{z^{\eps,\delta}}-v_\phi^0+v_\phi^0-\phi$,  we infer that for every $\eps>0$,
\begin{equation}
\label{eqn-proof-06}
\begin{aligned}
&\eps \log\sup_{\phi \in\,B_{0}(\mu)} \mathbb{P} \big( \vert u_\phi^{\eps,\delta}-\phi\vert_{C([0,T];\rH)} \geq 3\mu \big) \leq
\eps \log \mathbb{P} \big( \vert z^{\eps,\delta}\vert_{C([0,T];\rH)} \geq \frac{\mu}{3} \big)\\
&+\eps \log \sup_{\phi \in\, B_{0}(\mu)}\mathbb{P} \big( \vert v_\phi^{z^{\eps,\delta}}-v_\phi^0\vert_{C([0,T];\rH)} \geq \frac{\mu}{3} \big)
+\eps \log \sup_{\phi \in\, B_{0}(\mu)}\mathbb{P} \big( \vert v_\phi^0-\phi\vert_{C([0,T];\rH)} \geq \frac{7\mu}{3} \big).
\end{aligned}
\end{equation}

Let us note that by the second part of \eqref{eqn-proof-05},  the last term on the RHS of inequality \eqref{eqn-proof-06} is equal to $0$.

In order to estimate the first term on the RHS of inequality \eqref{eqn-proof-06} let us choose $T\leq T_0$ such that
\[ \frac{\beta^2}{CT} \geq  \frac{\lambda}{2}\]
and then   apply inequality \eqref{eqn-proof-01} with $R=\beta$. We get that $\eps \log \mathbb{P} \big( \vert z^{\eps,\delta}\vert_{C([0,T];\rH)} \geq \frac{\mu}{3} \big) \leq - \frac{\lambda}{2}$.

In order to estimate the second term on the RHS of inequality \eqref{eqn-proof-06}  we use inequalities \eqref{eqn-proof-05} and \eqref{eqn-proof-02}. Thus  we deduce that
\begin{equation}
\label{eqn-proof-08}
\begin{aligned}
&\eps \log \sup_{\phi \in\, B_{0}(\mu)}\mathbb{P} \big( \vert u_\phi^{\eps,\delta}-\phi\vert_{C([0,T];\rH)} \geq \mu \big) \leq
\eps \log \mathbb{P} \big( \vert z^{\eps,\delta}\vert_{C([0,T];\rH)} \geq \frac{\mu}{3} \big)\\
&+\eps \log \mathbb{P} \big( \vert z^{\eps,\delta} \vert_{C([0,T_0];L^4(\mathcal{O}))} \geq \beta \big)
\leq - \frac{\lambda}{2}- \frac{\lambda}{2}=-\lambda.
\end{aligned}
\end{equation}
This completes the proof.
\end{proof}

\end{document}